\documentclass[british,12pt]{article}
\usepackage[T1]{fontenc}
\usepackage[utf8]{inputenc}
\usepackage{geometry}
\usepackage{color}
\usepackage{babel}
\usepackage{array}
\usepackage{float}
\usepackage{url}
\usepackage{amsmath}
\usepackage{amsthm}
\usepackage{amssymb}
\usepackage{graphicx}
\usepackage[unicode=true,pdfusetitle,
 bookmarks=true,bookmarksnumbered=false,bookmarksopen=false,
 breaklinks=false,pdfborder={0 0 1},backref=false,colorlinks=true]
 {hyperref}

\makeatletter

\providecommand{\tabularnewline}{\\}

\newcommand{\lyxaddress}[1]{
	\par {\raggedright #1
	\vspace{1.4em}
	\noindent\par}
}
\theoremstyle{plain}
\newtheorem{thm}{\protect\theoremname}
\theoremstyle{definition}
\newtheorem{defn}[thm]{\protect\definitionname}
\theoremstyle{remark}
\newtheorem{rem}[thm]{\protect\remarkname}
\theoremstyle{plain}
\newtheorem{cor}[thm]{\protect\corollaryname}
\ifx\proof\undefined
\newenvironment{proof}[1][\protect\proofname]{\par
	\normalfont\topsep6\p@\@plus6\p@\relax
	\trivlist
	\itemindent\parindent
	\item[\hskip\labelsep\scshape #1]\ignorespaces
}{%
	\endtrivlist\@endpefalse
}
\providecommand{\proofname}{Proof}
\fi
\theoremstyle{plain}
\newtheorem{prop}[thm]{\protect\propositionname}
\theoremstyle{definition}
\newtheorem{example}[thm]{\protect\examplename}
\newenvironment{lyxlist}[1]
	{\begin{list}{}
		{\settowidth{\labelwidth}{#1}
		 \setlength{\leftmargin}{\labelwidth}
		 \addtolength{\leftmargin}{\labelsep}
		 }}
	{\end{list}}

\usepackage{tikz-cd}

\makeatother

\providecommand{\corollaryname}{Corollary}
\providecommand{\definitionname}{Definition}
\providecommand{\examplename}{Example}
\providecommand{\propositionname}{Proposition}
\providecommand{\remarkname}{Remark}
\providecommand{\theoremname}{Theorem}

\begin{document}
\title{Data-driven reduced order models using invariant foliations, manifolds
and autoencoders}
\author{Robert Szalai}
\date{latest revision: 21 April 2023}
\maketitle

\lyxaddress{University of Bristol, Department of Engineering Mathematics, email
\url{r.szalai@bristol.ac.uk}}
\begin{abstract}
This paper explores how to identify a reduced order model (ROM) from
a physical system. A ROM captures an invariant subset of the observed
dynamics. We find that there are four ways a physical system can be
related to a mathematical model: invariant foliations, invariant manifolds,
autoencoders and equation-free models. Identification of invariant
manifolds and equation-free models require closed-loop manipulation
of the system. Invariant foliations and autoencoders can also use
off-line data. Only invariant foliations and invariant manifolds can
identify ROMs, the rest identify complete models. Therefore, the common
case of identifying a ROM from existing data can only be achieved
using invariant foliations.

Finding an invariant foliation requires approximating high-dimensional
functions. For function approximation, we use polynomials with compressed
tensor coefficients, whose complexity increases linearly with increasing
dimensions. An invariant manifold can also be found as the fixed leaf
of a foliation. This only requires us to resolve the foliation in
a small neighbourhood of the invariant manifold, which greatly simplifies
the process. Combining an invariant foliation with the corresponding
invariant manifold provides an accurate ROM. We analyse the ROM in
case of a focus type equilibrium, typical in mechanical systems. The
nonlinear coordinate system defined by the invariant foliation or
the invariant manifold distorts instantaneous frequencies and damping
ratios, which we correct. Through examples we illustrate the calculation
of invariant foliations and manifolds, and at the same time show that
Koopman eigenfunctions and autoencoders fail to capture accurate ROMs
under the same conditions.
\end{abstract}

\section{Introduction}

There is a great interest in the scientific community to identify
explainable and/or parsimonious mathematical models from data. In
this paper we classify these methods and identify one concept that
is best suited to accurately calculate reduced order models (ROM)
from off-line data. A ROM must track some selected features of the
data over time and predict them into the future. We call this property
of the ROM \emph{invariance}. A ROM may also be \emph{unique}, which
means that barring a (nonlinear) coordinate transformation, the mathematical
expression of the ROM is independent of who and when obtained the
data as long as the sample size is sufficiently large and the distribution
of the data satisfies some minimum requirements.

\begin{figure}
\begin{centering}
\includegraphics[width=0.49\linewidth]{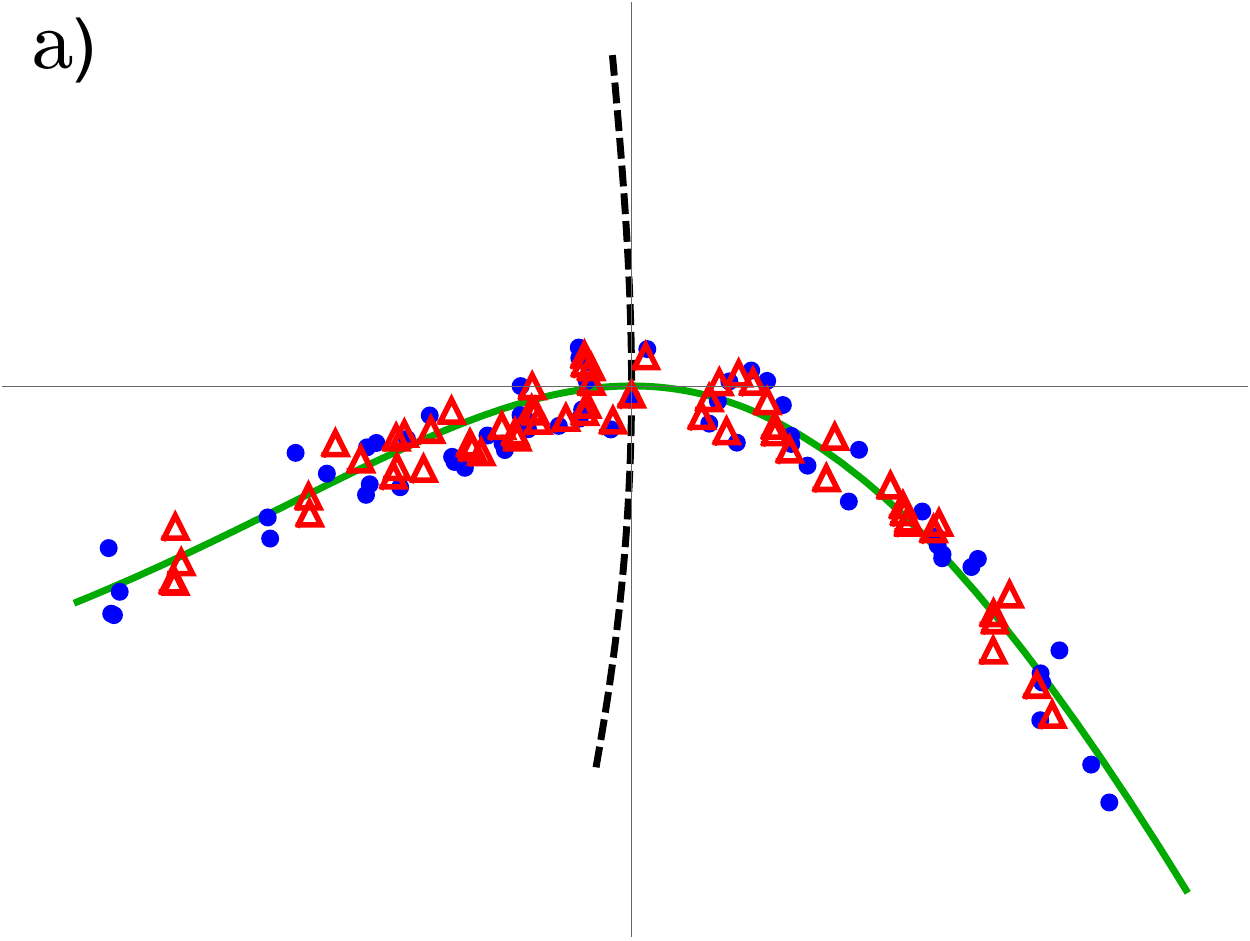}\includegraphics[width=0.49\linewidth]{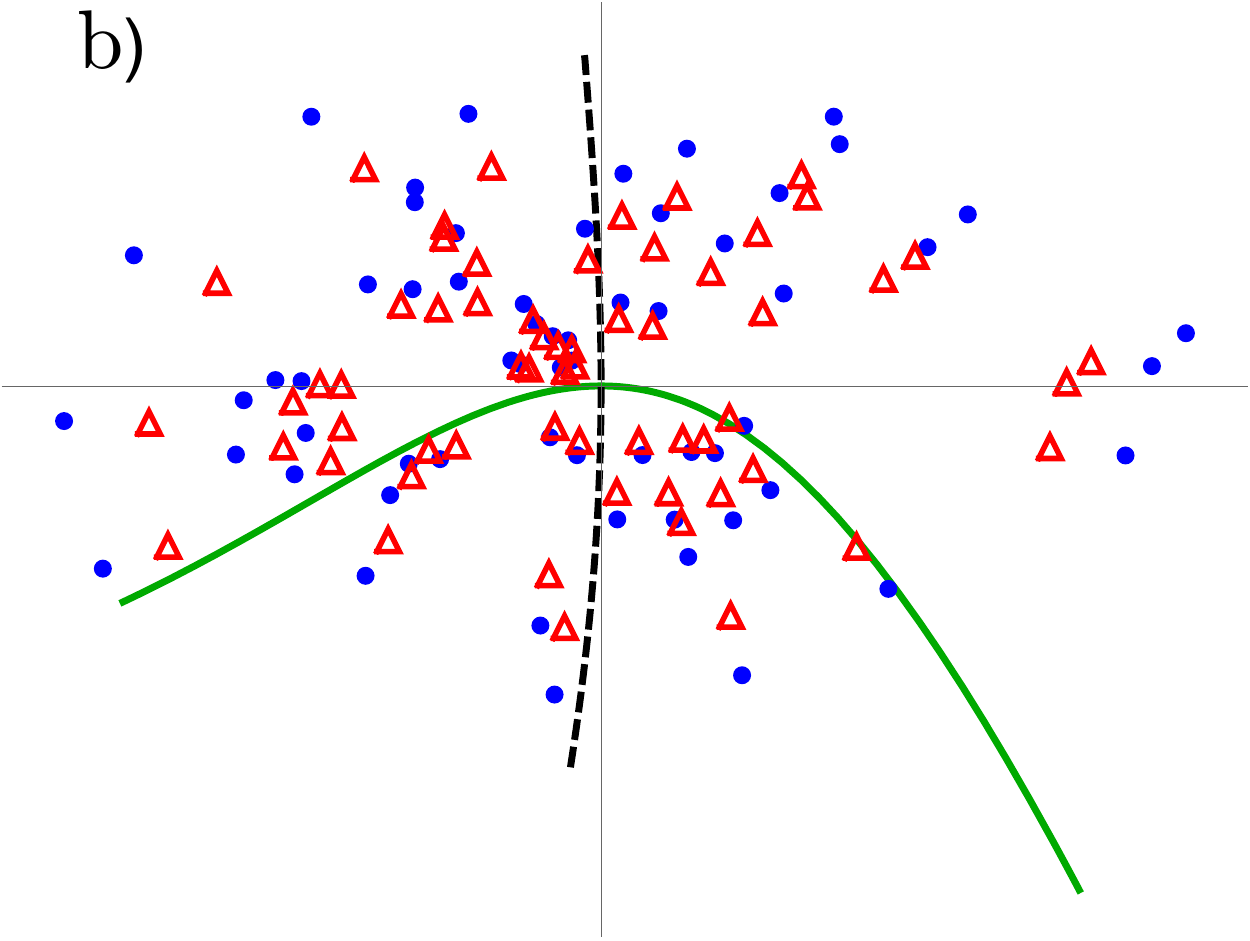}
\par\end{centering}
\caption{\label{fig:data-structures}Two cases of data distribution. The blue
dots (initial conditions) are mapped into the red triangles by the
nonlinear map $\boldsymbol{F}$ given by equation (\ref{eq:caricature-mod}).
a) data points are distributed in the neighbourhood of the solid green
curve, which is identified as an approximate invariant manifold. b)
Initial conditions are well-distributed in the neighbourhood of the
steady state and there is no obvious manifold structure. For completeness,
the second invariant manifold is denoted by the dashed line.}
\end{figure}

Not all methods that identify low-dimensional models produce ROMs.
In some cases the data lie on a low-dimensional manifold embedded
in a high-dimensional, typically Euclidean, space as in figure \ref{fig:data-structures}(a).
In this case the task is to parametrise the low-dimensional manifold
and fit a model to the data in the coordinates of the parametrisation.
The choice of parametrisation influences the form of the model. It
is desired to use a parametrisation that yields a model with the least
number of parameters. Approaches to reduce the number of parameters
include compressed sensing \cite{Donoho2006,billings2013nonlinear,Brunton2014,BruntonPNAS2016,Champion2019Autoencoder}
and normal form methods \cite{Read1998NormalForm,Yair2017DiffusionNormal,Cenedese2022NatComm}.
The methods to obtain a parametrisation of the manifold include diffusion
maps \cite{Coifman2006}, isomaps \cite{Tenenbaum2000isomap}, autoencoders
\cite{Kramer1991autoencoder,Champion2019Autoencoder,KaliaMeijerBrunton2021,Cenedese2022NatComm}
and equation-free models \cite{Kevrekidis2003,Samey}.

The main focus of this paper is genuine ROMs, where we need to find
structure in a cloud of data as illustrated in figure \ref{fig:data-structures}(b).
An invariant manifold provides such structure, since all trajectories
starting from the manifold stay on the manifold for all times. However
an invariant manifold is not defined by the dynamics on it but by
the dynamics in its neighbourhood \cite{Fenichel,delaLlave1997,CabreLlave2003}.
This means that identifying a manifold must also involve identifying
the dynamics in its neighbourhood, which is regularly omitted, such
as in \cite{Cenedese2022NatComm}. We find that resolving the nearby
dynamics is the same as identifying an invariant foliation, as explained
in section \ref{subsec:LocallyAccurateEncoder}. 

Invariant foliations \cite{Szalai2020ISF} can also be used to find
structure in data like in figure \ref{fig:data-structures}(b). An
invariant foliation consists of a family of leaves (manifolds) that
map onto each other under the dynamics of our system as in figure
\ref{fig:manifold-foliation}(a). Each leaf of the foliation has a
parameter from a space which has the dimensionality of the phase space
minus the dimensionality of a leaf. As the leaves map onto each other,
so do the parameters, which then defines a low-dimensional map or
ROM. A leaf that maps onto itself is an invariant manifold. Invariant
foliations are also useful to find suitable initial conditions for
ROMs \cite{Roberts89}.

A natural question is if we have exhausted all possible concepts that
identify structure in dynamic data. To address this, we systematically
explore how a physical (or any other data producing) system can be
related to a mathematical model. This allows us to categorise ROM
concepts and choose the most suitable one for a given purpose. Through
this process we arrive at the definitions of four cases: invariant
foliations, invariant manifolds, autoencoders and equation-free models,
and uncover their relations to each other. 

We find that not all methods are applicable to off-line produced data.
Indeed, calculating an invariant manifold requires actively probing
the mathematical model \cite{Haller2016}. If a model is not available,
the physical system must be placed under \emph{closed-loop} control,
such as in control-based continuation \cite{SieberCBC2008,BartonCBC2017},
which can currently identify equilibria or periodic orbits, but potentially
it could also be used to find invariant manifolds of those equilibria
and periodic orbits. Equation-free models were developed to recover
inherently low-order dynamics from large systems by systematically
probing the system input \cite{Kevrekidis2003,Samey}. However, not
all systems can be put under closed-loop control, mainly because either
setting up the required high-speed control loop is too costly or time
consuming, or the system is simply inaccessible to the data analyst.
In this case, the data collection is carried out separately from the
data analysis without instant feedback to the system. We call this
case \emph{open-loop} or \emph{off-line} data collection. We conclude
that invariant foliations and autoencoders can be fitted to off-line
data, but only invariant foliations produce genuine ROMs.

Surprisingly, invariant foliations were not explored for ROM identification
before paper \cite{Szalai2020ISF}. Here, we propose a two-stage process,
which finds an invariant foliation \cite{Szalai2020ISF}, that captures
a low-order model and then finds a transverse and locally defined
invariant foliation, whose fixed leaf is the invariant manifold with
the same dynamics as the globally defined invariant foliation. This
ensures that we take into account all data when the low-order dynamics
is uncovered, and at the same time also produce a familiar structure,
which is the corresponding invariant manifold. As per remark \ref{rem:extrasimpleROM},
this process may be simplified at the possible cost of losing some
accuracy.

Another contribution of the paper is that we resolve the problem with
high-dimensional data, which requires the approximation of high-dimensional
functions, when invariant foliations are identified. We use polynomials
that have compressed tensor coefficients \cite{TensorApproxSurvey},
and whose complexity scales linearly, instead of combinatorially with
the underlying dimension. Compressed tensors in the hierarchical Tucker
format (HT) \cite{HackbuschKuhn2009} are amenable to singular value
decomposition \cite{GrasedyckSVD}, hence by calculating the singular
values we can check the accuracy of the approximation. In addition,
compressed tensors can be truncated if some singular values are near
zero without losing their accuracy. However for our purpose, the most
important property is that within an optimisation framework a HT tensor
is linear in each of its parameter matrices (when the rest of the
parameters are fixed), which gives us a convex cost function. Indeed,
when solving the invariance equation of the foliation, we use a block
coordinate descent method, the Gauss-Southwell scheme \cite{GaussSouthwell2015},
which significantly speeds up the solution process. We also note that
optimisation of each coefficient matrix of a HT tensor is constrained
to a matrix manifold, hence we use a trust-region method designed
for matrix manifolds \cite{boumal2022intromanifolds,conn2000trust}
when solving for individual matrix components.

A ROM is usually represented in a nonlinear coordinate system, hence
the quantities predicted by the ROM may not behave the same way as
in Euclidean frames. This is particularly true for instantaneous frequencies
and damping ratios of decaying vibrations. In section \ref{sec:freq-damp}
we take the distortion of the nonlinear coordinate system into account
and derive correct values of instantaneous frequencies and damping
ratios.

The structure of the paper is as follows. We first discuss the type
of data assumed for ROM identification. Then we define what a ROM
is and go through all possible connections between a data producing
system and a ROM. We then classify the uncovered conceptual connections
along two properties: whether they are applicable to off-line data
or produce genuine ROMs. Next, we discuss instantaneous frequencies
and damping ratios in nonlinear frames. In section \ref{sec:ROM_id_processes}
we summarise the proposed and tested numerical algorithms, and describe
their implementation details. Finally, we discuss three example problems.
We start with a conceptual model that illustrates the non-applicability
of autoencoders to genuine model order reduction. We then use a nonlinear
ten-dimensional mathematical model to create synthetic data sets.
Here we demonstrate the accuracy of our method to full state-space
data, but also highlight problems with phase-space reconstruction
from scalar signals. Finally, we analyse vibration data from a jointed
beam, for which there is no accurate physical model due to the frictional
interface in the joint.

\subsection{Set-up}

The first step on our journey is to characterise the type of data
we are working with. We assume a real $n$-dimensional Euclidean space,
denoted by $X$ (a vector space with an inner product $\left\langle \cdot,\cdot\right\rangle _{X}:X\times X\to\mathbb{R}$),
which contains all our data. We further assume that the data is produced
by a deterministic process, which is represented by an unknown map
$\boldsymbol{F}:X\to X$. In particular, the data is organised into
$N\in\mathbb{N}^{+}$ pairs of vectors from space $X$, that is 
\[
\left(\boldsymbol{x}_{k},\boldsymbol{y}_{k}\right),\;k=1,\ldots,N
\]
that satisfy
\begin{equation}
\boldsymbol{y}_{k}=\boldsymbol{F}\left(\boldsymbol{x}_{k}\right)+\boldsymbol{\xi}_{k},\;k=1,\ldots,N,\label{eq:MAP-simple}
\end{equation}
where $\boldsymbol{\xi}_{k}\in X$ represents a small measurement
noise, which is sampled from a distribution with zero mean. Equation
(\ref{eq:MAP-simple}) describes pieces of trajectories if for some
$k\in\left\{ 1,\ldots,N\right\} $, $\boldsymbol{x}_{k+1}=\boldsymbol{y}_{k}$.
The state of the system can also be defined on a manifold, in which
case $X$ is chosen such that the manifold is embedded in $X$ according
to Whitney's embedding theorem \cite{WhitneyEmbedding1936}. It is
also possible that the state cannot be directly measured, in which
case Taken's delay embedding technique \cite{TakensEmbedding1981}
can be used to reconstruct the state, which we will do subsequently
in an optimal manner \cite{Casdagli1989DelayEmbed}. As a minimum,
we require that our system is observable \cite{Hermann1977NonlinearContrObs}.

For the purpose of this paper we also assume a fixed point at the
origin, such that $\boldsymbol{F}\left(\boldsymbol{0}\right)=\boldsymbol{0}$
and that the domain of $\boldsymbol{F}$ is a compact and connected
subset $G\subset X$ that includes a neighbourhood of the origin.

\section{Reduced order models}

We now describe a weak definition of a ROM, which only requires invariance.
There are two ingredients of a ROM: a function connecting vector space
$X$ to another vector space $Z$ of lower dimensionality and a map
on vector space $Z$. The connection can go two ways, either from
$Z$ to $X$ or from $X$ to $Z$. To make this more precise, we also
assume that $Z$ has an inner product $\left\langle \cdot,\cdot\right\rangle _{Z}:Z\times Z\to\mathbb{R}$
and that $\dim Z<\dim X$. The connection $\boldsymbol{U}:X\to Z$
is called an \emph{encoder} and the connection $\boldsymbol{W}:Z\to X$
is called a \emph{decoder}. We also assume that both the encoder and
the decoder are continuously differentiable and their Jacobians have
full rank. Our terminology is borrowed from computer science \cite{Kramer1991autoencoder},
but we can also use mathematical terms that calls $\boldsymbol{U}$
a manifold submersion \cite{Lawson1974} and $\boldsymbol{W}$ a manifold
immersion \cite{lang2012fundamentals}. In accordance with our assumption
that $\boldsymbol{F}\left(\boldsymbol{0}\right)=\boldsymbol{0}$,
we also assume that $\boldsymbol{U}\left(\boldsymbol{0}\right)=\boldsymbol{0}$
and $\boldsymbol{W}\left(\boldsymbol{0}\right)=\boldsymbol{0}$.
\begin{defn}
\label{def:ROM}Assume two maps $\boldsymbol{F}:X\to X$, $\boldsymbol{S}:Z\to Z$
and an encoder $\boldsymbol{U}:X\to Z$ or a decoder $\boldsymbol{W}:Z\to X$.
\begin{enumerate}
\item The encoder-map pair $\left(\boldsymbol{U},\boldsymbol{S}\right)$
is a \emph{ROM} of $\boldsymbol{F}$ if for all initial conditions
$\boldsymbol{x}_{0}\in G\subset X$, the trajectory $\boldsymbol{x}_{k+1}=\boldsymbol{F}\left(\boldsymbol{x}_{k}\right)$
and for initial condition $\boldsymbol{z}_{0}=\boldsymbol{U}\left(\boldsymbol{x}_{0}\right)$
the second trajectory $\boldsymbol{z}_{k+1}=\boldsymbol{S}\left(\boldsymbol{z}_{k}\right)$
are connected such that $\boldsymbol{z}_{k}=\boldsymbol{U}\left(\boldsymbol{x}_{k}\right)$
for all $k>0$.
\item The decoder-map pair $\left(\boldsymbol{W},\boldsymbol{S}\right)$
is a \emph{ROM} of $\boldsymbol{F}$ if for all initial conditions
$\boldsymbol{z}_{0}\in H=\left\{ \boldsymbol{z}\in Z:\boldsymbol{W}\left(z\right)\in G\right\} $,
the trajectory $\boldsymbol{z}_{k+1}=\boldsymbol{S}\left(\boldsymbol{z}_{k}\right)$
and for initial condition $\boldsymbol{x}_{0}=\boldsymbol{W}\left(\boldsymbol{z}_{0}\right)$
the second trajectory $\boldsymbol{x}_{k+1}=\boldsymbol{F}\left(\boldsymbol{x}_{k}\right)$
are connected such that $\boldsymbol{x}_{k}=\boldsymbol{W}\left(\boldsymbol{z}_{k}\right)$
for all $k>0$.
\end{enumerate}
\end{defn}
\begin{rem}
In essence, a ROM is a model whose trajectories are connected to the
trajectories of our system $\boldsymbol{F}$. We call this property
\emph{invariance.} It is possible to define a weaker ROM, where the
connections $\boldsymbol{z}_{k}=\boldsymbol{U}\left(\boldsymbol{x}_{k}\right)$
or $\boldsymbol{x}_{k}=\boldsymbol{W}\left(\boldsymbol{z}_{k}\right)$
are only approximate, which is not discussed here.
\end{rem}
The following corollary can be thought of as an equivalent definition
of a ROM.
\begin{cor}
The encoder-map pair $\left(\boldsymbol{U},\boldsymbol{S}\right)$
or decoder-map pair $\left(\boldsymbol{W},\boldsymbol{S}\right)$
is a ROM if and only if either invariance equation
\begin{align}
\boldsymbol{S}\left(\boldsymbol{U}\left(\boldsymbol{x}\right)\right) & =\boldsymbol{U}\left(\boldsymbol{F}\left(\boldsymbol{x}\right)\right),\quad\boldsymbol{x}\in G\quad\text{or}\label{eq:MAP-U-invariance}\\
\boldsymbol{W}\left(\boldsymbol{S}\left(\boldsymbol{z}\right)\right) & =\boldsymbol{F}\left(\boldsymbol{W}\left(\boldsymbol{z}\right)\right),\;\boldsymbol{z}\in H,\label{eq:MAP-W-invariance}
\end{align}
hold, where $H=\left\{ \boldsymbol{z}\in Z:\boldsymbol{W}\left(z\right)\in G\right\} $.
\end{cor}
\begin{proof}
Let us assume that (\ref{eq:MAP-U-invariance}) holds, and choose
an $\boldsymbol{x}_{0}\in X$ and let $\boldsymbol{z}_{0}=\boldsymbol{U}\left(\boldsymbol{x}_{0}\right)$.
First we check whether $\boldsymbol{z}_{k}=\boldsymbol{U}\left(\boldsymbol{x}_{k}\right)$,
if $\boldsymbol{x}_{k}=\boldsymbol{F}^{k}\left(\boldsymbol{x}_{0}\right)$,
$\boldsymbol{z}_{k}=\boldsymbol{S}^{k}\left(\boldsymbol{z}_{0}\right)$
and (\ref{eq:MAP-U-invariance}) hold. Substituting $\boldsymbol{x}=\boldsymbol{F}^{k}\left(\boldsymbol{x}_{0}\right)$
into (\ref{eq:MAP-U-invariance}) yields 
\begin{align*}
\boldsymbol{S}^{k}\left(\boldsymbol{z}_{0}\right) & =\boldsymbol{U}\left(\boldsymbol{F}^{k}\left(\boldsymbol{x}_{0}\right)\right)\\
\boldsymbol{z}_{k} & =\boldsymbol{U}\left(\boldsymbol{x}_{k}\right).
\end{align*}
Now in reverse, assuming that $\boldsymbol{z}_{k}=\boldsymbol{U}\left(\boldsymbol{x}_{k}\right)$,
$\boldsymbol{z}_{k}=\boldsymbol{S}^{k}\left(\boldsymbol{z}_{0}\right)$,
$\boldsymbol{x}_{k}=\boldsymbol{F}^{k}\left(\boldsymbol{x}_{0}\right)$
and setting $k=1$ yields that 
\begin{align*}
\boldsymbol{z}_{1}=\boldsymbol{S}\left(\boldsymbol{z}_{0}\right) & =\boldsymbol{U}\left(\boldsymbol{x}_{1}\right),\\
\boldsymbol{S}\left(\boldsymbol{U}\left(\boldsymbol{x}_{0}\right)\right) & =\boldsymbol{U}\left(\boldsymbol{F}\left(\boldsymbol{x}_{0}\right)\right),
\end{align*}
which is true for all $\boldsymbol{x}_{k-1}\in G$, hence (\ref{eq:MAP-U-invariance})
holds. The proof for the decoder is identical, except that we swap
$\boldsymbol{F}$ and $\boldsymbol{S}$ and replace $\boldsymbol{U}$
with $\boldsymbol{W}.$
\end{proof}
\begin{figure}[H]
\begin{centering}
\begin{center}
{\normalsize
a)
\begin{tikzcd} \tikz \node[draw,circle]{\textcolor{blue}{$X$}}; \arrow[r, dashed, "\boldsymbol{F}"] \arrow[d, "\boldsymbol{U}"] & X \arrow[d, dashed, "\boldsymbol{U}"]  \\ 
Z \arrow[r, "\boldsymbol{S}"]& \tikz \node[draw,rectangle]{\textcolor{red}{$Z$}}; \end{tikzcd}
b) 
\begin{tikzcd} X \arrow[r, dashed, "\boldsymbol{F}"] & \tikz \node[draw,rectangle]{\textcolor{red}{$X$}};   \\ 
\tikz \node[draw,circle]{\textcolor{blue}{$Z$}}; \arrow[r, "\boldsymbol{S}"]  \arrow[u,  dashed, "\boldsymbol{W}"]& Z \arrow[u, "\boldsymbol{W}"] \end{tikzcd}\\
c)\begin{tikzcd} \tikz \node[draw,circle]{\textcolor{blue}{$X$}}; \arrow[r, dashed, "\boldsymbol{F}"] \arrow[d, "\boldsymbol{U}"] & \tikz \node[draw,rectangle]{\textcolor{red}{$X$}}; \\
Z \arrow[r, "\boldsymbol{S}"]                             & Z \arrow[u, "\boldsymbol{W}"]                             \end{tikzcd} d)
\begin{tikzcd} X \arrow[r, dashed,  "\boldsymbol{F}"]                              & X  \arrow[d, dashed, "\boldsymbol{U}"]                             \\ 
\tikz \node[draw,circle]{\textcolor{blue}{$Z$}}; \arrow[r, "\boldsymbol{S}"] \arrow[u, dashed, "\boldsymbol{W}"] & \tikz \node[draw,rectangle]{\textcolor{red}{$Z$}}; \end{tikzcd}
}
\par\end{center}
\par\end{centering}
\caption{\label{fig:straight-commutative}Commutative diagrams of a) invariant
foliations, b) invariant manifolds, and connection diagrams of c)
autoencoders and d) equation-free models. The dashed arrows denote
the chain of function composition(s) that involves $\boldsymbol{F}$,
the continuous arrows denote the chain of function composition(s)
that involves map $\boldsymbol{S}$. The encircled vector space denotes
the domain of the invariance equation and the boxed vector space is
the target of the invariance equation.}
\end{figure}
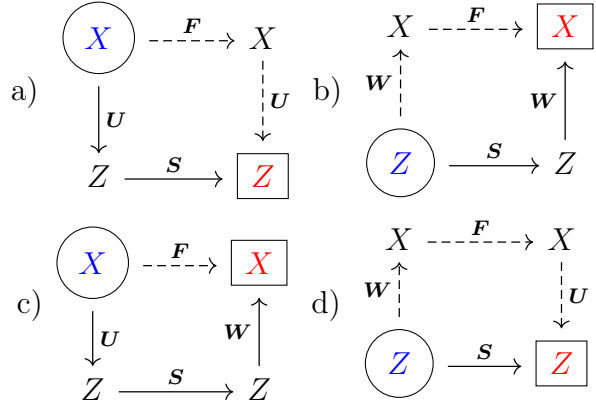

Now the question arises if we can use a combination of an encoder
and a decoder? This is the case of the autoencoder or nonlinear principal
component analysis \cite{Kramer1991autoencoder} and the equation-free
model \cite{Kevrekidis2003,Samey}. Indeed, there are four combinations
of encoders and decoders, which are depicted in diagrams \ref{fig:straight-commutative}(a,b,c,d).
We name these four scenarios as follows.
\begin{defn}
\label{def:Foil-Manif-AEnc}We call the connections displayed in figures
\ref{fig:straight-commutative}(a,b,c,d), invariant foliation, invariant
manifold, autoencoder and equation-free model, respectively.
\end{defn}
\begin{figure}
\begin{centering}
\includegraphics{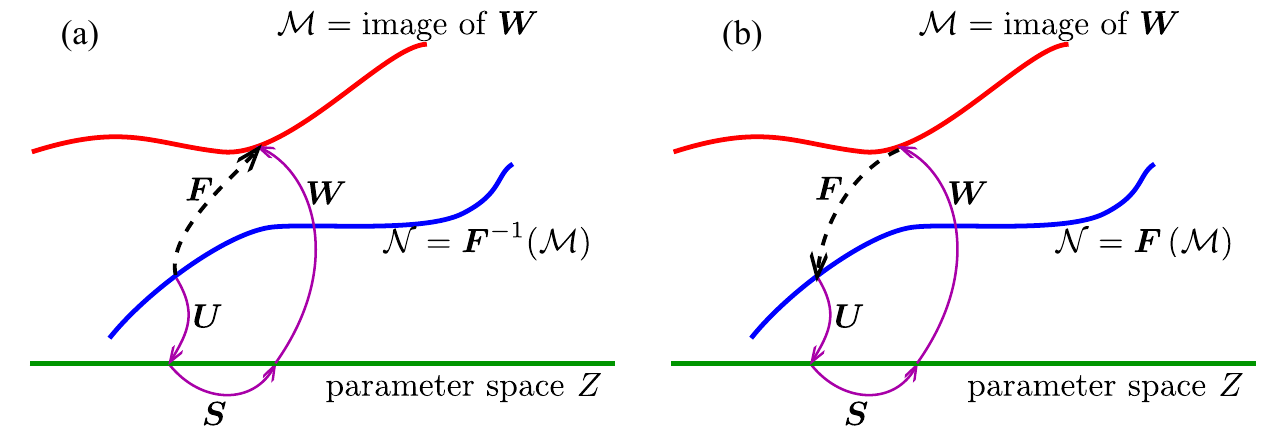}
\par\end{centering}
\caption{\label{fig:autoencoder}(a) Autoencoder. Encoder $\boldsymbol{U}$
maps data to parameter space $Z$, then the low-order map $\boldsymbol{S}$
takes the parameter forward in time and finally the decoder $\boldsymbol{W}$
brings the new parameter back to the state space. This chain of maps
denoted by solid purple arrows must be the same as map $\boldsymbol{F}$
denoted by the dashed arrow. The two results can only match between
two manifolds $\mathcal{M}$ and $\mathcal{N}$ and not elsewhere
in the state space. Manifold $\mathcal{M}$ is invariant if and only
if $\mathcal{M}\subset\mathcal{N}$, which is not guaranteed by this
construction. (b) Equation-free model. The only difference from the
autoencoder is the reversal of the arrow under $\boldsymbol{F}$,
hence $\boldsymbol{S}\left(\boldsymbol{z}\right)=\boldsymbol{U}\left(\boldsymbol{F}\left(\boldsymbol{W}\left(\boldsymbol{z}\right)\right)\right)$.}
\end{figure}

As we walk through the dashed and solid arrows in diagram \ref{fig:straight-commutative}(a),
we find the two sides of the invariance equation (\ref{eq:MAP-U-invariance}).
If we do the same for diagram \ref{fig:straight-commutative}(b),
we find equation (\ref{eq:MAP-W-invariance}). This implies that invariant
foliations and invariant manifolds are ROMs. The same is not true
for autoencoders and equation-free models. Reading off diagram \ref{fig:straight-commutative}(c),
the autoencoder must satisfy
\begin{equation}
\boldsymbol{W}\left(\boldsymbol{S}\left(\boldsymbol{U}\left(\boldsymbol{x}\right)\right)\right)=\boldsymbol{F}\left(\boldsymbol{x}\right).\label{eq:MAP-AE-noninvar}
\end{equation}
Equation (\ref{eq:MAP-AE-noninvar}) is depicted in figure \ref{fig:autoencoder}(a).
Since the decoder $\boldsymbol{W}$ maps onto a manifold $\mathcal{M}$,
equation (\ref{eq:MAP-AE-noninvar}) can only hold if $\boldsymbol{x}$
is chosen from the preimage of $\mathcal{M}$, that is, $\boldsymbol{x}\in\mathcal{N}=\boldsymbol{F}^{-1}\left(\mathcal{M}\right)$.
For invariance, we need $\boldsymbol{F}\left(\mathcal{M}\right)\subset\mathcal{M}$,
which is the same as $\mathcal{M}\subset\mathcal{N}$ if $\boldsymbol{F}$
is invertible. However, the inclusion $\mathcal{M}\subset\mathcal{N}$
is not guaranteed by equation (\ref{eq:MAP-AE-noninvar}). The only
way to guarantee $\mathcal{M}\subset\mathcal{N}$, is by stipulating
that the function composition $\boldsymbol{W}\circ\boldsymbol{U}$
is the identity map on the data. A trivial case is when $\dim Z=\dim X$
or, in general, when all data fall onto a $\dim Z$ dimensional submanifold
of $X$. Indeed, the standard way to find an autoencoder \cite{Kramer1991autoencoder}
is to solve
\begin{equation}
\arg\min_{\boldsymbol{U},\boldsymbol{W}}\sum_{k=1}^{N}\left\Vert \boldsymbol{W}\left(\boldsymbol{U}\left(\boldsymbol{x}_{k}\right)\right)-\boldsymbol{x}_{k}\right\Vert ^{2}.\label{eq:AE-find}
\end{equation}
Unfortunately, if the data is not on a $\dim Z$ dimensional submanifold
of $X$, which is the case of genuine ROMs, the minimum of $\sum_{k=1}^{N}\left\Vert \boldsymbol{W}\left(\boldsymbol{U}\left(\boldsymbol{x}_{k}\right)\right)-\boldsymbol{x}_{k}\right\Vert ^{2}$
will be far from zero and the location of $\mathcal{M}$ as the solution
of (\ref{eq:AE-find}) will only indicate where the data is in the
state space. It is customary to seek a solution to (\ref{eq:AE-find})
under the normalising condition that $\boldsymbol{U}\circ\boldsymbol{W}$
is the identity and hence $\boldsymbol{S}=\boldsymbol{U}\circ\boldsymbol{F}\circ\boldsymbol{W}$.

The equation-free model in diagram \ref{fig:straight-commutative}(d)
is identical to the autoencoder \ref{fig:straight-commutative}(c)
if we replace $\boldsymbol{F}$ with $\boldsymbol{F}^{-1}$ and $\boldsymbol{S}$
with $\boldsymbol{S}^{-1}$. Reading off diagram \ref{fig:straight-commutative}(d),
the equation-free model must satisfy
\begin{equation}
\boldsymbol{S}\left(\boldsymbol{z}\right)=\boldsymbol{U}\left(\boldsymbol{F}\left(\boldsymbol{W}\left(\boldsymbol{z}\right)\right)\right).\label{eq:MAP-RAE-noninvar}
\end{equation}
Equation (\ref{eq:MAP-RAE-noninvar}) immediately provides $\boldsymbol{S}$,
for any $\boldsymbol{U},$ $\boldsymbol{W}$, without identifying
any structure in the data.

Only invariant foliations, through equation (\ref{eq:MAP-U-invariance})
and autoencoders through equations (\ref{eq:MAP-AE-noninvar}) and
(\ref{eq:AE-find}) can be fitted to off-line data. Invariant manifolds
and equation-free models require the ability the manipulate the input
of our system $\boldsymbol{F}$ during the identification process.
Indeed, for off-line data, produced by equation (\ref{eq:MAP-simple}),
the foliation invariance equation (\ref{eq:MAP-U-invariance}) turns
into an optimisation problem
\begin{equation}
\arg\min_{\boldsymbol{S},\boldsymbol{U}}\sum_{k=1}^{N}\left\Vert \boldsymbol{x}_{k}\right\Vert ^{-2}\left\Vert \boldsymbol{S}\left(\boldsymbol{U}\left(\boldsymbol{x}_{k}\right)\right)-\boldsymbol{U}\left(\boldsymbol{y}_{k}\right)\right\Vert ^{2},\label{eq:MAP-U-optim}
\end{equation}
and the autoencoder equation (\ref{eq:MAP-AE-noninvar}) turns into
\begin{equation}
\arg\min_{\boldsymbol{S},\boldsymbol{U},\boldsymbol{W}}\sum_{k=1}^{N}\left\Vert \boldsymbol{x}_{k}\right\Vert ^{-2}\left\Vert \boldsymbol{W}\left(\boldsymbol{S}\left(\boldsymbol{U}\left(\boldsymbol{x}_{k}\right)\right)\right)-\boldsymbol{y}_{k}\right\Vert ^{2}.\label{eq:MAP-AE-optim}
\end{equation}
For the optimisation problem (\ref{eq:MAP-U-optim}) to have a unique
solution, we need to apply a constraint to $\boldsymbol{U}$. One
possible constraint is explained in remark \ref{rem:U-constraint}.
In case of the autoencoder (\ref{eq:MAP-AE-optim}), the constraint,
in addition to the one restricting $\boldsymbol{U}$, can be that
$\boldsymbol{U}\circ\boldsymbol{W}$ is the identity, as also stipulated
in \cite{Cenedese2022NatComm}.

\subsection{\label{subsec:FoliationsManifolds}Invariant foliations and invariant
manifolds}

\begin{figure}
\begin{centering}
\includegraphics[width=0.49\linewidth]{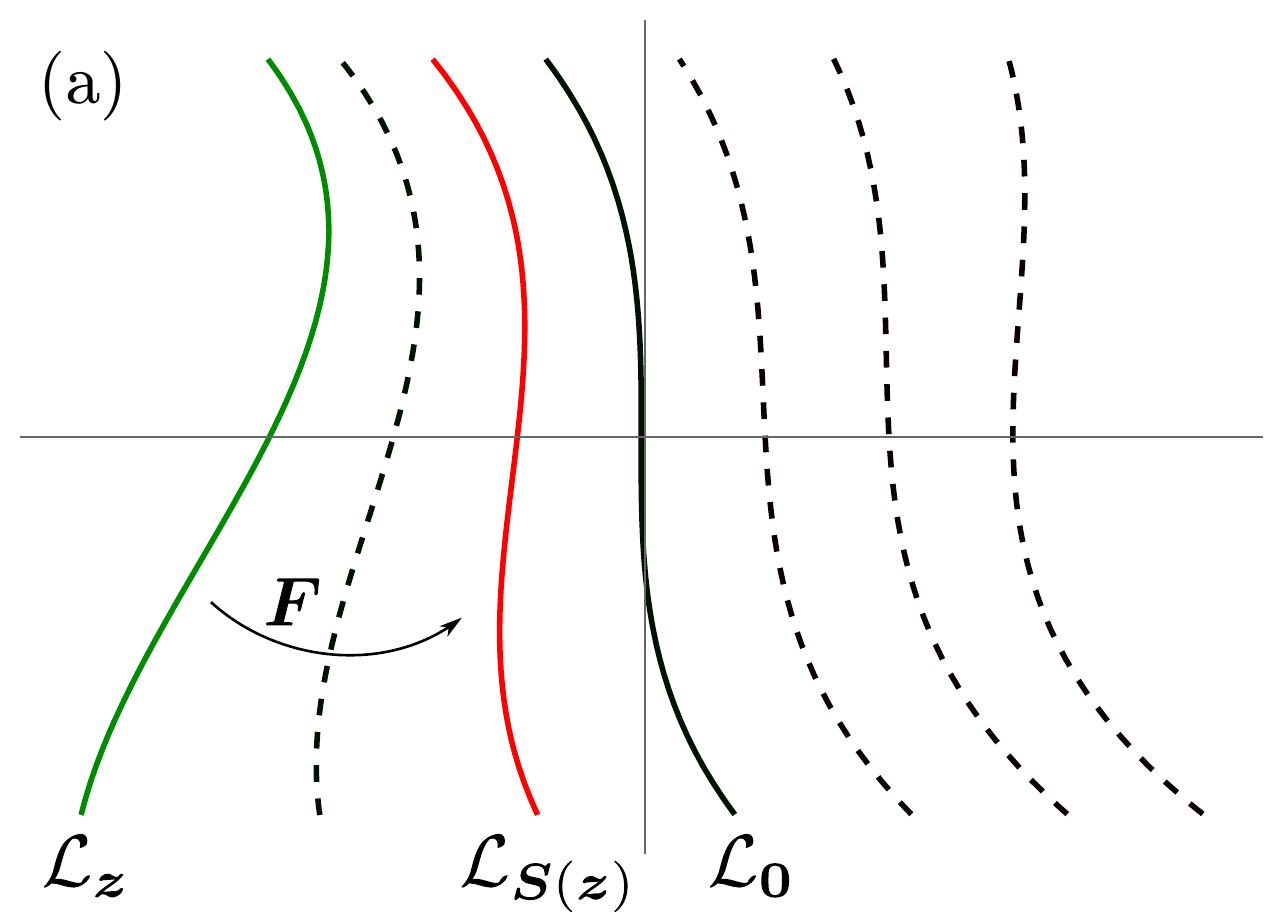}\includegraphics[width=0.49\linewidth]{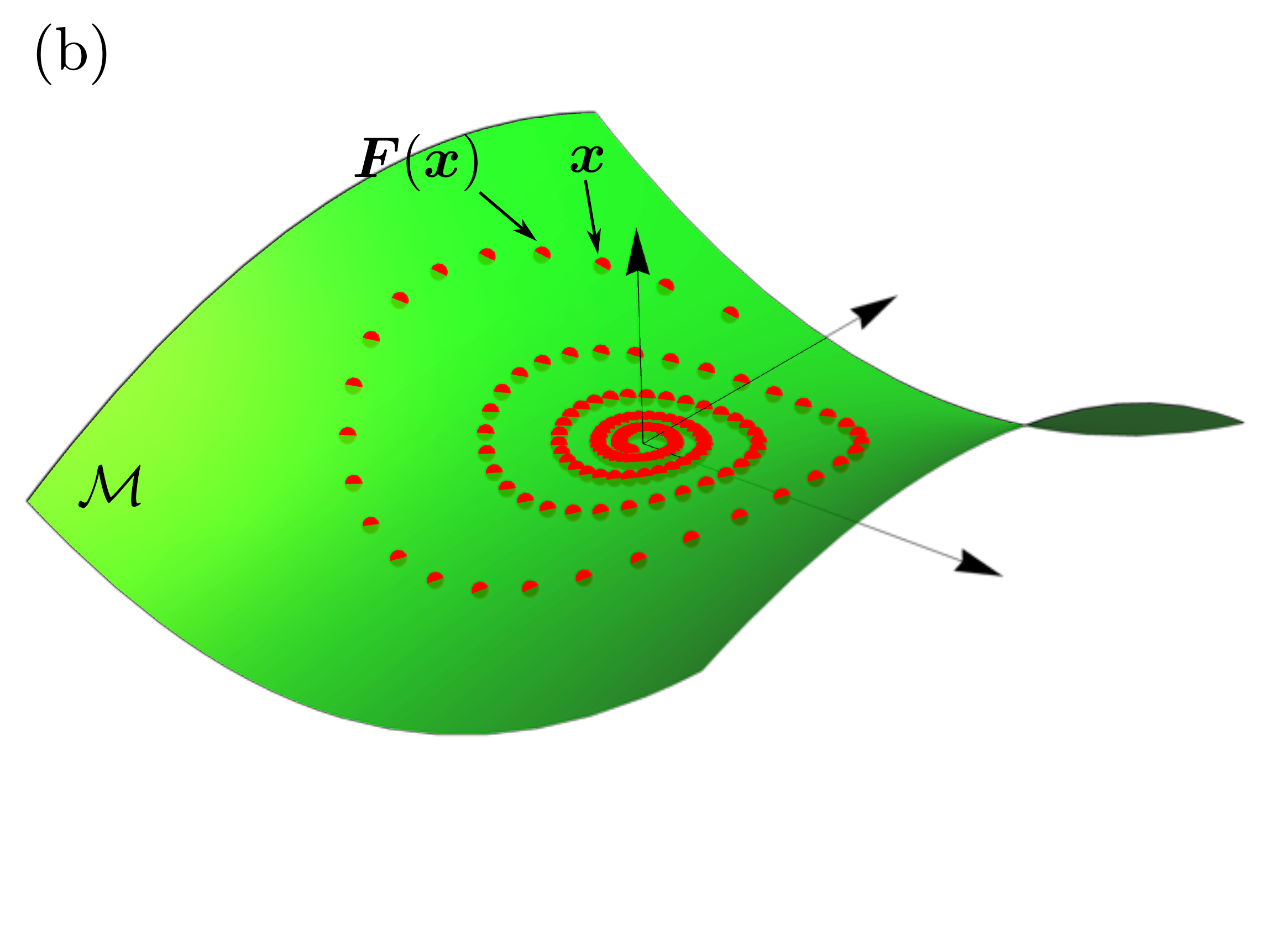}
\par\end{centering}
\caption{\label{fig:manifold-foliation}(a) Invariant foliation: a leaf $\mathcal{L}_{\boldsymbol{z}}$
is being mapped onto leaf $\mathcal{L}_{\boldsymbol{S}\left(\boldsymbol{z}\right)}$.
Since the origin $\boldsymbol{0}\in Z$ is a steady state of $\boldsymbol{S}$,
the leaf $\mathcal{L}_{\boldsymbol{0}}$ is mapped into itself and
therefore it is an invariant manifold. (b) Invariant manifold: a single
trajectory, represented by red dots, starting from the invariant manifold
$\mathcal{M}$ remains on the green invariant manifold.}
\end{figure}
An encoder represents a family of manifolds, called \emph{foliation},
which is the set of constant level surfaces of $\boldsymbol{U}$.
A single level surface is called a \emph{leaf} of the foliation, hence
the foliation is a collection of leaves. In mathematical terms a leaf
with parameter $\boldsymbol{z}\in Z$ is denoted by
\begin{equation}
\mathcal{L}_{\boldsymbol{z}}=\left\{ \boldsymbol{x}\in G\subset X:\boldsymbol{U}\left(\boldsymbol{x}\right)=\boldsymbol{z}\right\} .\label{eq:MAP-leaf-def}
\end{equation}
All leaves are $\dim X-\dim Z$ dimensional differentiable manifolds,
because we assumed that the Jacobian $D\boldsymbol{U}$ has full rank
\cite{Lawson1974}. The collection of all leaves is a \emph{foliation,}
denoted by 
\[
\mathcal{F}=\left\{ \mathcal{L}_{\boldsymbol{z}}:\boldsymbol{z}\in H\right\} ,
\]
where $H=\boldsymbol{U}\left(G\right)$. The foliation is characterised
by its co-dimension, which is the same as $\dim Z$. Invariance equation
(\ref{eq:MAP-U-invariance}) means that each leaf of the foliation
is mapped onto another leaf, in particular the leaf with parameter
$\boldsymbol{z}$ is mapped onto the leaf with parameter $\boldsymbol{S}\left(\boldsymbol{z}\right)$,
that is 
\[
\boldsymbol{F}\left(\mathcal{L}_{\boldsymbol{z}}\right)\subset\mathcal{L}_{\boldsymbol{S}\left(\boldsymbol{z}\right)}.
\]
Due to our assumptions, leaf $\mathcal{L}_{\boldsymbol{0}}$ is an
invariant manifold, because $\boldsymbol{F}\left(\mathcal{L}_{\boldsymbol{0}}\right)\subset\mathcal{L}_{\boldsymbol{0}}$.
This geometry is illustrated in figure \ref{fig:manifold-foliation}(a).

We now characterise the existence and uniqueness of invariant foliations
about a fixed point. We assume that $\boldsymbol{F}$ is a $C^{r}$,
$r\ge2$ map and that the Jacobian matrix $D\boldsymbol{F}\left(\boldsymbol{0}\right)$
has eigenvalues $\mu_{1},\ldots,\mu_{n}$ such that $\left|\mu_{i}\right|<1$,
$i=1,\ldots,n$. To select the invariant manifold or foliation we
assume two $\nu$-dimensional linear subspaces $E$ of $X$ and $E^{\star}$
of $X$ corresponding to eigenvalues $\mu_{1},\ldots,\mu_{\nu}$ such
that $D\boldsymbol{F}\left(\boldsymbol{0}\right)E\subset E$ and for
the adjoint map $\left(D\boldsymbol{F}\left(\boldsymbol{0}\right)\right)^{\star}E^{\star}\subset E^{\star}$.
\begin{defn}
The number
\[
\beth_{E^{\star}}=\frac{\min_{k=1\ldots\nu}\log\left|\mu_{k}\right|}{\max_{k=1\ldots n}\log\left|\mu_{k}\right|}
\]
is called the spectral quotient of the left-invariant linear subspace
$E^{\star}$ of $\boldsymbol{F}$ about the origin.
\end{defn}
\begin{thm}
\label{thm:MapFoliation}Assume that $D\boldsymbol{F}\left(\boldsymbol{0}\right)$
is semisimple and that there exists an integer $\sigma\ge2$, such
that $\beth_{E^{\star}}<\sigma\le r$. Also assume that
\begin{equation}
\prod_{k=1}^{n}\mu_{k}^{m_{k}}\neq\mu_{j},\;j=1,\ldots,\nu\label{eq:MapFoliationNonResonance}
\end{equation}
for all $m_{k}\ge0$, $1\le k\le n$ with at least one $m_{l}\neq0$,
$\nu+1\le l\le n$ and with $\sum_{k=0}^{n}m_{k}\le\sigma-1$. Then
in a sufficiently small neighbourhood of the origin there exists an
invariant foliation $\mathcal{F}$ tangent to the left-invariant linear
subspace $E^{\star}$ of the $C^{r}$ map $\boldsymbol{F}$. The foliation
$\mathcal{F}$ is unique among the $\sigma$ times differentiable
foliations and it is also $C^{r}$ smooth.
\end{thm}
\begin{proof}
The proof is carried out in \cite{Szalai2020ISF}. Note that the assumption
of $D\boldsymbol{F}\left(\boldsymbol{0}\right)$ being semisimple
was used to simplify the proof in \cite{Szalai2020ISF}, therefore
it is unlikely to be needed.
\end{proof}
\begin{rem}
\label{rem:U-constraint}Theorem \ref{thm:MapFoliation} only concerns
the uniqueness of the foliation, but not the encoder $\boldsymbol{U}$.
However, for any smooth and invertible map $\boldsymbol{R}:Z\to Z$,
the encoder $\tilde{\boldsymbol{U}}=\boldsymbol{R}\circ\boldsymbol{U}$
represents the same foliation and the nonlinear map $\boldsymbol{S}$
transforms into $\tilde{\boldsymbol{S}}=\boldsymbol{R}\circ\boldsymbol{S}\circ\boldsymbol{R}^{-1}$.
If we want to solve the invariance equation (\ref{eq:MAP-U-invariance}),
we need to constrain $\boldsymbol{U}$. The simplest such constraint
is that 
\begin{equation}
\boldsymbol{U}\left(\boldsymbol{W}_{1}\boldsymbol{z}\right)=\boldsymbol{z},\label{eq:FOIL-graph}
\end{equation}
where $\boldsymbol{W}_{1}:Z\to X$ is a linear map with full rank
such that $E^{\star}\cap\ker\boldsymbol{W}_{1}^{\star}=\left\{ \boldsymbol{0}\right\} $.
To explain the meaning of equation (\ref{eq:FOIL-graph}), we note
that the image of $\boldsymbol{W}_{1}$ is a linear subspace $\mathcal{H}$
of $X$. Equation (\ref{eq:FOIL-graph}) therefore means that each
leaf $\mathcal{L}_{\boldsymbol{z}}$ must intersect subspace $\mathcal{H}$
exactly at parameter $\boldsymbol{z}$. The condition $E^{\star}\cap\ker\boldsymbol{W}_{1}^{\star}=\left\{ \boldsymbol{0}\right\} $
then means that the leaf $\mathcal{L}_{\boldsymbol{0}}$ has a transverse
intersection with subspace $\mathcal{H}$. This is similar to the
graph-style parametrisation of a manifold over a linear subspace.
\end{rem}
\begin{rem}
\label{rem:Koopman}Eigenvalues and eigenfunctions of the Koopman
operator \cite{Mezic2005,Mezic2021} are invariant foliations. Indeed,
the Koopman operator is defined as $\left(\mathcal{K}\boldsymbol{u}\right)\left(\boldsymbol{x}\right)=\boldsymbol{u}\left(\boldsymbol{F}\left(\boldsymbol{x}\right)\right)$.
If we assume that $\boldsymbol{U}=\left(\boldsymbol{u}_{1},\ldots,\boldsymbol{u}_{\nu}\right)$
is a collection of functions $\boldsymbol{u}_{j}:X\to\mathbb{R}$,
$Z=\mathbb{R}^{\nu}$, then $\boldsymbol{U}$ spans an invariant subspace
of $\mathcal{K}$ if there exists a linear map $\boldsymbol{S}$ such
that $\mathcal{K}\left(\boldsymbol{U}\right)=\boldsymbol{S}\boldsymbol{U}$.
Expanding this equation yields $\boldsymbol{U}\left(\boldsymbol{F}\left(\boldsymbol{x}\right)\right)=\boldsymbol{S}\boldsymbol{U}\left(\boldsymbol{x}\right)$,
which is the same as the invariance equation (\ref{eq:MAP-U-invariance}),
except that $\boldsymbol{S}$ is linear. The existence of linear map
$\boldsymbol{S}$ requires further non-resonance conditions, which
are
\begin{equation}
\prod_{k=1}^{\nu}\mu_{k}^{m_{k}}\neq\mu_{j},\;j=1,\ldots,\nu\label{eq:Internal-nonresonance}
\end{equation}
for all $m_{k}\ge0$ such that $\sum_{k=0}^{n}m_{k}\le\sigma-1$.
Equation (\ref{eq:Internal-nonresonance}) is referred to as the set
of \emph{internal non-resonance} conditions, because these are intrinsic
to the invariant subspace $E^{\star}$. In many cases $\boldsymbol{S}$
represents the slowest dynamics, hence even if there are no internal
resonances, the two sides of (\ref{eq:Internal-nonresonance}) will
be close to each other for some set of $m_{1},\ldots,m_{\nu}$ exponents
and that causes numerical issues leading to undesired inaccuracies.
We will illustrate this in section \ref{subsec:10dimsys-example}.
\end{rem}
Now we discuss invariant manifolds. A decoder $\boldsymbol{W}$ defines
a differentiable manifold
\[
\mathcal{M}=\left\{ \boldsymbol{W}\left(\boldsymbol{z}\right):\boldsymbol{z}\in H\right\} ,
\]
where $H=\left\{ \boldsymbol{z}\in Z:\boldsymbol{W}\left(\boldsymbol{z}\right)\in G\right\} $.
Invariance equation (\ref{eq:MAP-W-invariance}) is equivalent to
the geometric condition that $\boldsymbol{F}\left(\mathcal{M}\right)\subset\mathcal{M}$.
This geometry is shown in figure \ref{fig:manifold-foliation}(b),
which illustrates that if a trajectory is started on $\mathcal{M}$,
all subsequent points of the trajectory stay on $\mathcal{M}$. 

Invariant manifolds as a concept cannot be used to identify ROMs from
off-line data. As we will see below, invariant manifolds can still
be identified as a leaf of an invariant foliation, but not through
the invariance equation (\ref{eq:MAP-W-invariance}). Indeed, it is
not possible to guess the manifold parameter $\boldsymbol{z}\in Z$
from data. Introducing an encoder $\boldsymbol{U}:X\to Z$ to calculate
$\boldsymbol{z}=\boldsymbol{U}\left(\boldsymbol{x}\right)$, transforms
the invariant manifold into an autoencoder, which does not guarantee
invariance. 

We now state the conditions of the existence and uniqueness of an
invariant manifold.
\begin{defn}
The number 
\[
\aleph_{E}=\frac{\min_{k=\nu+1\ldots n}\log\left|\mu_{k}\right|}{\max_{k=1\ldots\nu}\log\left|\mu_{k}\right|}
\]
is called the spectral quotient of the right-invariant linear subspace
$E$ of map $\boldsymbol{F}$ about the origin.
\end{defn}
\begin{thm}
\label{thm:MapManifold}Assume that there exists an integer $\sigma\ge2$,
such that $\aleph_{E}<\sigma\le r$. Also assume that
\begin{equation}
\prod_{k=1}^{\nu}\mu_{k}^{m_{k}}\neq\mu_{j},\;j=\nu+1,\ldots,n\label{eq:MapManifNonResonance}
\end{equation}
for all $m_{k}\ge0$ such that $\sum_{k=0}^{\nu}m_{k}\le\sigma-1$.
Then in a sufficiently small neighbourhood of the origin there exists
an invariant manifold $\mathcal{\mathcal{M}}$ tangent to the invariant
linear subspace $E$ of the $C^{r}$ map $\boldsymbol{F}$. The manifold
$\mathcal{M}$ is unique among the $\sigma$-times differentiable
manifolds and it is also $C^{r}$ smooth.
\end{thm}
\begin{proof}
The theorem is a subset of theorem 1.1 in \cite{CabreLlave2003}.
\end{proof}
\begin{rem}
\label{rem:W-constraint}To calculate an invariant manifold with a
unique representation, we need to impose a constraint on $\boldsymbol{W}$
and/or $\boldsymbol{S}$. The simplest constraint is imposed by 
\begin{equation}
\boldsymbol{U}_{1}\boldsymbol{W}\left(\boldsymbol{z}\right)=\boldsymbol{z},\label{eq:MAP-W-graphstyle}
\end{equation}
where $\boldsymbol{U}_{1}:Z\to X$ is a linear map with full rank
such that $E\cap\ker\boldsymbol{U}_{1}^{\star}=\left\{ \boldsymbol{0}\right\} $.
This is similar to a graph-style parametrisation (akin to theorem
1.2 in \cite{CabreLlave2003}), where the range of $\boldsymbol{U}_{1}$
must span the linear subspace $E$. Constraint (\ref{eq:MAP-W-graphstyle})
can break down for large $\left\Vert \boldsymbol{z}\right\Vert $,
when $\boldsymbol{U}_{1}D\boldsymbol{W}\left(\boldsymbol{z}\right)$
does not have full rank. A globally suitable constraint is that $D\boldsymbol{W}^{\star}\left(\boldsymbol{z}\right)D\boldsymbol{W}\left(\boldsymbol{z}\right)=\boldsymbol{I}$.
\end{rem}
For linear subspaces $E$ and $E^{\star}$ with eigenvalues closest
to the complex unit circle (representing the slowest dynamics), $\beth_{E}=1$
and $\aleph_{E}$ is maximal. Therefore the foliation corresponding
to the slowest dynamics requires the least smoothness, while the invariant
manifold requires the maximum smoothness for uniqueness. 

Table \ref{tab:Comparison} summarises the main properties of the
three conceptually different model identification techniques. Ultimately,
in the presence of off-line data, only invariant foliations can be
fitted to the data and produce a ROM at the same time. 

\begin{table}
\begin{centering}
\begin{tabular}{|>{\centering}p{2.5cm}|>{\centering}p{2.5cm}|>{\centering}p{2.5cm}|>{\centering}p{2.5cm}|>{\centering}p{2.5cm}|}
\hline 
 & Invariant foliation & Invariant manifold & Autoencoder & Eq.-free model\tabularnewline
\hline 
\hline 
Usable closed-loop & YES & YES & YES & YES\tabularnewline
\hline 
Usable open-loop & YES & NO & YES & NO\tabularnewline
\hline 
Obtains a ROM & YES & YES & NO & NO\tabularnewline
\hline 
Uniqueness & slowest most unique & slowest least unique & NO & NO\tabularnewline
\hline 
References & \cite{Szalai2020ISF,Mezic2005,Roberts89,Roberts91} & \cite{ShawPierre,delaLlave1997,CabreLlave2003,CabreP3-2005,Haller2016,Szalai20160759,VIZZACCARO2021normalForm} & \cite{Cenedese2022NatComm,Champion2019Autoencoder,KaliaMeijerBrunton2021} & \cite{Kevrekidis2003,Samey}\tabularnewline
\hline 
\end{tabular}
\par\end{centering}
\caption{\label{tab:Comparison}Comparison of invariant foliations, invariant
manifolds and autoencoders. }
\end{table}

\subsection{\label{subsec:LocallyAccurateEncoder}Invariant manifolds represented
by locally defined invariant foliations}

As discussed before, we cannot fit invariant manifolds to data, instead
we can fit an invariant foliation that contains our invariant manifold
(which is the leaf containing the fixed point $\mathcal{L}_{\boldsymbol{0}}$
at the origin). This invariant foliation only needs to be defined
near the invariant manifold and therefore we can simplify the functional
representation of the encoder that defines the foliation. In this
section we discuss this simplification.

To begin with, assume that we already have an invariant foliation
$\mathcal{F}$ with an encoder $\boldsymbol{U}$ and nonlinear map
$\boldsymbol{S}$. Our objective is to find the invariant manifold
$\mathcal{M}$, represented by decoder $\boldsymbol{W}$ that has
the same dynamics $\boldsymbol{S}$ as the foliation. This is useful
if we want to know quantities that are only defined for invariant
manifolds, such as instantaneous frequencies and damping ratios. Formally,
we are looking for a simplified invariant foliation $\hat{\mathcal{F}}$
with encoder $\hat{\boldsymbol{U}}:X\to\hat{Z}$ that together with
$\mathcal{F}$ form a coordinate system in $X$. (Technically speaking,
$\left(\boldsymbol{U},\hat{\boldsymbol{U}}\right):X\to Z\times\hat{Z}$
must be a isomorphism.) In this case our invariant manifold is the
zero level surface of encoder $\hat{\boldsymbol{U}}$, i.e., $\mathcal{M}=\hat{\mathcal{L}}_{0}\in\hat{\mathcal{F}}$.
Naturally, we must have $\dim Z+\dim\hat{Z}=\dim X$.

\begin{figure}
\begin{centering}
\includegraphics[width=0.7\linewidth]{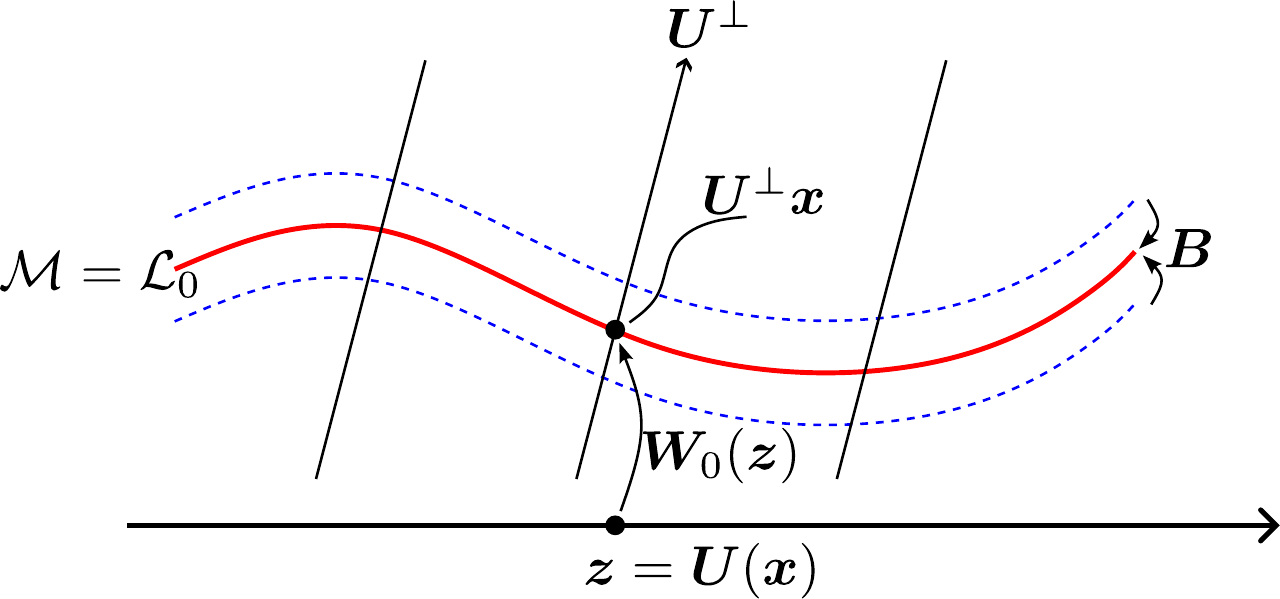}
\par\end{centering}
\caption{\label{fig:approx-foil}Approximate invariant foliation. The linear
map $\boldsymbol{U}^{\perp}$ and the nonlinear map $\boldsymbol{U}$
create a coordinate system, so that in this frame the invariant manifold
$\mathcal{M}$ is given by $\boldsymbol{U}^{\perp}\boldsymbol{x}=\boldsymbol{W}_{\boldsymbol{0}}\left(\boldsymbol{z}\right)$,
where $\boldsymbol{z}=\boldsymbol{U}\left(\boldsymbol{z}\right)$.
We then allow to shift manifold $\mathcal{M}$ in the $\boldsymbol{U}^{\perp}$
direction such that $\boldsymbol{U}^{\perp}\boldsymbol{x}-\boldsymbol{W}_{\boldsymbol{0}}\left(\boldsymbol{z}\right)=\hat{\boldsymbol{z}}$,
which creates a foliation parametrised by $\hat{\boldsymbol{z}}\in\hat{Z}$.
If this foliation satisfies the invariance equation in a neighbourhood
of $\mathcal{M}$ with a linear map $\boldsymbol{B}$ as per equation
(\ref{eq:Uhat-invariance}), then $\mathcal{M}$ is an invariant manifold.
In other words, nearby blue dashed leaves are mapped towards $\mathcal{M}$
by linear map $\boldsymbol{B}$ the same way as the underlying dynamics
does.}
\end{figure}

The sufficient condition that $\mathcal{F}$ and $\hat{\mathcal{F}}$
form a coordinate system locally about the fixed point is that the
square matrix
\[
\boldsymbol{Q}=\begin{pmatrix}D\boldsymbol{U}\left(\boldsymbol{0}\right)\\
D\hat{\boldsymbol{U}}\left(\boldsymbol{0}\right)
\end{pmatrix}
\]
is invertible. Let us represent our approximate (or locally defined)
encoder by
\begin{equation}
\hat{\boldsymbol{U}}\left(\boldsymbol{x}\right)=\boldsymbol{U}^{\perp}\boldsymbol{x}-\boldsymbol{W}_{0}\left(\boldsymbol{U}\left(\boldsymbol{x}\right)\right),\label{eq:Uhat-decoder}
\end{equation}
where $\boldsymbol{W}_{0}:Z\to\hat{Z}$ is a nonlinear map with $D\boldsymbol{W}_{0}\left(\boldsymbol{0}\right)=\boldsymbol{0}$,
$\boldsymbol{U}^{\perp}:X\to\hat{Z}$ is an orthogonal linear map,
$\boldsymbol{U}^{\perp}\left(\boldsymbol{U}^{\perp}\right)^{\star}=\boldsymbol{I}$
and $\boldsymbol{U}^{\perp}\boldsymbol{W}_{0}\left(\boldsymbol{z}\right)=\boldsymbol{0}$.
Here, the linear map $\boldsymbol{U^{\perp}}$ measures coordinates
in a transversal direction to the manifold and $\boldsymbol{W}_{0}$
prescribes where the actual manifold is along this transversal direction,
while $\boldsymbol{U}$ provides the parametrisation of the manifold.
All other leaves of the approximate foliation $\hat{\mathcal{F}}$
are shifted copies of $\mathcal{M}$ along the $\boldsymbol{U}^{\perp}$
direction as displayed in figure \ref{fig:approx-foil}. A locally
defined foliation means that $\hat{\boldsymbol{z}}=\hat{\boldsymbol{U}}\left(\boldsymbol{x}\right)\in\hat{Z}$
is assumed to be small, hence we can also assume linear dynamics among
the leaves of $\hat{\mathcal{F}}$, which is represented by a linear
operator $\boldsymbol{B}:\hat{Z}\to\hat{Z}$. Therefore the invariance
equation (\ref{eq:MAP-U-invariance}) becomes
\begin{equation}
\boldsymbol{B}\hat{\boldsymbol{U}}\left(\boldsymbol{x}\right)=\hat{\boldsymbol{U}}\left(\boldsymbol{F}\left(\boldsymbol{x}\right)\right).\label{eq:Uhat-invariance}
\end{equation}
Once\textbf{ $\boldsymbol{B}$}, $\boldsymbol{U}^{\perp}$ and $\boldsymbol{W}_{0}$
are found, the final step is to reconstruct the decoder $\boldsymbol{W}$
of our invariant manifold $\mathcal{M}$.
\begin{prop}
The decoder $\boldsymbol{W}$ of the invariant manifold $\mathcal{M}=\left\{ \boldsymbol{x}\in X:\hat{\boldsymbol{U}}\left(\boldsymbol{x}\right)=\boldsymbol{0}\right\} $
is the unique solution of the system of equations
\begin{equation}
\left.\begin{array}{rl}
\boldsymbol{U}\left(\boldsymbol{W}\left(\boldsymbol{z}\right)\right) & =\boldsymbol{z}\\
\boldsymbol{U}^{\perp}\boldsymbol{W}\left(\boldsymbol{z}\right) & =\boldsymbol{W}_{0}\left(\boldsymbol{z}\right)
\end{array}\right\} .\label{eq:Uhat-W-reconstruct}
\end{equation}
\end{prop}
\begin{proof}
First we show that conditions (\ref{eq:Uhat-W-reconstruct}) imply
$\hat{\boldsymbol{U}}\left(\boldsymbol{W}\left(\boldsymbol{x}\right)\right)=\boldsymbol{0}$.
We expand our expression using (\ref{eq:Uhat-decoder}) into $\hat{\boldsymbol{U}}\left(\boldsymbol{W}\left(\boldsymbol{x}\right)\right)=\boldsymbol{U}^{\perp}\boldsymbol{W}\left(\boldsymbol{z}\right)-\boldsymbol{W}_{0}\left(\boldsymbol{U}\left(\boldsymbol{W}\left(\boldsymbol{z}\right)\right)\right)$,
then use equations (\ref{eq:Uhat-W-reconstruct}), which yields $\hat{\boldsymbol{U}}\left(\boldsymbol{W}\left(\boldsymbol{x}\right)\right)=\boldsymbol{0}$.
To solve equations (\ref{eq:Uhat-W-reconstruct}) for $\boldsymbol{W}$,
we decompose $\boldsymbol{U}$, $\boldsymbol{W}$ into linear and
nonlinear components, such that $\boldsymbol{U}\left(\boldsymbol{x}\right)=D\boldsymbol{U}\left(0\right)\boldsymbol{x}+\widetilde{\boldsymbol{U}}\left(\boldsymbol{x}\right)$
and $\boldsymbol{W}\left(\boldsymbol{z}\right)=D\boldsymbol{W}\left(0\right)\boldsymbol{z}+\widetilde{\boldsymbol{W}}\left(\boldsymbol{z}\right)$.
Expanding equation (\ref{eq:Uhat-W-reconstruct}) with the decomposed
$\boldsymbol{U}$, $\boldsymbol{W}$ yields
\begin{equation}
\left.\begin{array}{rl}
D\boldsymbol{U}\left(0\right)\left(D\boldsymbol{W}\left(0\right)\boldsymbol{z}+\widetilde{\boldsymbol{W}}\left(\boldsymbol{z}\right)\right)+\widetilde{\boldsymbol{U}}\left(D\boldsymbol{W}\left(0\right)\boldsymbol{z}+\widetilde{\boldsymbol{W}}\left(\boldsymbol{z}\right)\right) & =\boldsymbol{z}\\
\boldsymbol{U}^{\perp}\left(D\boldsymbol{W}\left(0\right)\boldsymbol{z}+\widetilde{\boldsymbol{W}}\left(\boldsymbol{z}\right)\right) & =\boldsymbol{W}_{0}\left(\boldsymbol{z}\right)
\end{array}\right\} .\label{eq:Uhat-W-expand}
\end{equation}
The linear part of equation (\ref{eq:Uhat-W-expand}) is 
\[
\boldsymbol{Q}D\boldsymbol{W}\left(0\right)=\begin{pmatrix}\boldsymbol{I}\\
\boldsymbol{0}
\end{pmatrix},
\]
hence $D\boldsymbol{W}\left(0\right)=\boldsymbol{Q}^{-1}\begin{pmatrix}\boldsymbol{I}\\
\boldsymbol{0}
\end{pmatrix}$. The nonlinear part of (\ref{eq:Uhat-W-expand}) is 
\[
\boldsymbol{Q}\widetilde{\boldsymbol{W}}\left(\boldsymbol{z}\right)=\begin{pmatrix}-\widetilde{\boldsymbol{U}}\left(D\boldsymbol{W}\left(0\right)\boldsymbol{z}+\widetilde{\boldsymbol{W}}\left(\boldsymbol{z}\right)\right)\\
\boldsymbol{W}_{0}\left(\boldsymbol{z}\right)
\end{pmatrix},
\]
which can be solved by the iteration
\begin{equation}
\widetilde{\boldsymbol{W}}_{k+1}\left(\boldsymbol{z}\right)=\boldsymbol{Q}^{-1}\begin{pmatrix}-\widetilde{\boldsymbol{U}}\left(D\boldsymbol{W}\left(0\right)\boldsymbol{z}+\widetilde{\boldsymbol{W}}_{k}\left(\boldsymbol{z}\right)\right)\\
\boldsymbol{W}_{0}\left(\boldsymbol{z}\right)
\end{pmatrix},\;\widetilde{\boldsymbol{W}}_{1}\left(\boldsymbol{z}\right)=\boldsymbol{0},\;k=1,2,\ldots.\label{eq:Wtilde-iteration}
\end{equation}
Iteration (\ref{eq:Wtilde-iteration}) converges for $\left|\boldsymbol{z}\right|$
sufficiently small, due to $\widetilde{\boldsymbol{U}}$$\left(\boldsymbol{x}\right)=\mathcal{O}\left(\left|\boldsymbol{x}\right|^{2}\right)$
and $\boldsymbol{W}_{0}\left(\boldsymbol{z}\right)=\mathcal{O}\left(\left|\boldsymbol{z}\right|^{2}\right)$.
\end{proof}
As we will see in section \ref{sec:Examples}, this approach provides
better results than using an autoencoder. Here we have resolved the
dynamics transversal to the invariant manifold $\mathcal{M}$ up to
linear order. It is essential to resolve this dynamics to find invariance,
not just the location of data points. 
\begin{rem}
\label{rem:localiseROM}More consideration is needed in case $\hat{\boldsymbol{U}}\left(\boldsymbol{x}_{k}\right)$
assumes large values over some data points. This either requires to
replace $\boldsymbol{B}$ with a nonlinear map or we need to filter
out data points that are not in a small neighbourhood of the invariant
manifold $\mathcal{M}$. Due to $\boldsymbol{B}$ being high-dimensional,
replacing it with a nonlinear map leads to numerical difficulties.
Filtering data is easier. For example, we can assign weights to each
term in our optimisation problem (\ref{eq:MAP-U-optim}) depending
on how far a data point is from the predicted manifold, which is the
zero level surface of $\hat{\boldsymbol{U}}$. This can be done using
the optimisation problem
\begin{equation}
\arg\min_{\boldsymbol{B},\boldsymbol{U}^{\perp},\boldsymbol{W}_{0}}\sum_{k=1}^{N}\left\Vert \boldsymbol{x}_{k}\right\Vert ^{-2}\phi_{\kappa}\left(\left\Vert \hat{\boldsymbol{U}}\left(\boldsymbol{x}_{k}\right)\right\Vert ^{2}\right)\left\Vert \boldsymbol{B}\hat{\boldsymbol{U}}\left(\boldsymbol{x}_{k}\right)-\hat{\boldsymbol{U}}\left(\boldsymbol{y}_{k}\right)\right\Vert ^{2},\label{eq:MAP-Uhat-optim}
\end{equation}
where 
\[
\phi_{\kappa}\left(x\right)=\begin{cases}
\exp\frac{x}{x-\kappa^{2}} & 0\le x<\kappa^{2}\\
0 & x\ge\kappa^{2}
\end{cases}
\]
is the bump function and $\kappa>0$ determines the size of the neighbourhood
of the invariant manifold $\mathcal{M}$ that we take into account.
\end{rem}
\begin{rem}
The approximation (\ref{eq:Uhat-invariance}) can be made more accurate
if we allow matrix $\boldsymbol{U}^{\perp}$ to vary with the parameter
of the manifold $\boldsymbol{z}=\boldsymbol{U}\left(\boldsymbol{x}\right)$.
In this case the encoder becomes 
\[
\breve{\boldsymbol{U}}\left(\boldsymbol{x}\right)=\boldsymbol{U}^{\perp}\left(\boldsymbol{U}\left(\boldsymbol{x}\right)\right)\boldsymbol{x}-\boldsymbol{W}_{0}\left(\boldsymbol{U}\left(\boldsymbol{x}\right)\right).
\]
This does increase computational costs, but not nearly as much as
if we were calculating a globally accurate invariant foliation. For
our example problems, we find that this extension is not necessary.
\end{rem}
\begin{rem}
\label{rem:extrasimpleROM}It is also possible to eliminate a-priori
calculation of $\boldsymbol{U}$. We can assume that $\boldsymbol{U}$
is a linear map, such that $\boldsymbol{U}\boldsymbol{U}^{\star}=\boldsymbol{I}$
and treat it as an unknown in representation (\ref{eq:Uhat-decoder}).
The assumption that $\boldsymbol{U}$ is linear makes sense if we
limit ourselves to a small neighbourhood of the invariant manifold
$\mathcal{M}$ by setting $\kappa<\infty$ in (\ref{eq:MAP-Uhat-optim}),
as we have already assumed a linear dynamics among the leaves of the
associated foliation $\hat{\mathcal{F}}$ given by linear map $\boldsymbol{B}$.
Once $\boldsymbol{B}$, $\boldsymbol{U},$ $\boldsymbol{U}^{\perp}$
and $\boldsymbol{W}_{0}$ are found, map $\boldsymbol{S}$ can also
be fitted to the invariance equation (\ref{eq:MAP-U-invariance}).
The equation to fit $\boldsymbol{S}$ to data is
\[
\arg\min_{\boldsymbol{S}}\sum_{k=1}^{N}\left\Vert \boldsymbol{x}_{k}\right\Vert ^{-2}\phi_{\kappa}\left(\left\Vert \hat{\boldsymbol{U}}\left(\boldsymbol{x}_{k}\right)\right\Vert ^{2}\right)\left\Vert \boldsymbol{U}\boldsymbol{y}_{k}-\boldsymbol{S}\left(\boldsymbol{U}\boldsymbol{x}_{k}\right)\right\Vert ^{2},
\]
which is a straightforward linear least squares problem, if $\boldsymbol{S}$
is linear in its parameters. This approach will be further explored
elsewhere.
\end{rem}

\section{\label{sec:freq-damp}Instantaneous frequencies and damping ratios}

Instantaneous damping ratios and frequencies are usually defined with
respect to a model that is fitted to data \cite{JinBrake2020FreqDamp}.
Here we take a similar approach and stress that these quantities only
make sense in an Euclidean frame and not in the nonlinear frame of
an invariant manifold or foliation. The geometry of a manifold or
foliation depends on an arbitrary parametrisation, hence uncorrected
results are not unique. Many studies mistakenly use nonlinear coordinate
systems, for example one by the present author \cite{Szalai20160759}
and colleagues \cite{Breunung2017,PONSIOEN2018269}. Such calculations
are only asymptotically accurate near the equilibrium. Here we describe
how to correct this error.

We assume a two-dimensional invariant manifold $\mathcal{M}$, parametrised
by a decoder $\boldsymbol{W}$ in polar coordinates $r,\theta$. The
invariance equation (\ref{eq:MAP-W-invariance}) for the decoder $\boldsymbol{W}$
can be written as
\begin{equation}
\boldsymbol{W}\left(R\left(r\right),\theta+T\left(r\right)\right)=\boldsymbol{F}\left(\boldsymbol{W}\left(r,\theta\right)\right).\label{eq:Polar-Invariance}
\end{equation}
Without much thinking, (as described in \cite{Szalai20160759,Szalai2020ISF}),
the instantaneous frequency and damping could be calculated as
\begin{align}
\omega\left(r\right) & =T\left(r\right) & \left[\text{rad}/\text{step}\right],\label{eq:Naive-natFreq}\\
\zeta\left(r\right) & =-\frac{\log\frac{R\left(r\right)}{r}}{\omega\left(r\right)} & \left[-\right],\label{eq:Naive-dampRatio}
\end{align}
respectively. The instantaneous amplitude is a norm $A\left(r\right)=\left\Vert \boldsymbol{W}\left(r,\cdot\right)\right\Vert _{A}$,
for example 
\begin{equation}
\left\Vert \boldsymbol{f}\right\Vert _{A}=\sqrt{\left\langle \boldsymbol{f},\boldsymbol{f}\right\rangle _{A}},\;\left\langle \boldsymbol{f},\boldsymbol{f}\right\rangle _{A}=\frac{1}{2\pi}\int_{0}^{2\pi}\left\langle \boldsymbol{f}\left(\theta\right),\boldsymbol{f}\left(\theta\right)\right\rangle _{X}\mathrm{d}\theta,\label{eq:amplitude}
\end{equation}
where $\left\langle \cdot,\cdot\right\rangle _{X}$ is the inner product
on vector space $X$. 

The frequency and damping ratio values are only accurate if there
is a linear relation between $\left\Vert \boldsymbol{W}\left(r,\cdot\right)\right\Vert _{A}$
and $r$, for example
\begin{equation}
\left\Vert \boldsymbol{W}\left(r,\cdot\right)\right\Vert _{A}=r\label{eq:Polar-amplitude}
\end{equation}
and the relative phase between two closed curves satisfies
\begin{equation}
\arg\min_{\gamma}\left\Vert \boldsymbol{W}\left(r_{1},\cdot\right)-\boldsymbol{W}\left(r_{2},\cdot+\gamma\right)\right\Vert _{A}=0.\label{eq:Polar-phase}
\end{equation}
Equation (\ref{eq:Polar-amplitude}) means that the instantaneous
amplitude of the trajectories on manifold $\mathcal{M}$ is the same
as parameter $r$, hence the map $r\mapsto R\left(r\right)$ determines
the change in amplitude. Equation (\ref{eq:Polar-phase}) stipulates
that the parametrisation in the angular variable is such that there
is no phase shift between the closed curves $\boldsymbol{W}\left(r_{1},\cdot\right)$
and $\boldsymbol{W}\left(r_{2},\cdot\right)$ for $r_{1}\neq r_{2}$.
If there would be a phase shift $\gamma$, a trajectory that within
a period moves from amplitude $r_{1}$ to $r_{2}$, would misrepresent
its instantaneous period of vibration by phase $\gamma$, hence the
frequency given by $T\left(r\right)$ would be inaccurate. In fact,
one can set a continuous phase shift $\gamma\left(r\right)$ among
the closed curves $\boldsymbol{W}\left(r,\cdot\right)$, such that
the frequency given by $T\left(r\right)$ has a prescribed value.
The following result provides accurate values for instantaneous frequencies
and damping ratios.
\begin{prop}
\label{prop:Polar-fr-dm}Assume a decoder $\boldsymbol{W}:\left[0,r_{1}\right]\times\left[0,2\pi\right]\to X$
and functions $R,T:\left[0,r_{1}\right]\to\mathbb{R}$ such that they
satisfy invariance equation (\ref{eq:Polar-Invariance}). 
\begin{enumerate}
\item A new parametrisation $\tilde{\boldsymbol{W}}$ of the manifold generated
by $\boldsymbol{W}$ that satisfies the constraints (\ref{eq:Polar-amplitude})
and (\ref{eq:Polar-phase}) is given by
\[
\tilde{\boldsymbol{W}}\left(r,\theta\right)=\boldsymbol{W}\left(t,\theta+\gamma\left(t\right)\right),\;t=\kappa^{-1}\left(r\right),
\]
where 
\begin{align}
\gamma\left(r\right) & =-\int_{0}^{r}\frac{\int_{0}^{2\pi}\left\langle D_{1}\boldsymbol{W}\left(\rho,\theta\right),D_{2}\boldsymbol{W}\left(\rho,\theta\right)\right\rangle _{X}\mathrm{d}\theta}{\int_{0}^{2\pi}\left\langle D_{2}\boldsymbol{W}\left(\rho,\theta\right),D_{2}\boldsymbol{W}\left(\rho,\theta\right)\right\rangle _{X}\mathrm{d}\theta}\mathrm{d}\rho,\label{eq:Polar-delta-prime-1}\\
\kappa\left(r\right) & =\sqrt{\frac{1}{2\pi}\int_{0}^{2\pi}\left\langle \boldsymbol{W}\left(r,\theta\right),\boldsymbol{W}\left(r,\theta\right)\right\rangle _{X}\mathrm{d}\theta}\label{eq:Polar-kappa-prime-1}
\end{align}
and $\left\langle \cdot,\cdot\right\rangle _{X}$ is the inner product
on vector space $X$.
\item The instantaneous natural frequency and damping ratio are calculated
as 
\begin{align}
\omega\left(r\right) & =T\left(t\right)+\gamma\left(t\right)-\gamma\left(R\left(t\right)\right) & \left[\text{rad}/\text{step}\right],\label{eq:Polar-omega-correct}\\
\zeta\left(r\right) & =-\frac{\log\left[r^{-1}\kappa\left(R\left(t\right)\right)\right]}{\omega\left(r\right)} & \left[-\right].\label{eq:Polar-zeta-correct}
\end{align}
where $t=\kappa^{-1}\left(r\right)$.
\end{enumerate}
\end{prop}
\begin{proof}
A proof is given in appendix \ref{sec:proof-FreqDamp}.
\end{proof}
\begin{rem}
The transformed expressions (\ref{eq:Polar-omega-correct}), (\ref{eq:Polar-zeta-correct})
for the instantaneous frequency and damping ratio show that any instantaneous
frequency can be achieved for all $r>0$ if $R\left(r\right)\neq r$
by choosing appropriate functions $\rho,\gamma$. For example, zero
frequency is achieved by solving 
\begin{align}
\gamma\left(r\right) & =\gamma\left(R\left(r\right)\right)-T\left(r\right),\label{eq:Polar-zero-freq}
\end{align}
which is a functional equation. For an $\epsilon>0$, fix $\gamma\left(R\left(\epsilon\right)\right)=0$,
$\gamma\left(\epsilon\right)=T\left(\epsilon\right)$ and some interpolating
values in the interior of the interval $\left[R\left(\epsilon\right),\epsilon\right]$,
then use contraction mapping to arrive at a unique solution for function
$\gamma$.
\end{rem}
\begin{rem}
\label{rem:VF-freq-dm}The same calculation applies to vector fields,
$\dot{\boldsymbol{x}}=\boldsymbol{f}\left(\boldsymbol{x}\right)$,
but the final result is somewhat different. Assume a decoder $\boldsymbol{W}:\left[0,r_{1}\right]\times\left[0,2\pi\right]\to X$
and functions $R,T:\left[0,r_{1}\right]\to\mathbb{R}$ such that they
satisfy the invariance equation 
\begin{equation}
D_{1}\boldsymbol{W}\left(r,\theta\right)R\left(r\right)+D_{2}\boldsymbol{W}\left(r,\theta\right)T\left(r\right)=\boldsymbol{f}\left(\boldsymbol{W}\left(r,\theta\right)\right).\label{eq:Polar-VF-invariance}
\end{equation}
The instantaneous natural frequency and damping ratio is calculated
by 
\begin{align*}
\omega\left(r\right) & =T\left(t\right)-\dot{\gamma}\left(t\right)R\left(t\right) & \left[\text{rad}/\text{unit time}\right],\\
\zeta\left(r\right) & =-\frac{\dot{\kappa}\left(t\right)R\left(t\right)}{r\omega\left(r\right)} & \left[-\right],
\end{align*}
where $t=\kappa^{-1}\left(r\right)$. All other quantities are as
in proposition \ref{prop:Polar-fr-dm}. A proof is given in appendix
\ref{sec:proof-FreqDamp}.
\end{rem}
\begin{rem}
\label{rem:FreqDamp-Scaling}Note that proposition \ref{prop:Polar-fr-dm}
also applies if we create a different measure of amplitude, for example
$\hat{\boldsymbol{W}}\left(r,\theta\right)=\boldsymbol{w}^{\star}\cdot\boldsymbol{W}\left(r,\theta\right)$,
where $\boldsymbol{w}^{\star}$ is a linear map. Indeed, in the proof
we did not use that $\boldsymbol{W}$ is a manifold immersion, it
only served as Euclidean coordinates of points on the invariant manifold.
Hence, the transformed function $\hat{\boldsymbol{W}}$ gives us the
same frequencies and damping ratios as $\boldsymbol{W}$, but at linearly
scaled amplitudes.
\end{rem}
The following example illustrates that if a system is linear in a
nonlinear coordinate system, we can recover the actual instantaneous
damping ratios and frequencies using proposition \ref{prop:Polar-fr-dm}.
This is the case of Koopman eigenfunctions (as in remark \ref{rem:Koopman}),
or normal form transformations where all the nonlinear terms are eliminated.
\begin{example}
Let us consider the linear map
\begin{equation}
\begin{array}{rl}
r_{k+1} & =\mathrm{e}^{\zeta_{0}\omega_{0}}r_{k}\\
\theta_{k+1} & =\theta_{k}+\omega_{0}
\end{array}\label{eq:Polar-Linear-Model}
\end{equation}
and the corresponding nonlinear decoder
\begin{equation}
\boldsymbol{W}\left(r,\theta\right)=\left(\begin{array}{c}
r\cos\theta-\frac{1}{4}r^{3}\cos^{3}\theta\\
r\sin\theta+\frac{1}{2}r^{3}\cos^{3}\theta
\end{array}\right).\label{eq:Polar-Linear-Decoder}
\end{equation}
In terms of the polar invariance equation (\ref{eq:Polar-Invariance})
the linear map (\ref{eq:Polar-Linear-Model}) translates to $T\left(r\right)=\omega_{0}$
and $R\left(r\right)=\mathrm{e}^{\zeta_{0}\omega_{0}}r$. If we disregard
the nonlinearity of $\boldsymbol{W}$, the instantaneous frequency
and the instantaneous damping ratio of our hypothetical system would
be constant, that is $\omega\left(r\right)=\omega_{0}$ and $\zeta\left(r\right)=\zeta_{0}$.
Using proposition \ref{prop:Polar-fr-dm}, we calculate the effect
of $\boldsymbol{W}$, and find that
\begin{align*}
\gamma\left(r\right) & =-\frac{2}{\sqrt{19}}\left(\tan^{-1}\left(\frac{15r^{2}-8}{8\sqrt{19}}\right)+\cot^{-1}\left(\sqrt{19}\right)\right),\\
\kappa\left(r\right) & =r^{2}-\frac{3r^{4}}{16}+\frac{25r^{6}}{256}.
\end{align*}
Finally, we plot expressions (\ref{eq:Polar-omega-correct}) and (\ref{eq:Polar-zeta-correct})
in figure \ref{fig:Freq-Damp-graph}(a) and \ref{fig:Freq-Damp-graph}(b),
respectively. It can be seen that frequencies and damping ratios change
with the vibration amplitude (red lines), but they are constant without
taking the decoder (\ref{eq:Polar-Linear-Decoder}) into account (blue
lines).

\begin{figure}
\begin{centering}
\includegraphics[height=7cm]{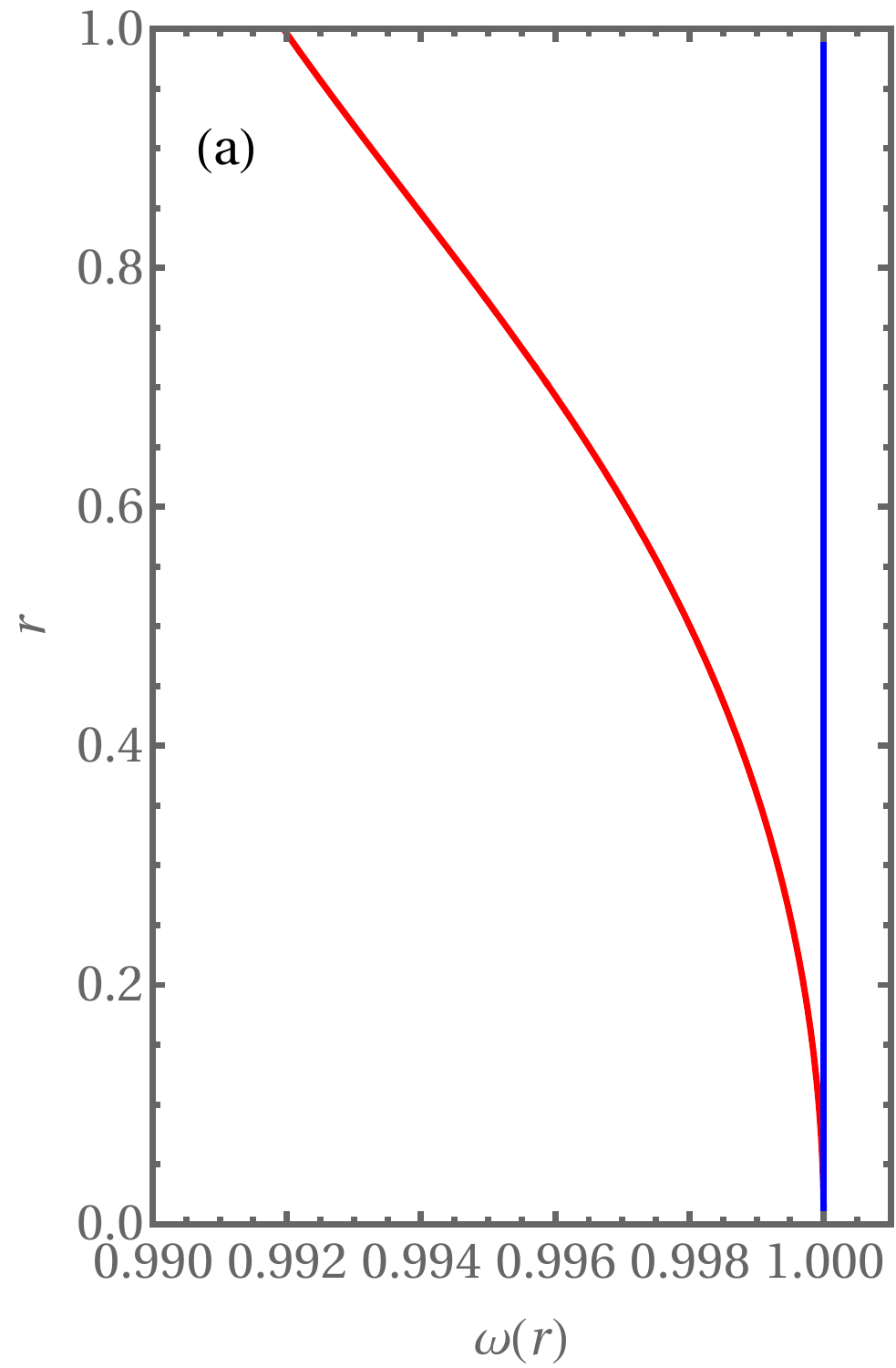}\includegraphics[height=7cm]{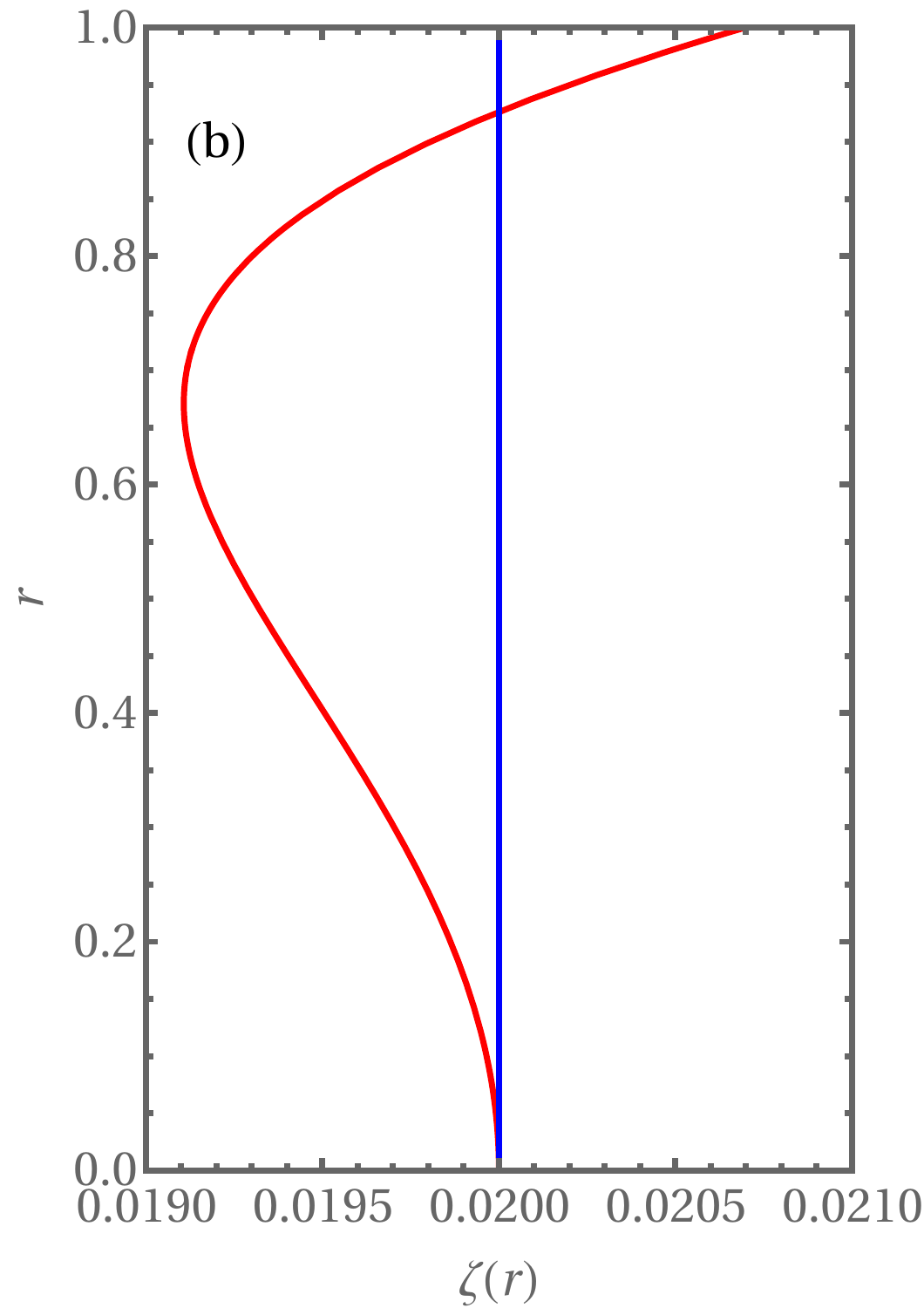}
\par\end{centering}
\caption{\label{fig:Freq-Damp-graph}The instantaneous frequency (a) and instantaneous
damping of system (\ref{eq:Polar-Linear-Model}) together with the
decoder (\ref{eq:Polar-Linear-Decoder}). The blue line represents
the naive calculation just from system (\ref{eq:Polar-Linear-Model})
and the red lines represent the corrected values (\ref{eq:Polar-omega-correct})
and (\ref{eq:Polar-zeta-correct}).}
\end{figure}
The geometry of the re-parametrisation is illustrated in figure \ref{fig:Polar-repair}.
\begin{figure}
\begin{centering}
\includegraphics[height=5cm]{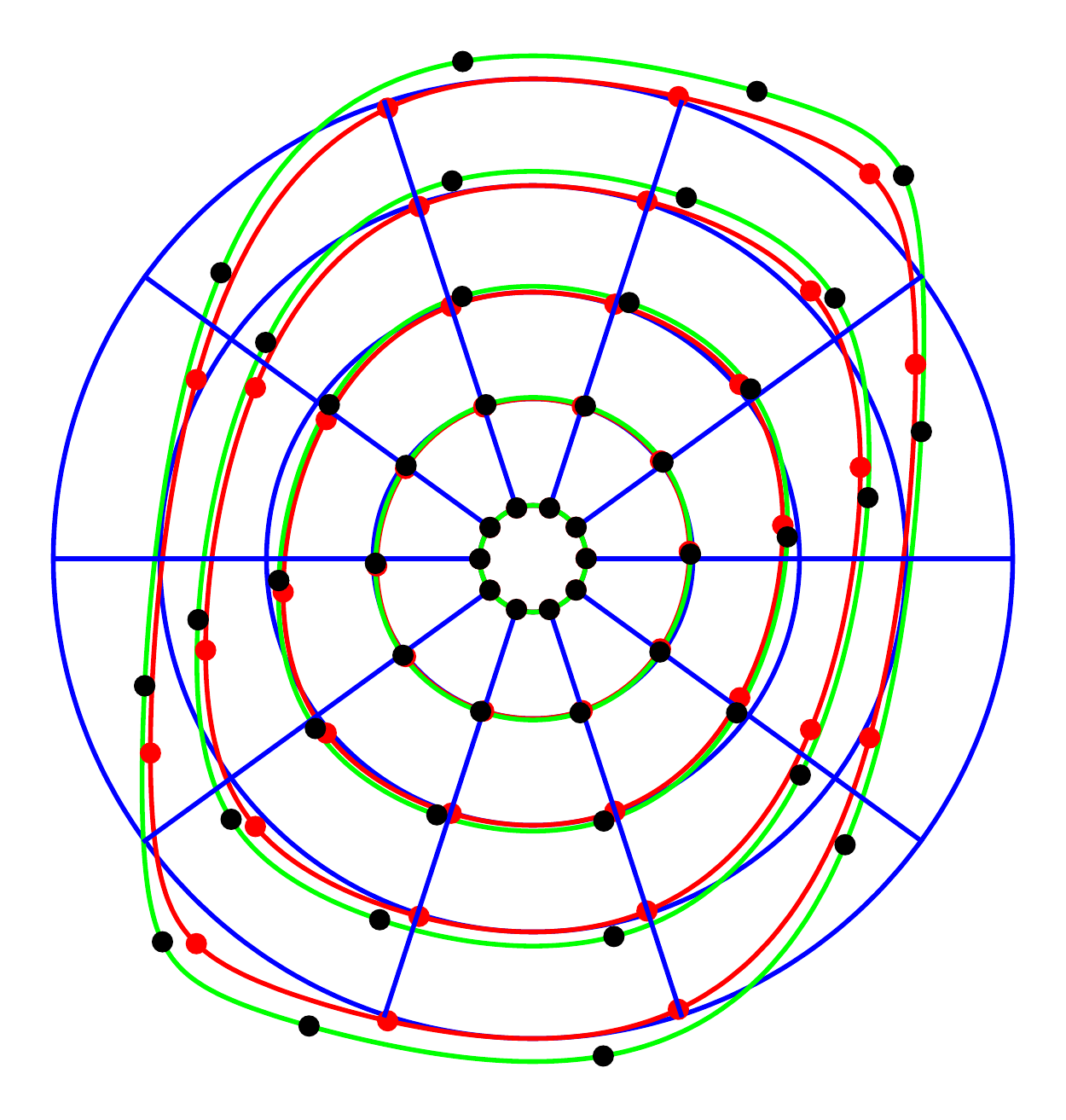}
\par\end{centering}
\caption{\label{fig:Polar-repair}Geometry of the decoder (\ref{eq:Polar-Linear-Decoder}).
The blue grid represents the polar coordinate system with maximum
radius $r=1$. The red curves are the images of the blue concentric
circles under the decoder $\boldsymbol{W}$. The red dots are images
of the intersection points of the blue circles with the blue radial
lines under $\boldsymbol{W}$. The red dots do not fall onto the radial
blue lines, which indicates a phase shift. The green curves correspond
to the images of the blue concentric circles under the re-parametrised
decoder $\tilde{\boldsymbol{W}}$ for the same parameters as the red
curves. The amplitudes of the green curves are now the same as the
amplitude of the corresponding blue curves. Due to the re-parametrisation,
there is no phase shift between the black dots and the blue radial
lines, on average.}
\end{figure}
\end{example}

\section{\label{sec:ROM_id_processes}ROM identification procedures}

Here we describe our methodology of finding invariant foliations,
manifolds and autoencoders. These steps involve methods described
so far and further methods from the appendices. First we start with
finding an invariant foliation together with the invariant manifold
that has the same dynamics as the foliation.
\begin{lyxlist}{00.00.0000}
\item [{F1.}] If the data is not from the full state space, especially
when it is a scalar signal, we use a state-space reconstruction technique
as described in appendix \ref{sec:DelayEmbedding}. At this point
we have data points $\boldsymbol{x}_{k},\boldsymbol{y}_{k}\in X$,
$k=1,\ldots,N$.
\item [{F2.\label{F2}}] To ensure that the solution of (\ref{eq:MAP-U-optim})
converges to the desired foliation, we calculate a linear approximation
of the foliation. We only consider a small neighbourhood of the fixed
point, where the dynamics is nearly linear. Hence, we define the index
set $\mathcal{I}_{\rho}=\left\{ k\in\left\{ 1,\ldots,N\right\} :\left|\boldsymbol{x}_{k}\right|<\rho\right\} $
with $\rho$ sufficiently small, but large enough that it encompasses
enough data for linear parameter estimation. 
\begin{enumerate}
\item First we fit a linear model to the data using a create a least-squares
method \cite{boyd_vandenberghe_2018}. The linear map is assumed to
be $\boldsymbol{y}_{k}=\boldsymbol{A}\boldsymbol{x}_{k}$, where the
coefficient matrix is calculated as 
\begin{align*}
\boldsymbol{K} & =\sum_{k\in\mathcal{I}_{\rho}}\left|\boldsymbol{x}_{k}\right|^{-2}\boldsymbol{x}_{k}\otimes\boldsymbol{x}_{k}^{T},\\
\boldsymbol{L} & =\sum_{k\in\mathcal{I}_{\rho}}\left|\boldsymbol{x}_{k}\right|^{-2}\boldsymbol{y}_{k}\otimes\boldsymbol{x}_{k}^{T},\\
\boldsymbol{A} & =\boldsymbol{L}\boldsymbol{K}^{-1}.
\end{align*}
Matrix $\boldsymbol{A}$ approximates the Jacobian $D\boldsymbol{F}\left(\boldsymbol{0}\right)$,
which then can be used to identify the invariant subspaces $E^{\star}$
and $E$. 
\item The linearised version of foliation invariance (\ref{eq:MAP-U-invariance})
and manifold invariance (\ref{eq:MAP-W-invariance}) are 
\begin{align*}
\boldsymbol{S}_{1}\boldsymbol{U}_{1} & =\boldsymbol{U}_{1}\boldsymbol{A}\;\text{and}\\
\boldsymbol{W}_{1}\boldsymbol{S}_{1} & =\boldsymbol{A}\boldsymbol{W}_{1},
\end{align*}
respectively. We also need to calculate the linearised version of
our locally defined foliation (\ref{eq:Uhat-invariance}), that is
\[
\boldsymbol{B}\boldsymbol{U}_{1}^{\perp}=\boldsymbol{U}_{1}^{\perp}\boldsymbol{A}.
\]
To find the unknown linear maps $\boldsymbol{U}_{1}$, $\boldsymbol{W}_{1}$,
$\boldsymbol{U}_{1}^{\perp}$, $\boldsymbol{S}_{1}$ and\textbf{ $\boldsymbol{B}$}
from\textbf{ $\boldsymbol{A}$}, let us calculate the using real Schur
decomposition of $\boldsymbol{A}$, that is $\boldsymbol{A}\boldsymbol{Q}=\boldsymbol{Q}\boldsymbol{H}$
or $\boldsymbol{Q}^{T}\boldsymbol{A}=\boldsymbol{H}\boldsymbol{Q}^{T}$,
where $\boldsymbol{Q}$ is unitary and $\boldsymbol{H}$ is in an
upper Hessenberg matrix, which is zero below the first subdiagonal.
The Schur decomposition is calculated (or rearranged) such that the
first $\nu$ column vectors, $\boldsymbol{q}_{1},\ldots,\boldsymbol{q}_{\nu}$
of $\boldsymbol{Q}$ span the required right invariant subspace $E$
and correspond to eigenvalues $\mu_{1},\ldots,\mu_{\nu}$. Therefore
we find that $\boldsymbol{W}_{1}=\left[\boldsymbol{q}_{1},\ldots,\boldsymbol{q}_{\nu}\right]$.
In addition, the last $n-\nu$ column vectors of $\boldsymbol{Q}$,
when transposed, define a left-invariant subspace of $\boldsymbol{A}$.
Therefore we define $\boldsymbol{U}_{1}^{\perp}=\left[\boldsymbol{q}_{\nu+1},\ldots,\boldsymbol{q}_{n}\right]^{T}$
and $\boldsymbol{B}_{1}=\boldsymbol{U}_{1}^{\perp}\boldsymbol{A}\boldsymbol{U}_{1}^{\perp T}$,
which provides the initial guesses $\boldsymbol{U}^{\perp}\approx\boldsymbol{U}_{1}^{\perp}$
and $\boldsymbol{B}\approx\boldsymbol{B}_{1}$ in optimisation (\ref{eq:MAP-Uhat-optim}).
In order to find the left-invariant subspace of $\boldsymbol{A}$
corresponding to the selected eigenvalues, we rearrange the Schur
decomposition into $\hat{\boldsymbol{Q}}^{T}\boldsymbol{A}=\hat{\boldsymbol{H}}\hat{\boldsymbol{Q}}^{T}$,
where $\hat{\boldsymbol{Q}}=\left[\hat{\boldsymbol{q}}_{1},\ldots,\hat{\boldsymbol{q}}_{n}\right]$
(using \texttt{ordschur} in Julia or Matlab) such that $\mu_{1},\ldots,\mu_{\nu}$
now appear last in the diagonal of $\hat{\boldsymbol{H}}$ at indices
$\nu+1,\ldots,n$. This allows us to define $\boldsymbol{U}_{1}=\left[\hat{\boldsymbol{q}}_{n-\nu},\ldots,\hat{\boldsymbol{q}}_{n}\right]^{T}$
and $\boldsymbol{S}_{1}=\boldsymbol{U}_{1}\boldsymbol{A}\boldsymbol{U}_{1}^{T}$
which provides the initial guess for $D\boldsymbol{U}\left(\boldsymbol{0}\right)=\boldsymbol{U}_{1}$
and $D\boldsymbol{S}\left(\boldsymbol{0}\right)=\boldsymbol{S}_{1}$. 
\end{enumerate}
\item [{F3.}] We solve the optimisation problem (\ref{eq:MAP-U-optim})
with the initial guess $D\boldsymbol{U}\left(\boldsymbol{0}\right)=\boldsymbol{U}_{1}$
and $D\boldsymbol{S}\left(\boldsymbol{0}\right)=\boldsymbol{S}_{1}$
as calculated in the previous step. We also prescribe the constraints
that $D\boldsymbol{U}\left(\boldsymbol{0}\right)\left(D\boldsymbol{U}\left(\boldsymbol{0}\right)\right)^{T}=\boldsymbol{I}$
and that \textbf{$\boldsymbol{U}\left(\boldsymbol{W}_{1}\boldsymbol{z}\right)$}
is linear in $\boldsymbol{z}$. The numerical representation of the
encoder $\boldsymbol{U}$ is described in appendix \ref{subsec:sparse-poly}.
\item [{F4.\label{F4}}] We perform a normal form transformation on map
$\boldsymbol{S}$ and transform $\boldsymbol{U}$ into the new coordinates.
The normal form $\boldsymbol{S}_{n}$ satisfies the invariance equation
$\boldsymbol{U}_{n}\circ\boldsymbol{S}=\boldsymbol{S}_{n}\circ\boldsymbol{U}_{n}$
(c.f. equation (\ref{eq:MAP-U-invariance})), where $\boldsymbol{U}_{n}:Z\to Z$
is a nonlinear map. We then replace $\boldsymbol{U}$ with $\boldsymbol{U}_{n}\circ\boldsymbol{U}$
and \textbf{$\boldsymbol{S}$} with $\boldsymbol{S}_{n}$ as our ROM.
This step is optional and only required if we want to calculate instantaneous
frequencies and damping ratios. In a two-dimensional coordinate system,
where we have a complex conjugate pair of eigenvalues $\mu_{1}=\overline{\mu}_{2}$,
the real valued normal form is
\begin{equation}
\begin{pmatrix}z_{1}\\
z_{2}
\end{pmatrix}_{k+1}=\begin{pmatrix}z_{1}f_{r}\left(z_{1}^{2}+z_{2}^{2}\right)-z_{2}f_{r}\left(z_{1}^{2}+z_{2}^{2}\right)\\
z_{1}f_{i}\left(z_{1}^{2}+z_{2}^{2}\right)+z_{2}f_{r}\left(z_{1}^{2}+z_{2}^{2}\right)
\end{pmatrix},\label{eq:RealNormalForm}
\end{equation}
which leads to the polar form (\ref{eq:Polar-Invariance}) with $R\left(r\right)=r\sqrt{f_{r}^{2}\left(r^{2}\right)+f_{i}^{2}\left(r^{2}\right)}$
and $T\left(r\right)=\tan^{-1}\frac{f_{i}\left(r^{2}\right)}{f_{r}\left(r^{2}\right)}$.
This normal form calculation is described in \cite{Szalai2020ISF}.
\item [{F5.}] To find the invariant manifold, we calculate a locally defined
foliation (\ref{eq:Uhat-decoder}) and solve the optimisation problem
(\ref{eq:MAP-Uhat-optim}) to find the decoder $\boldsymbol{W}$ of
invariant manifold $\mathcal{M}$. The initial guess in problem (\ref{eq:MAP-Uhat-optim})
is such that $\boldsymbol{U}^{\perp}=\boldsymbol{U}_{1}^{\perp}$
and $\boldsymbol{B}=\boldsymbol{B}_{1}$. We also need to set a $\kappa$
parameter, which is assumed to be $\kappa=0.2$ throughout the paper.
We have found that results are not sensitive to the value of kappa
except for extreme choices, such as $0,\infty$.
\item [{F6.\label{F6}}] In case of an oscillatory dynamics in a two-dimensional
ROM, we recover the actual instantaneous frequencies and damping ratios
using proposition (\ref{prop:Polar-fr-dm}).
\end{lyxlist}
The procedure for the Koopman eigenfunction calculation is the same
as steps F1-F6, except that $\boldsymbol{S}$ is assumed to be linear.
To identify an autoencoder, we use the same setup as in \cite{Cenedese2022NatComm}.
The numerical representation of the autoencoder is described in appendix
\ref{subsec:AENC-repr}. We carry out the following steps
\begin{lyxlist}{00.00.0000}
\item [{AE1.}] We identify $\boldsymbol{W}_{1}$,\textbf{ $\boldsymbol{S}_{1}$}
as in step F2 and set $\boldsymbol{U}=\boldsymbol{W}_{1}^{T}$ and
$D\boldsymbol{S}\left(\boldsymbol{0}\right)=\boldsymbol{S}_{1}$
\item [{AE2.}] Solve the optimisation problem
\[
\arg\min_{\boldsymbol{U},\boldsymbol{W}_{nl}}\sum_{k=1}^{N}\left\Vert \boldsymbol{y}_{k}\right\Vert ^{-2}\left\Vert \boldsymbol{W}\left(\boldsymbol{U}\boldsymbol{y}_{k}\right)-\boldsymbol{y}_{k}\right\Vert ^{2},
\]
which tries to ensure that $\boldsymbol{W}\circ\boldsymbol{U}$ is
the identity, and finally solve
\[
\arg\min_{\boldsymbol{S}}\sum_{k=1}^{N}\left\Vert \boldsymbol{x}_{k}\right\Vert ^{-2}\left\Vert \boldsymbol{W}\left(\boldsymbol{S}\left(\boldsymbol{U}\left(\boldsymbol{x}_{k}\right)\right)\right)-\boldsymbol{y}_{k}\right\Vert ^{2}.
\]
\item [{AE3.}] Perform a normal form transformation on $\boldsymbol{S}$
by seeking the simplest $\boldsymbol{S}_{n}$ that satisfies $\boldsymbol{W}_{n}\circ\boldsymbol{S}_{n}=\boldsymbol{S}\circ\boldsymbol{W}_{n}$.
This is similar to step F4, except that the normal form is in the
style of the invariance equation of a manifold (\ref{eq:MAP-W-invariance}).
\item [{AE4.}] Same as step F6, but applied to nonlinear map $\boldsymbol{S}_{n}$
and decoder $\boldsymbol{W}\circ\boldsymbol{W}_{n}$.
\end{lyxlist}

\section{\label{sec:Examples}Examples}

We are considering three examples. The first is a caricature model
to illustrate all the techniques discussed in this paper and why certain
techniques fail. The second example is a series of synthetic data
sets with higher dimensionality, to illustrate the methods in more
detail using a polynomial representations of the encoder with HT tensor
coefficients (see appendix \ref{subsec:sparse-poly}). This example
also illustrates two different methods to reconstruct the state space
of the system from a scalar measurement. The final example is a physical
experiment of a jointed beam, where only a scalar signal is recorded
and we need to reconstruct the state space with our previously tested
technique.

\subsection{\label{sec:Caricature}A caricature model}

To illustrate the performance of autoencoders, invariant foliations
and locally defined invariant foliations, we construct a simple two-dimensional
map $\boldsymbol{F}$ with a node-type fixed point using the expression
\begin{equation}
\boldsymbol{F}\left(\boldsymbol{x}\right)=\boldsymbol{V}\left(\boldsymbol{A}\boldsymbol{V}^{-1}\left(\boldsymbol{x}\right)\right),\label{eq:caricature-mod}
\end{equation}
where 
\[
\boldsymbol{A}=\begin{pmatrix}\frac{9}{10} & 0\\
0 & \frac{4}{5}
\end{pmatrix},
\]
the near-identity coordinate transformation is

\begin{align*}
\boldsymbol{V}\left(\boldsymbol{x}\right) & =\begin{pmatrix}\begin{array}{l}
x_{1}+\frac{1}{4}\left(x_{1}^{3}-3\left(x_{1}-1\right)x_{2}x_{1}+2x_{2}^{3}+\left(5x_{1}-2\right)x_{2}^{2}\right)\\
x_{2}+\frac{1}{4}\left(2x_{2}^{3}+\left(2x_{1}-1\right)x_{2}^{2}-x_{1}^{2}\left(x_{1}+2\right)\right)
\end{array}\end{pmatrix},
\end{align*}
and the state vector is defined as $\boldsymbol{x}=\left(x_{1},x_{2}\right)$.
In a neighbourhood of the origin, transformation $\boldsymbol{V}$
has a unique inverse, which we calculate numerically. Map $\boldsymbol{F}$
is constructed such that we can immediately identify the smoothest
(hence unique) invariant manifolds corresponding to the two eigenvalues
of $D\boldsymbol{F}\left(0\right)$ as
\begin{align*}
\mathcal{M}_{\frac{9}{10}} & =\left\{ \boldsymbol{V}\left(z,0\right),z\in\mathbb{R}\right\} ,\\
\mathcal{M}_{\frac{4}{5}} & =\left\{ \boldsymbol{V}\left(0,z\right),z\in\mathbb{R}\right\} .
\end{align*}
We can also calculate the leaves of the two invariant foliations as
\begin{align}
\mathcal{L}_{\frac{9}{10},z} & =\left\{ \boldsymbol{V}\left(z,\overline{z}\right),\overline{z}\in\mathbb{R}\right\} ,\label{eq:caricature-foil-vert}\\
\mathcal{L}_{\frac{4}{5},z} & =\left\{ \boldsymbol{V}\left(\overline{z},z\right),\overline{z}\in\mathbb{R}\right\} .\label{eq:caricature-foil-horiz}
\end{align}
To test the methods, we created 500 times 30 points long trajectories
with initial conditions sampled from a uniform distribution over the
rectangle $\left[-\frac{4}{5},\frac{4}{5}\right]\times\left[-\frac{1}{4},-\frac{1}{4}\right]$
to fit our ROM to.

We first attempt to fit an autoencoder to the data. We assume that
the encoder, decoder and the nonlinear map $\boldsymbol{S}$ are 
\begin{align}
\boldsymbol{U}\left(x_{1},x_{2}\right) & =x_{1},\label{eq:exa-AU-enc}\\
\boldsymbol{W}\left(z\right) & =\left(z,h\left(z\right)\right)^{T},\label{eq:exa-AU-dec}\\
\boldsymbol{S}\left(z\right) & =\lambda\left(z\right),\label{eq:exa-AU-map}
\end{align}
where $\lambda,h$ are polynomials of order-5. Our expressions already
contain the invariant subspace $E=\mathrm{span}\left(1,0\right)^{T}$,
which should make the fitting easier. Finally, we solve the optimisation
problem (\ref{eq:MAP-AE-optim}). The result of the fitting can be
seen in figure \ref{fig:caricature-sim}(a) as depicted by the red
curve. The fitted curve is more dependent on the distribution of data
than the actual position of the invariant manifold, which is represented
by the blue dashed line in figure \ref{fig:caricature-sim}(a). Various
other expressions for $\boldsymbol{U}$ and \textbf{$\boldsymbol{W}$}
were also tried that do not assume the direction of the invariant
subspace $E$ with similar results.

To calculate the invariant foliation in the horizontal direction,
we assume that
\begin{equation}
\left.\begin{array}{rl}
\boldsymbol{U}\left(x_{1},x_{2}\right) & =x_{1}+u\left(x_{1},x_{2}\right)\\
\boldsymbol{S}\left(z\right) & =\lambda z
\end{array}\right\} ,\label{eq:exa-FOL-enc}
\end{equation}
where $u$ is an order-5 polynomial which lacks the constant and linear
terms. The exact expression of $\boldsymbol{U}$ is not a polynomial,
because it is the second coordinate of the inverse of function $\boldsymbol{V}$.
The fitting is carried out by solving the optimisation problem (\ref{eq:MAP-U-optim}).
The result can be seen in figure \ref{fig:caricature-sim}(b), where
the red curves are contour plots of the identified encoder $\boldsymbol{U}$
and the dashed blue lines are the leaves as defined by equation (\ref{eq:caricature-foil-vert}).
Figure \ref{fig:caricature-sim}(c) is produced in the same way as
\ref{fig:caricature-sim}(b), except that the encoder is defined as
$\boldsymbol{U}\left(x_{1},x_{2}\right)=x_{2}+h\left(x_{1},x_{2}\right)$
and the blue lines are the leaves given by (\ref{eq:caricature-foil-horiz}).

As we have discussed in section \ref{subsec:LocallyAccurateEncoder},
a locally defined encoder can also be constructed from a decoder.
In the expression of the encoder (\ref{eq:Uhat-decoder}) we take
\[
\boldsymbol{W}_{0}\left(z\right)=h\left(z\right)
\]
and $\boldsymbol{U}^{\perp}=\left(0,1\right)$, where $h$ is an order-9
polynomial without constant and linear terms. The expressions for
$\boldsymbol{U}$ and \textbf{$\boldsymbol{S}$} were already found
as (\ref{eq:exa-FOL-enc}), hence our approximate encoder becomes
\[
\hat{\boldsymbol{U}}\left(\boldsymbol{x}\right)=x_{2}-h\left(x_{1}+u\left(x_{1},x_{2}\right)\right).
\]
We solve the optimisation problem (\ref{eq:MAP-Uhat-optim}) with
$\kappa=0.13$. We do not reconstruct the decoder $\boldsymbol{W}$,
as it is straightforward to plot the level surfaces of $\hat{\boldsymbol{U}}$
directly. The result can be seen in figure \ref{fig:caricature-sim}(d),
where the green line is the approximate invariant manifold (the zero
level surface of $\hat{\boldsymbol{U}}$) and the red lines are other
level surfaces of $\hat{\boldsymbol{U}}$.

In conclusion, this simple example shows that only invariant foliations
can be fitted to data, autoencoders give spurious results.

\begin{figure}
\begin{centering}
\includegraphics[width=0.35\linewidth]{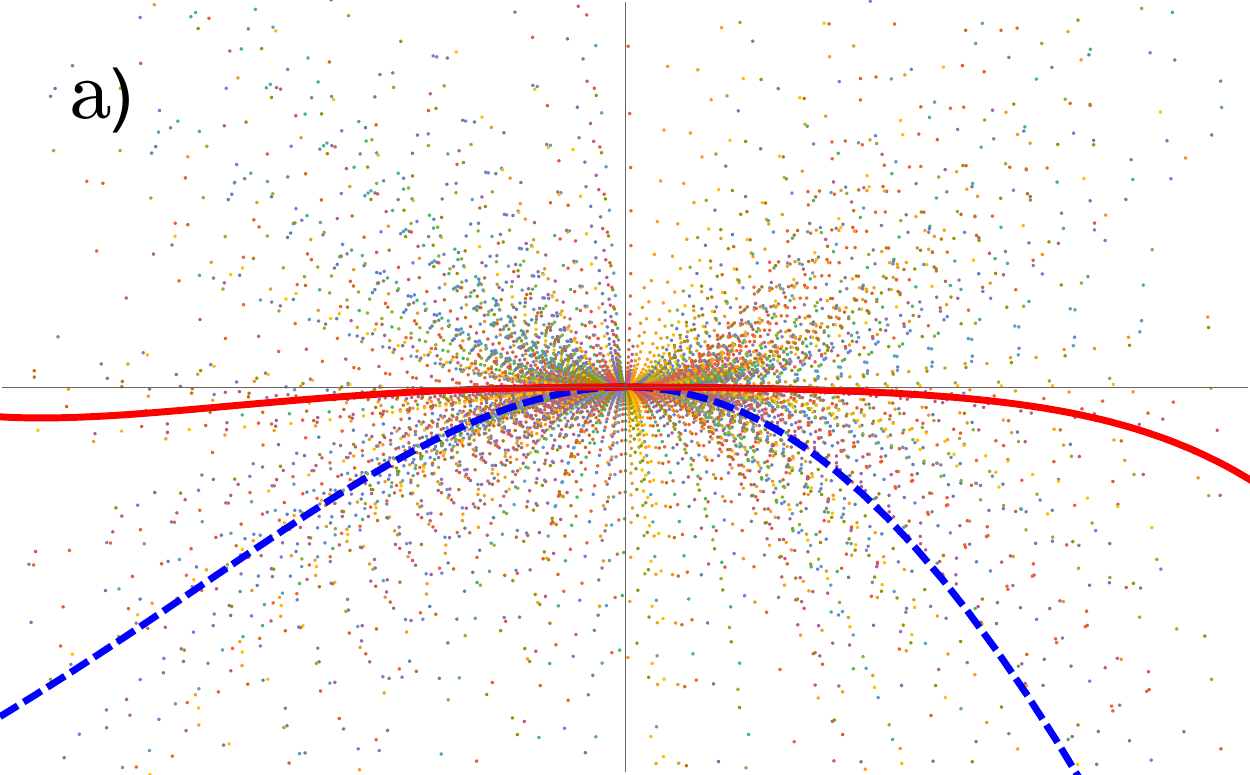}\includegraphics[width=0.35\linewidth]{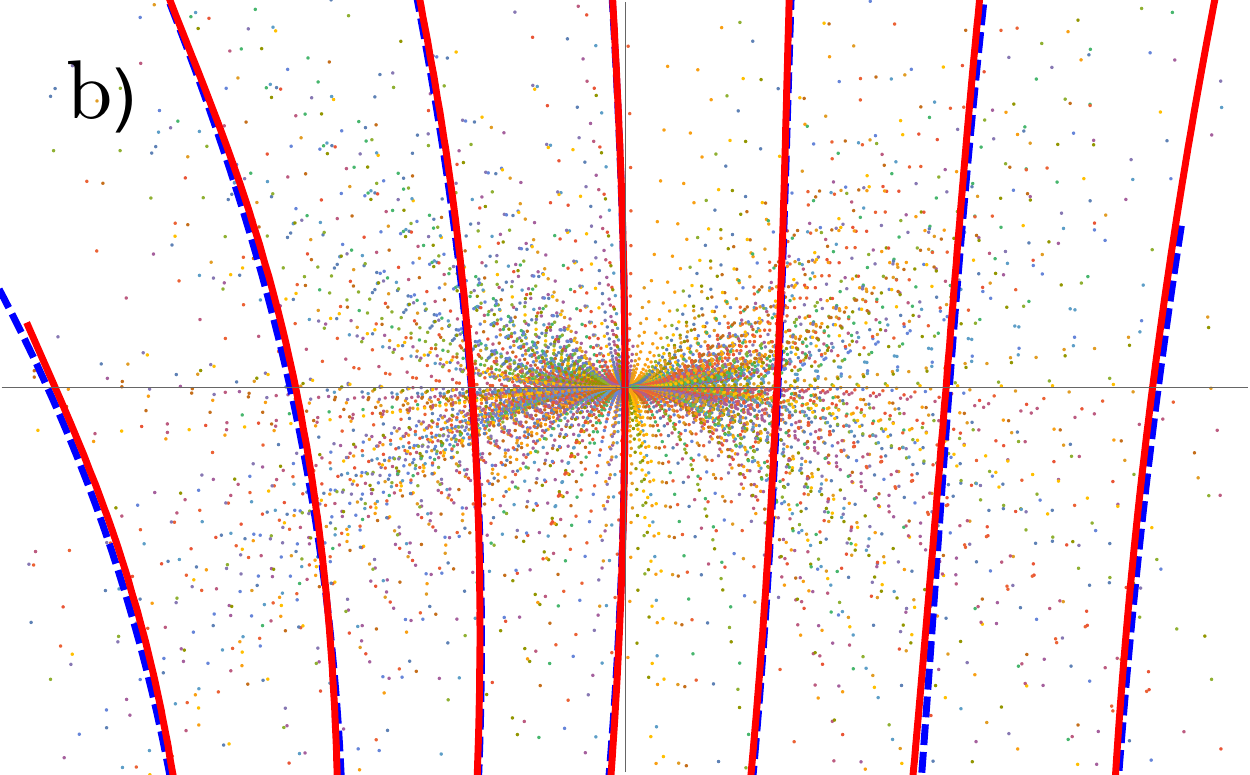}\\
\includegraphics[width=0.35\linewidth]{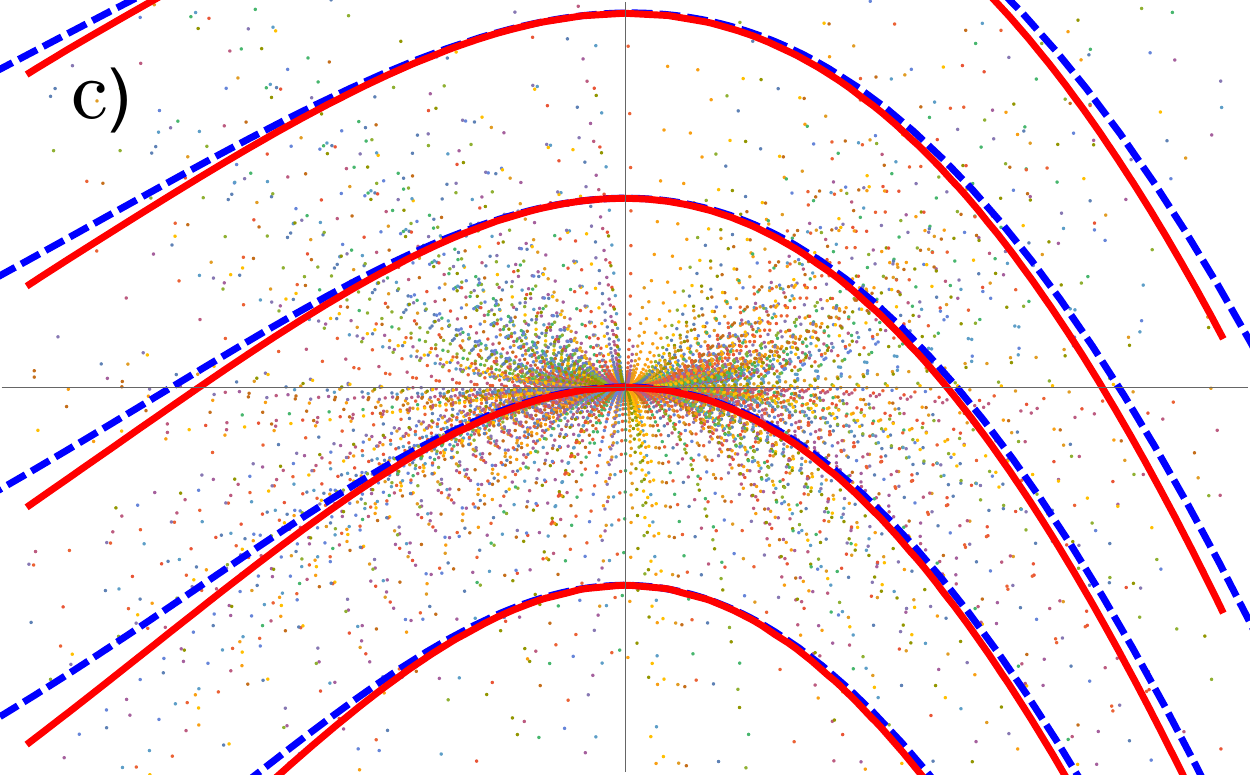}\includegraphics[width=0.35\linewidth]{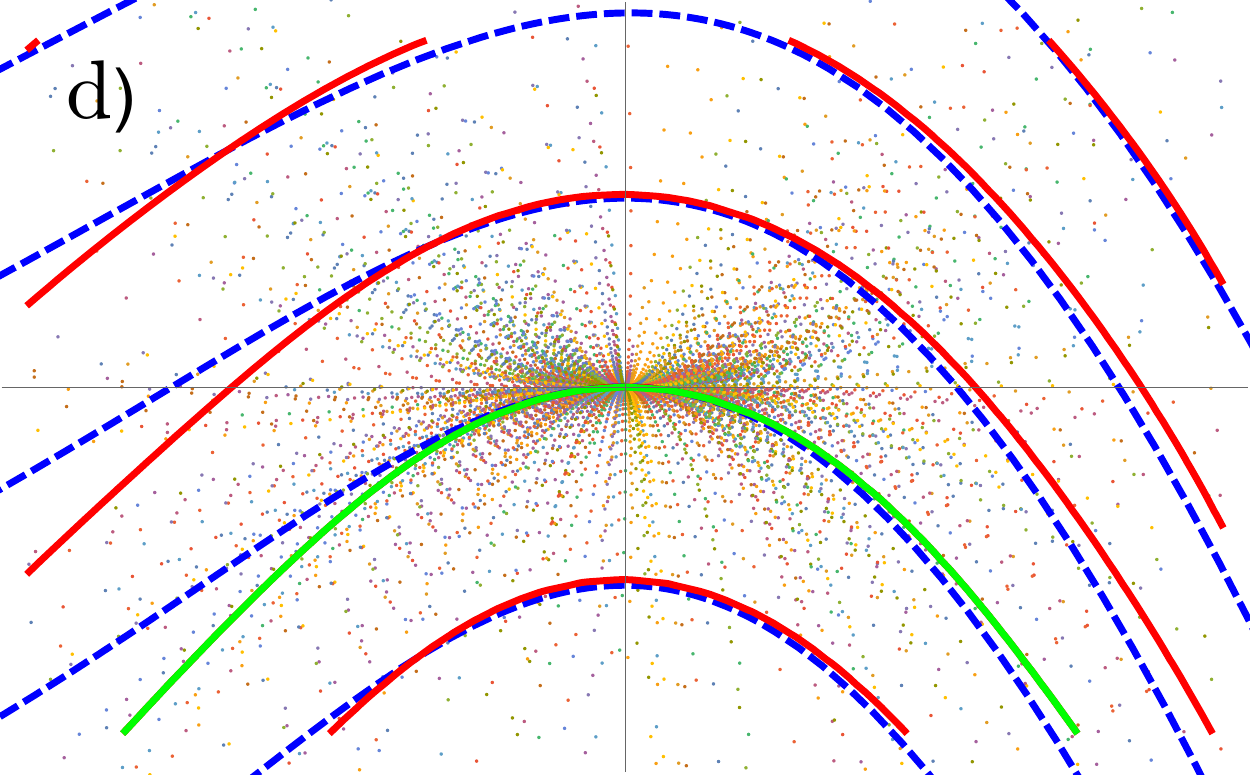}
\par\end{centering}
\caption{\label{fig:caricature-sim}Identifying invariant objects in equation
(\ref{eq:caricature-mod}). The data contains 500 trajectories of
length 30 with initial conditions picked from a uniform distribution;
a) fitting an autoencoder (red continuous curve) does not reproduce
the invariant manifold (blue dashed curve), instead it follows the
distribution of data; (b) invariant foliation in the horizontal direction;
c) invariant foliation in the vertical direction; d) an invariant
manifold is calculated as the leaf of a locally defined invariant
foliation. The green line is the invariant manifold. Each diagram
represent the box $\left[-\frac{4}{5},\frac{4}{5}\right]\times\left[-\frac{1}{4},-\frac{1}{4}\right]$,
the axes labels are intentionally hidden.}
\end{figure}

\subsection{\label{subsec:10dimsys-example}A ten-dimensional system}

To create a numerically challenging example, we construct a ten-dimensional
differential equation from five decoupled second-order nonlinear oscillators
using two successive coordinate transformations. The system of decoupled
oscillators is denoted by $\dot{\boldsymbol{x}}=\boldsymbol{f}_{0}\left(\boldsymbol{x}\right)$,
where the state variable is in the form of 
\[
\boldsymbol{x}=\left(r_{1},\ldots,r_{5},\theta_{1},\ldots,\theta_{5}\right)
\]
and the dynamics is given by 
\begin{equation}
\begin{array}{ll}
\dot{r}_{1}=-\frac{1}{500}r_{1}+\frac{1}{100}r_{1}^{3}-\frac{1}{10}r_{1}^{5}, & \dot{\theta}_{1}=1+\frac{1}{4}r_{1}^{2}-\frac{3}{10}r_{1}^{4},\\
\dot{r}_{2}=-\frac{\mathrm{e}}{500}r_{2}-\frac{1}{10}r_{2}^{5}, & \dot{\theta}_{2}=\mathrm{e}+\frac{3}{20}r_{2}^{2}-\frac{1}{5}r_{2}^{4},\\
\dot{r}_{3}=-\frac{1}{50}\sqrt{\frac{3}{10}}r_{3}+\frac{1}{100}r_{3}^{3}-\frac{1}{10}r_{3}^{5},\quad & \dot{\theta}_{3}=\sqrt{30}+\frac{9}{50}r_{3}^{2}-\frac{19}{100}r_{3}^{4},\\
\dot{r}_{4}=-\frac{1}{500}\pi^{2}r_{4}+\frac{1}{100}r_{4}^{3}-\frac{1}{10}r_{4}^{5}, & \dot{\theta}_{4}=\pi^{2}+\frac{4}{25}r_{4}^{2}-\frac{17}{100}r_{4}^{4},\\
\dot{r}_{5}=-\frac{13}{500}r_{5}+\frac{1}{100}r_{5}^{3}, & \dot{\theta}_{5}=13+\frac{4}{25}r_{5}^{2}-\frac{9}{50}r_{5}^{4}.
\end{array}\label{eq:10dim-model}
\end{equation}
The first transformation brings the polar form of equation (\ref{eq:10dim-model})
into Cartesian coordinates using the transformation $\boldsymbol{y}=\boldsymbol{g}\left(\boldsymbol{x}\right)$,
which is defined by $y_{2k-1}=r_{k}\cos\theta_{k}$ and $y_{2k}=r_{k}\sin\theta_{k}$.
Finally, we couple all variables using the second nonlinear transformation
$\boldsymbol{y}=\boldsymbol{h}\left(\boldsymbol{z}\right)$, which
reads 
\begin{equation}
\begin{array}{rlrl}
y_{1} & =z_{1}+z_{3}-\frac{1}{12}z_{3}z_{5}, & y_{2} & =z_{2}-z_{3},\\
y_{3} & =z_{3}+z_{5}-\frac{1}{12}z_{5}z_{7}, & y_{4} & =z_{4}-z_{5},\\
y_{5} & =z_{5}+z_{7}+\frac{1}{12}z_{7}z_{9}, & y_{6} & =z_{6}-z_{7},\\
y_{7} & =z_{7}+z_{9}-\frac{1}{12}z_{1}z_{9}, & y_{8} & =z_{8}-z_{9},\\
y_{9} & =z_{9}+z_{1}-\frac{1}{12}z_{3}z_{1},\quad & y_{10} & =z_{10}-z_{1},
\end{array}\label{eq:10dim-transform}
\end{equation}
and where $\boldsymbol{y}=\left(y_{1},\ldots,y_{10}\right)$ and $\boldsymbol{z}=\left(z_{1},\ldots,z_{10}\right)$.
The two transformations give us the differential equation $\dot{\boldsymbol{z}}=\boldsymbol{f}\left(\boldsymbol{z}\right)$,
where
\begin{equation}
\boldsymbol{f}\left(\boldsymbol{z}\right)=\left[D\boldsymbol{g}^{-1}\left(\boldsymbol{h}\left(\boldsymbol{z}\right)\right)D\boldsymbol{h}\left(\boldsymbol{z}\right)\right]^{-1}\boldsymbol{f}_{0}\left(\boldsymbol{g}^{-1}\left(\boldsymbol{h}\left(\boldsymbol{z}\right)\right)\right).\label{eq:10dim-finmod}
\end{equation}
The natural frequencies of our system at the origin are
\[
\omega_{1}=1,\omega_{2}=\mathrm{e},\omega_{3}=\sqrt{30},\omega_{4}=\pi^{2},\omega_{5}=13
\]
and the damping ratios are the same $\zeta_{1}=\cdots=\zeta_{5}=1/500$.

We select the first natural frequency to test various methods. We
also test the methods on three types of data. Firstly, full state
space information is used, secondly the state space is reconstructed
from the signal $\xi_{k}=\frac{1}{10}\sum_{j=1}^{10}z_{k,j}$ using
principal component analysis (PCA) as described in appendix \ref{subsec:PCA}
with 16 PCA components, finally the state space is reconstructed from
$\xi_{k}$ using a discrete Fourier transform (DFT) as described in
appendix \ref{subsec:DFT}. When data is recorded in state space form,
$1000$ trajectories $16$ points long each with time step $\Delta T=0.1$
were created by numerically solving (\ref{eq:10dim-finmod}). Initial
conditions were sampled from unit balls of radius $0.8$, $1.0$,
$1.2$ and $1.4$ about the origin. The Euclidean norm of the initial
conditions were uniformly distributed. The four data sets are labelled
ST-1, ST-2, ST-3, ST-4 in the diagrams. For state space reconstruction,
100 trajectories, 3000 points each, with time step $\Delta T=0.01$
were created by numerically solving (\ref{eq:10dim-finmod}). The
initial conditions for this data was similarly sampled from unit balls
of radius $0.8$, $1.0$, $1.2$ and $1.4$ about the origin, such
that the Euclidean norm of the initial conditions are uniformly distributed.
The PCA reconstructed data are labelled PCA-1, PCA-2, PCA-3, PCA-4
and the DFT reconstructed data are labelled DFT-1, DFT-2, DFT-3, DFT-4.

The amplitude for each ROM is calculated as $A\left(r\right)=\sqrt{\frac{1}{2\pi}\int\left(\boldsymbol{w}^{\star}\cdot\boldsymbol{W}\left(r,\theta\right)\right)^{2}\mathrm{d}\theta}$,
where $\boldsymbol{w}^{\star}=\frac{1}{10}\left(1,1,\ldots,1\right)$
for the state-space data and $\boldsymbol{w}^{\star}$ is calculated
in appendices \ref{subsec:PCA}, \ref{subsec:DFT} when state-space
reconstruction is used. We can also attach an amplitude to each data
point $\boldsymbol{x}_{k}$ through the encoder and the decoder. If
the ROM assumes the normal form (\ref{eq:RealNormalForm}), the radial
parameter $r$ is simply calculated as $r_{k}=\left\Vert \boldsymbol{U}\left(\boldsymbol{x}_{k}\right)\right\Vert $,
hence the amplitude is $A\left(r_{k}\right)$. 

\begin{figure}
\begin{centering}
\includegraphics[width=0.99\linewidth]{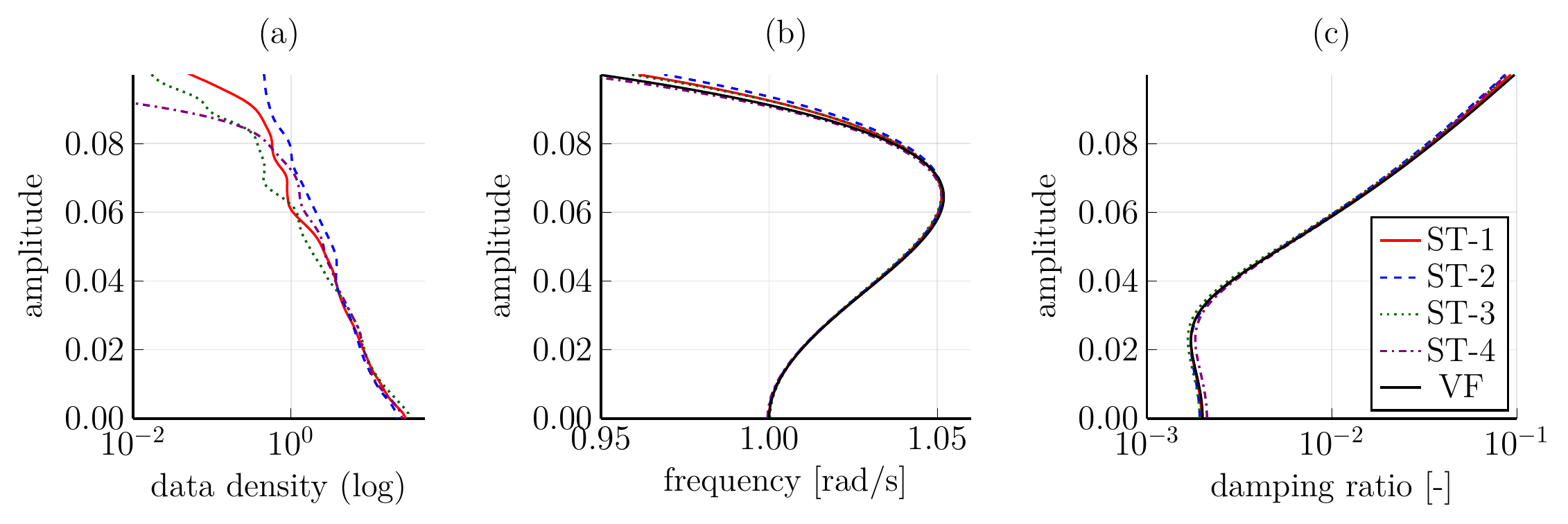}\\
\includegraphics[width=0.99\linewidth]{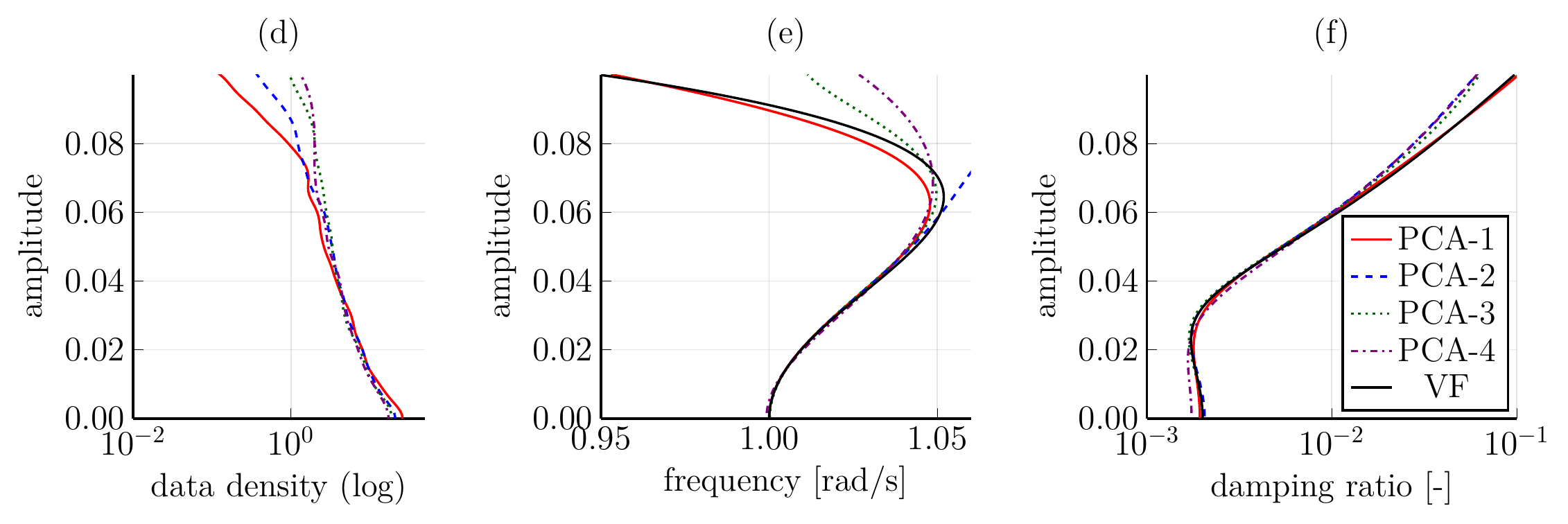}\\
\includegraphics[width=0.99\linewidth]{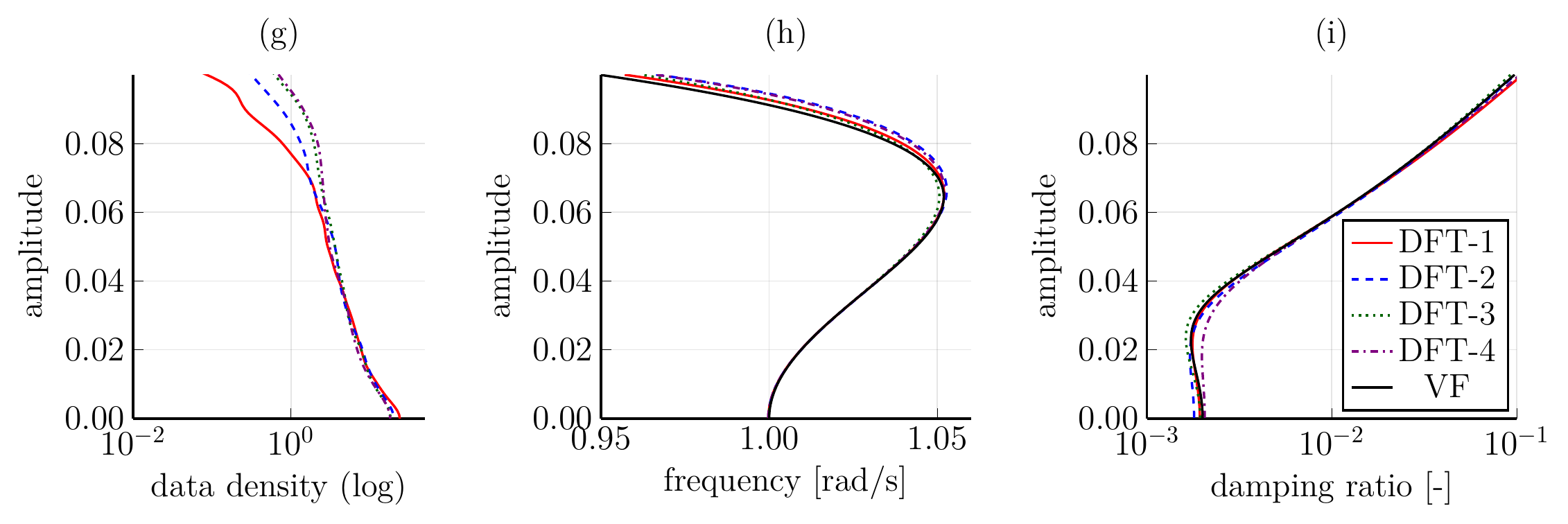}
\par\end{centering}
\caption{\label{fig:10dim-foliations}Instantaneous frequencies and damping
ratios of differential equation (\ref{eq:10dim-finmod}) using invariant
foliations. The first column shows the data density with respect to
the amplitude $A\left(r\right)$, the second column is the instantaneous
frequency and the third column is the instantaneous damping. The first
row corresponds to state space data, the second row shows PCA reconstructed
data and the third row is DFT reconstructed data. The data density
is controlled by sampling initial conditions from different sized
neighbourhoods of the origin.}
\end{figure}
\begin{figure}
\begin{centering}
\includegraphics[width=1\linewidth]{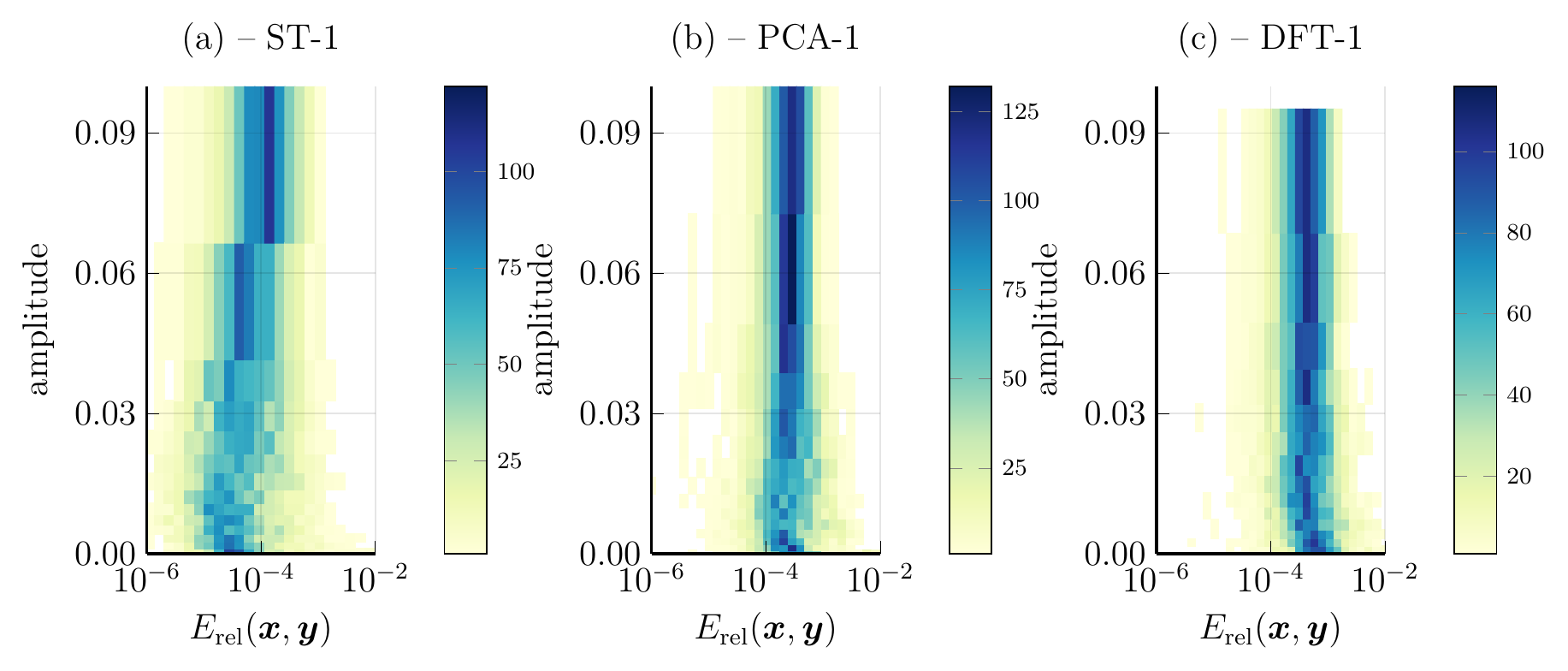}
\par\end{centering}
\caption{\label{fig:10dim-accuracy}Histograms of the fitting error (\ref{eq:FittingError})
for the data sets ST-1, PCA-1 and DFT-1 as a function of vibration
amplitude. The distribution appears uniform with respect to the amplitude.}

\end{figure}
Figure \ref{fig:10dim-foliations} shows the result of our calculation
for the three types of data. In the first column data density is displayed
with respect to amplitude $A\left(r\right)$ in the ROM. Lower amplitudes
have higher densities, because trajectories exponentially converge
to the origin. In figure \ref{fig:10dim-foliations} we also display
the identified instantaneous frequencies and damping ratios. The results
are then compared to the analytically calculated frequencies and damping
ratios labelled by VF.

State space data gives the closest match to the analytical reference
(labelled as VF). We find that the PCA method cannot embed the data
in a 10-dimesional space, only an 16-dimensional embedding is acceptable,
but still inaccurate. Using a perfect reproducing filter bank (DFT)
yields better results, probably because the original signal can be
fully reconstructed and we expect a correct state-space reconstruction
at small amplitudes. Indeed, the PCA results diverge at higher amplitudes,
where the state space reconstruction is no longer valid. The author
has also tried non-optimal delay embedding, with inferior results.
None of the techniques had any problem with the less the challenging
Shaw-Pierre example \cite{ShawPierre,Szalai2020ISF} (data not shown).

Figure \ref{fig:10dim-accuracy} shows the accuracy of the data fit
as a function of the amplitude $A\left(r\right)$. The relative error
displayed is defined by 
\begin{equation}
E_{rel}\left(\boldsymbol{x},\boldsymbol{y}\right)=\frac{\left\Vert \boldsymbol{S}\left(\boldsymbol{U}\left(\boldsymbol{x}\right)\right)-\boldsymbol{U}\left(\boldsymbol{y}\right)\right\Vert }{\left\Vert \boldsymbol{U}\left(\boldsymbol{x}\right)\right\Vert }.\label{eq:FittingError}
\end{equation}
It turns out that the error is roughly independent of the amplitude,
except for data set ST-1, which has lower errors at low amplitudes.
The accuracy is the highest for state space data, while the accuracy
of the DFT reconstructed data is slightly worse than for the PCA reconstructed
data. In contrast, the comparison with the analytically calculated
result is worse for the PCA data than for the DFT data. The reason
is that PCA reconstruction cannot exactly reproduce the original signal
from the identified components while the DFT method can.

\begin{figure}
\begin{centering}
\includegraphics[width=0.99\columnwidth]{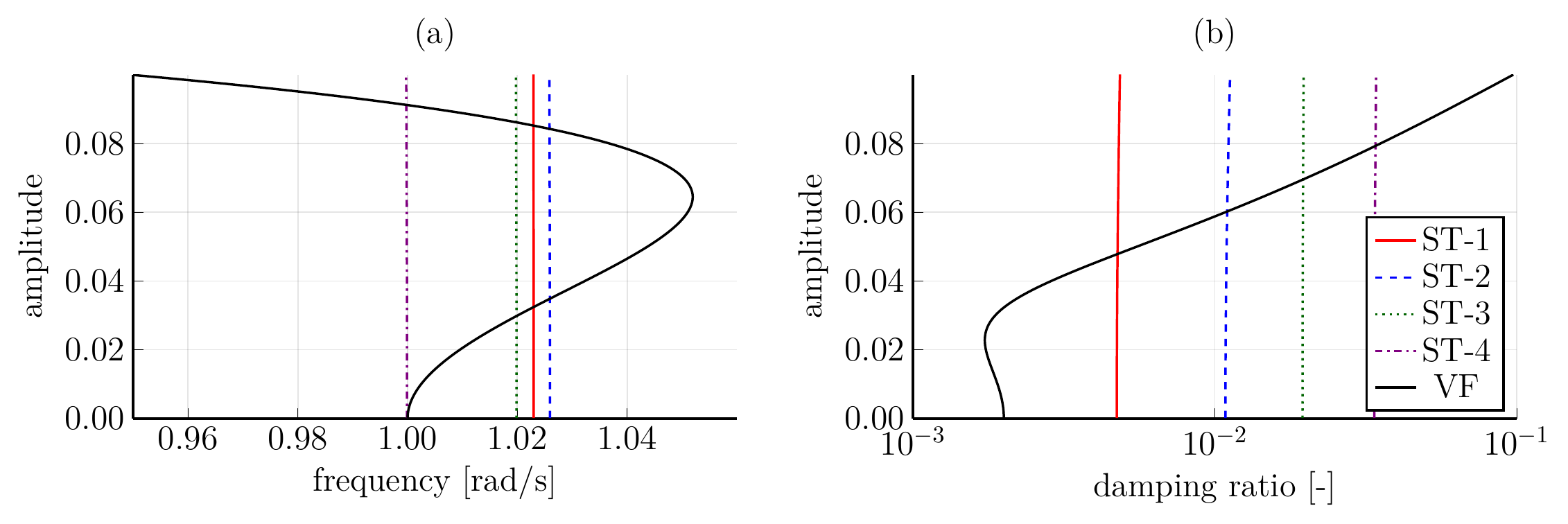}
\par\end{centering}
\caption{\label{fig:10dim-Koopman}ROM by Koopman eigenfunctions. The same
quantities are calculated as in figure \ref{fig:10dim-foliations},
except that map $\boldsymbol{S}$ is assumed to be linear. There is
some variation in the frequencies and damping ratios with respect
to the amplitude due to the corrections in section \ref{sec:freq-damp},
but accurate values could not be recovered as the linear approximation
does not account for near internal resonances, as in formula (\ref{eq:Internal-nonresonance}).}
\end{figure}
When restricting map $\boldsymbol{S}$ to be linear, we are identifying
Koopman eigenfunctions. Despite that linear dynamics is identified
we should be able to reproduce the nonlinearities as illustrated in
section \ref{sec:freq-damp}. However, we also have near internal
resonances as per equation (\ref{eq:Internal-nonresonance}), which
make certain terms of encoder $\boldsymbol{U}$ large, which are difficult
to find by optimisation. The result can be seen in figure \ref{fig:10dim-Koopman}.
The identified frequencies and damping ratios show little variation
with amplitude and mostly capture the average of the reference values.
Fitting the Koopman eigenfunction achieves maximum and average values
of $E_{rel}$ at $8.66\%$ and $0.189\%$ over data set ST-1, respectively.
Better accuracy could be achieved using higher rank HT tensor coefficients
in the encoder, which would significantly increase the number of model
parameters. In contrast, fitting the invariant foliation to the same
data set yields maximum and the average values of $E_{rel}$ at $2.21\%$
and $0.0118\%$, respectively (also illustrated in figure \ref{fig:10dim-accuracy}(a)).
This better accuracy is achieved with a small number of extra parameters
that make the two-dimensional map $\boldsymbol{S}$ nonlinear.

\begin{figure}
\begin{centering}
\includegraphics[width=0.99\linewidth]{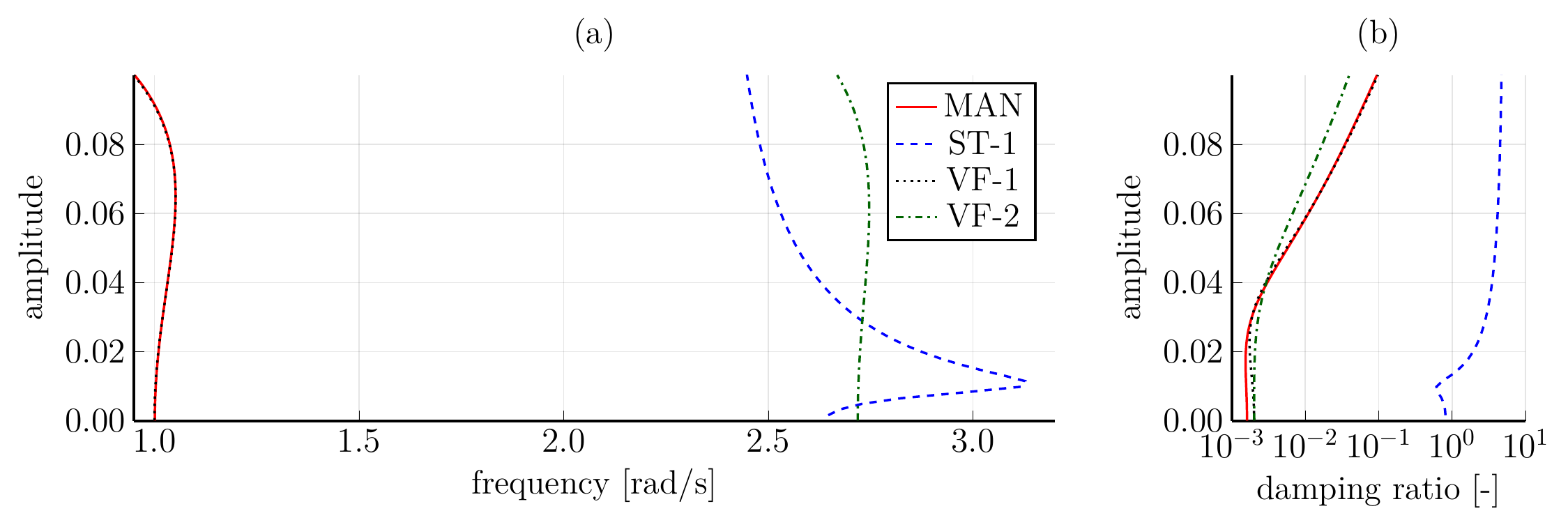}
\par\end{centering}
\caption{\label{fig:10dim-autoencoder}Data analysis by autoencoder. An autoencoder
can recover system dynamics if all data is on the invariant manifold.
For the solid line MAN, data was artificially forced to be on an a-priori
calculated invariant manifold. However if the data is not on an invariant
manifold, such as for data set ST-1, the autoencoder calculation is
meaningless. The dotted line VF-1 represents the analytic calculation
for the first natural frequency, the dash-dotted VF-2 depicts the
analytic calculation for the second natural frequency of vector field
(\ref{eq:10dim-finmod}).}
\end{figure}
Knowing that autoencoders are only useful if all the dynamics is on
the manifold, we have synthetically created data consisting of trajectories
with initial conditions from the invariant manifold of the first natural
frequency. We used 800 trajectories, 24 points each with time-step
$\Delta t=0.2$ starting on the manifold. Fitting an autoencoder to
this data yields a good match in figure (\ref{fig:10dim-autoencoder}),
the corresponding lines are labelled MAN. Then we tried dataset ST-1,
that matched the reference best when calculating an invariant foliation.
However, our data does not lie on a manifold and it is impossible
to make $\boldsymbol{W}\circ\boldsymbol{U}$ close to the identity
on our data. In fact the result (blue dashed line) is closer to the
second mode of vibration (green dash-dotted curve), which seems to
be the dominant vibration of the system.

\subsection{Jointed beam}

\begin{figure}
\begin{centering}
\includegraphics[width=0.99\linewidth]{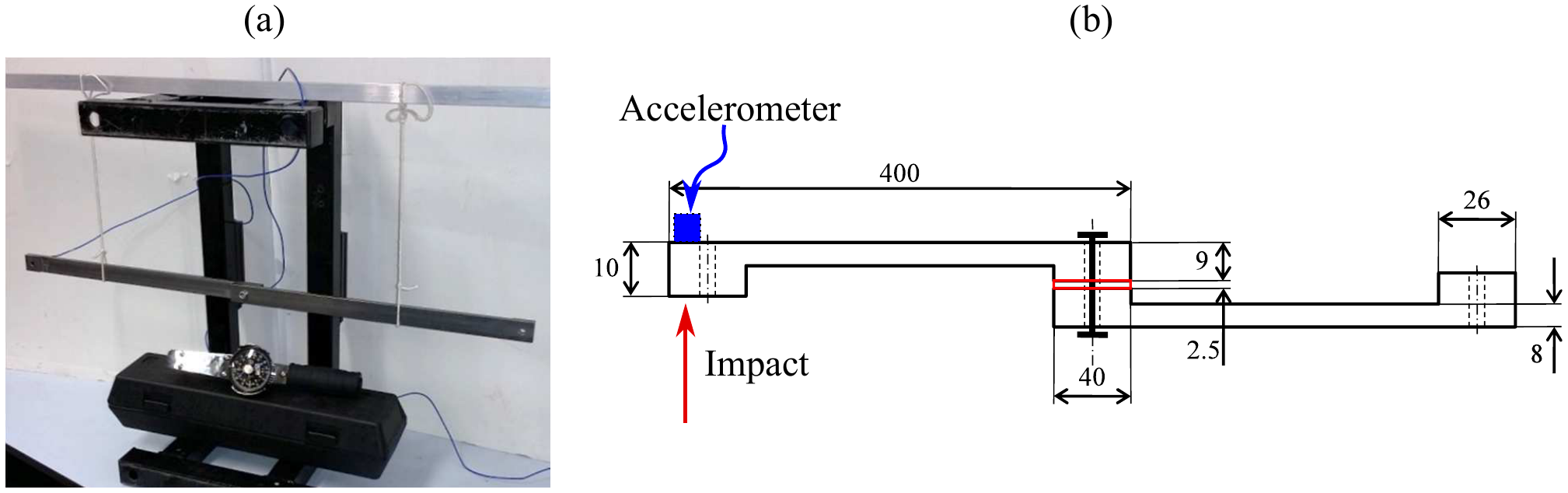}
\par\end{centering}
\caption{\label{fig:JointedBeam}(a) Experimental setup. (b) Schematic of the
jointed beam. The width of the beam (not shown) is 25mm. All measurements
are in millimetres.}
\end{figure}
It is challenging to accurately model mechanical friction, hence data
oriented methods can play a huge role in identifying dynamics affected
by frictional forces. Therefore we analyse the data published in~\cite{Titurus2016}.
The experimental setup can be seen in figure \ref{fig:JointedBeam}.
The two halves of the beam were joined together with an M6 bolt. The
two interfaces of the beams were sandblasted to increase friction
and a polished steel plate was placed between them, finally the bolt
was tightened using four different torques: minimal torque so that
the beam does not collapse under its own weight (denoted as 0~Nm),
1~Nm, 2.1~Nm and 3.1~Nm. The free vibration of the steel beam was
recorded using an accelerometer placed at the end of the beam. The
vibration was initiated using an impact hammer at the position of
the accelerometer. Calibration data for the accelerometer is not available.
For each torque value a number of 20 seconds long signals were recorded
with sampling frequency of 2048~Hz. The impacts were of different
magnitude so that the amplitude dependency of the dynamics could be
tracked. In~\cite{Titurus2016}, a linear model was fitted to each
signal and the first five vibration frequencies and damping ratios
were identified. These are represented by various markers in figure
\ref{fig:JointedBeam-ROM}. In order to make a connection between
the peak impact force and the instantaneous amplitude we also calculated
the peak root mean square (RMS) amplitude for signals with the largest
impact force for each tightening torque and found that the average
conversion factor between the peak RMS amplitude and the peak impact
force was 443, which we used to divide the peak force and plot the
equivalent peak RMS in figure \ref{fig:JointedBeam-ROM}. We also
band filtered each trajectory with a 511 point FIR filter with 3 dB
points at 30 Hz and 75 Hz and estimated the instantaneous frequency
and damping ratios from two consecutive vibration cycles, which are
then drawn as thick semi-transparent lines in figure \ref{fig:JointedBeam-ROM}
for each trajectory. It is worth noting that the more friction there
is in the system, the less reproducible the frequencies and damping
ratios become when using short trajectory segments for estimation.

To calculate the invariant foliation we used a 10-dimensional DFT
reconstructed state space, to include all five captured frequencies,
as described is appendix \ref{subsec:DFT}. We chose $\kappa=0.2$
in the optimisation problem (\ref{eq:MAP-Uhat-optim}) when finding
the invariant manifold. The result can be seen in figure \ref{fig:JointedBeam-ROM}.
Since we do not have the ground truth for this system it is not possible
to tell which method is more accurate, especially that our naive alternative
calculation (thick semi-transparent lines) displays a wide spread
of results.

\begin{figure}
\begin{centering}
\includegraphics[width=0.99\linewidth]{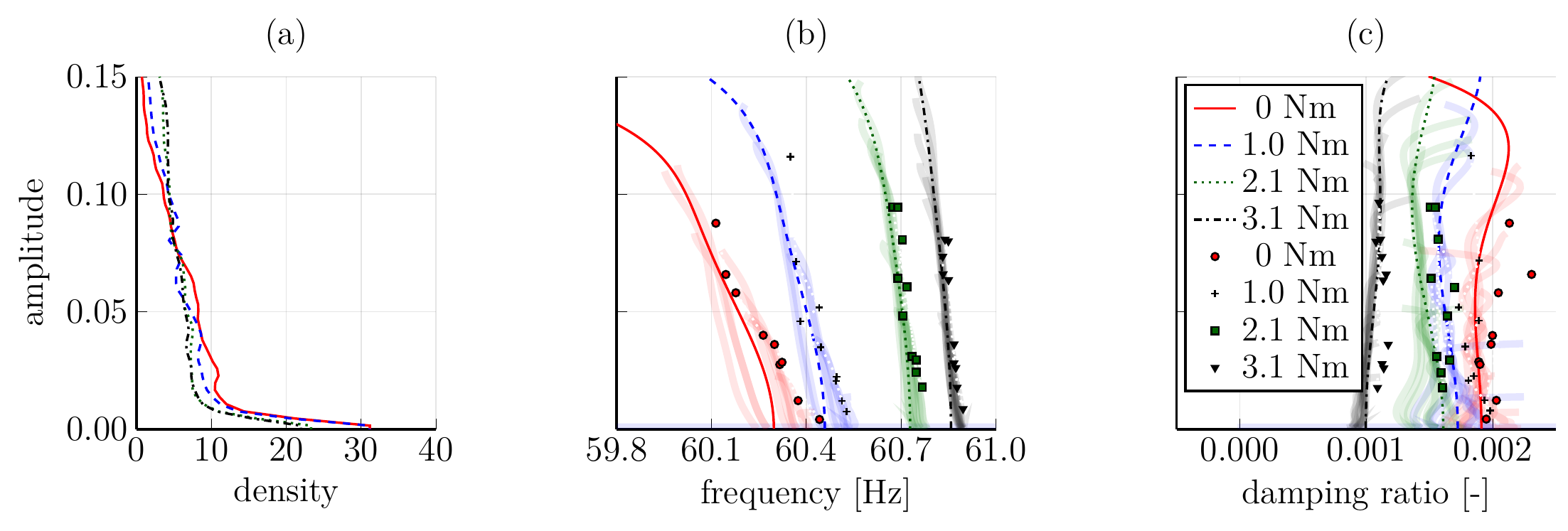}
\par\end{centering}
\caption{\label{fig:JointedBeam-ROM}Instantaneous frequencies and damping
ratios of the jointed beam set-up in figure \ref{fig:JointedBeam}.
Solid red lines correspond to the invariant foliation of the system
with minimal tightening torque, blue dashed lines with tightening
torque of 2.1 Nm and green dotted lines with tightening torque of
3.1 Nm. The markers of the same colour show results calculated in
\cite{Titurus2016} using curve fitting. (a) data density with respect
to amplitude $A\left(r\right)$, (b) instantaneous frequency, (c)
instantaneous damping ratio.}
\end{figure}
\begin{figure}
\begin{centering}
\includegraphics[width=1\linewidth]{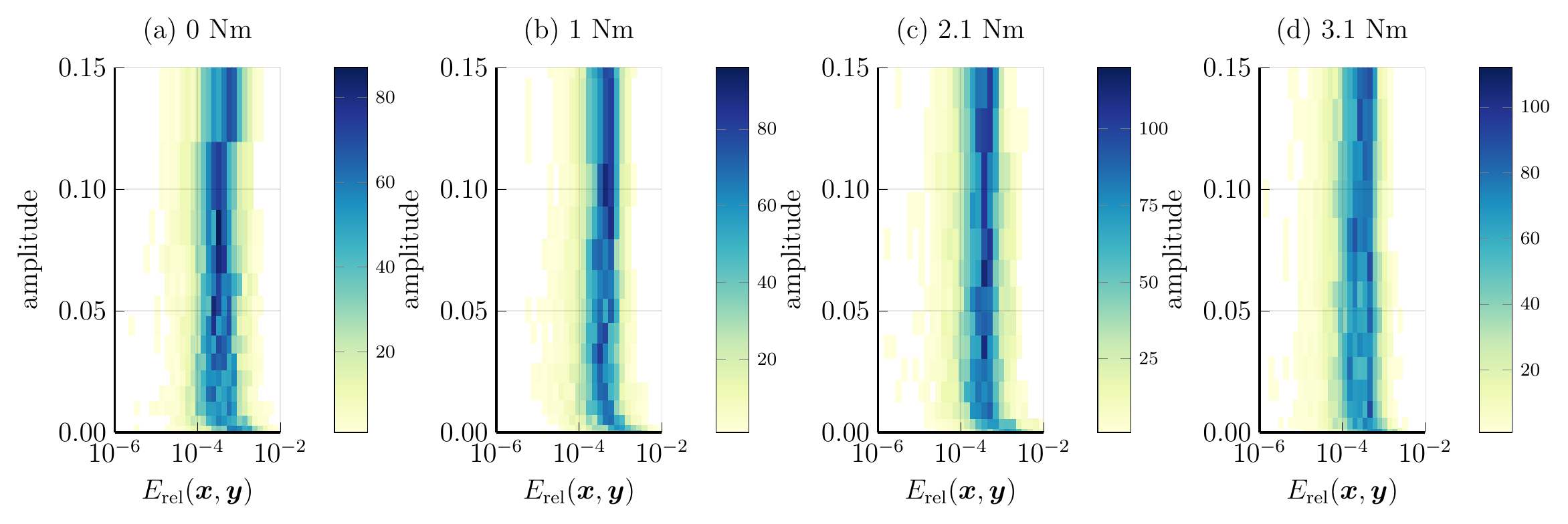}
\par\end{centering}
\caption{\label{fig:JointedBeam-Histogram}Histograms of the fitting error
(\ref{eq:FittingError}) for all four tightening torques of the jointed
beam. Note that the errors look very similar despite that repeatability
for lower tightening torques appear worse in figure \ref{fig:JointedBeam-ROM}.}
\end{figure}
\begin{figure}
\begin{centering}
\includegraphics[width=0.49\linewidth]{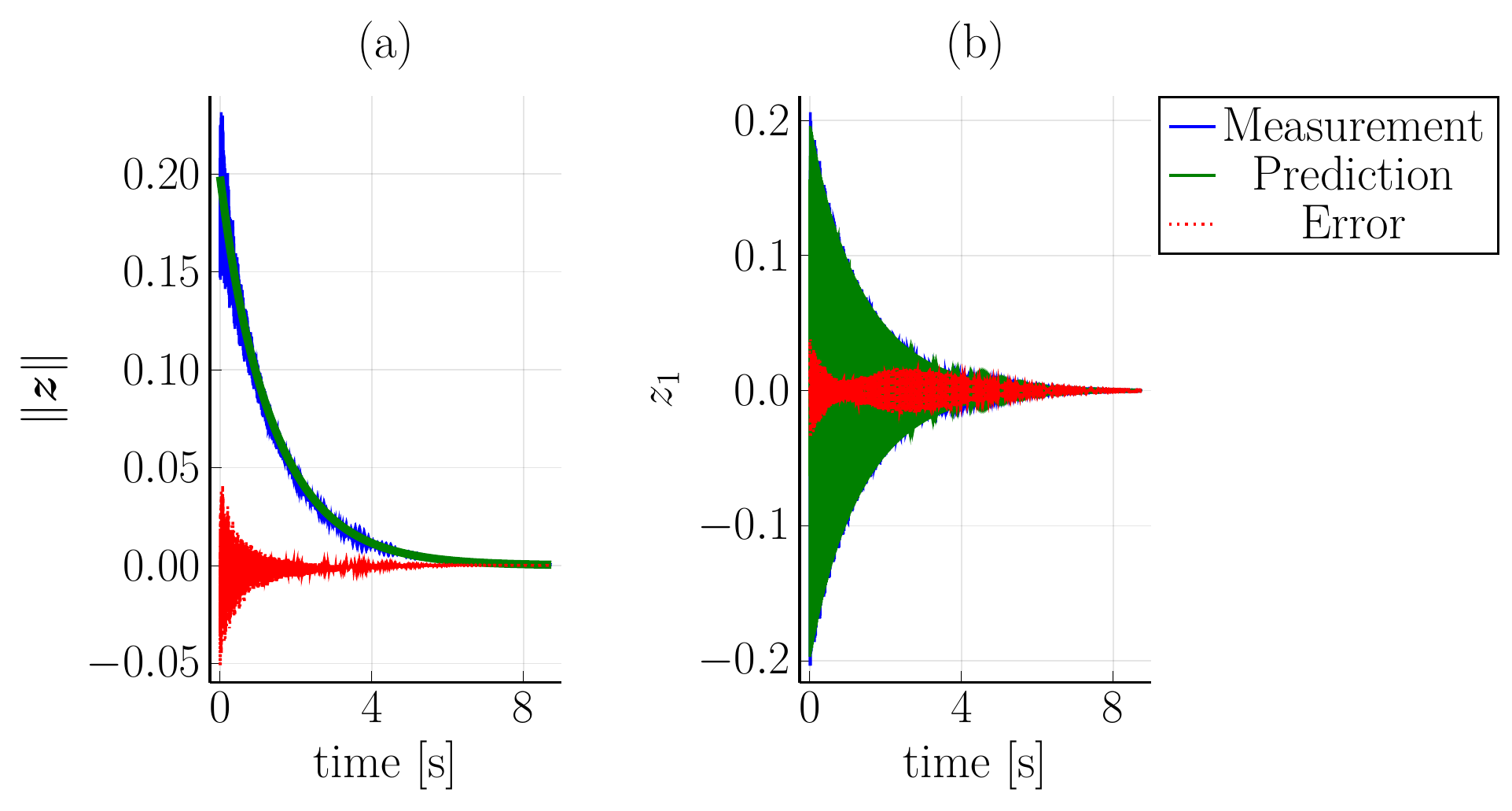}\includegraphics[width=0.49\linewidth]{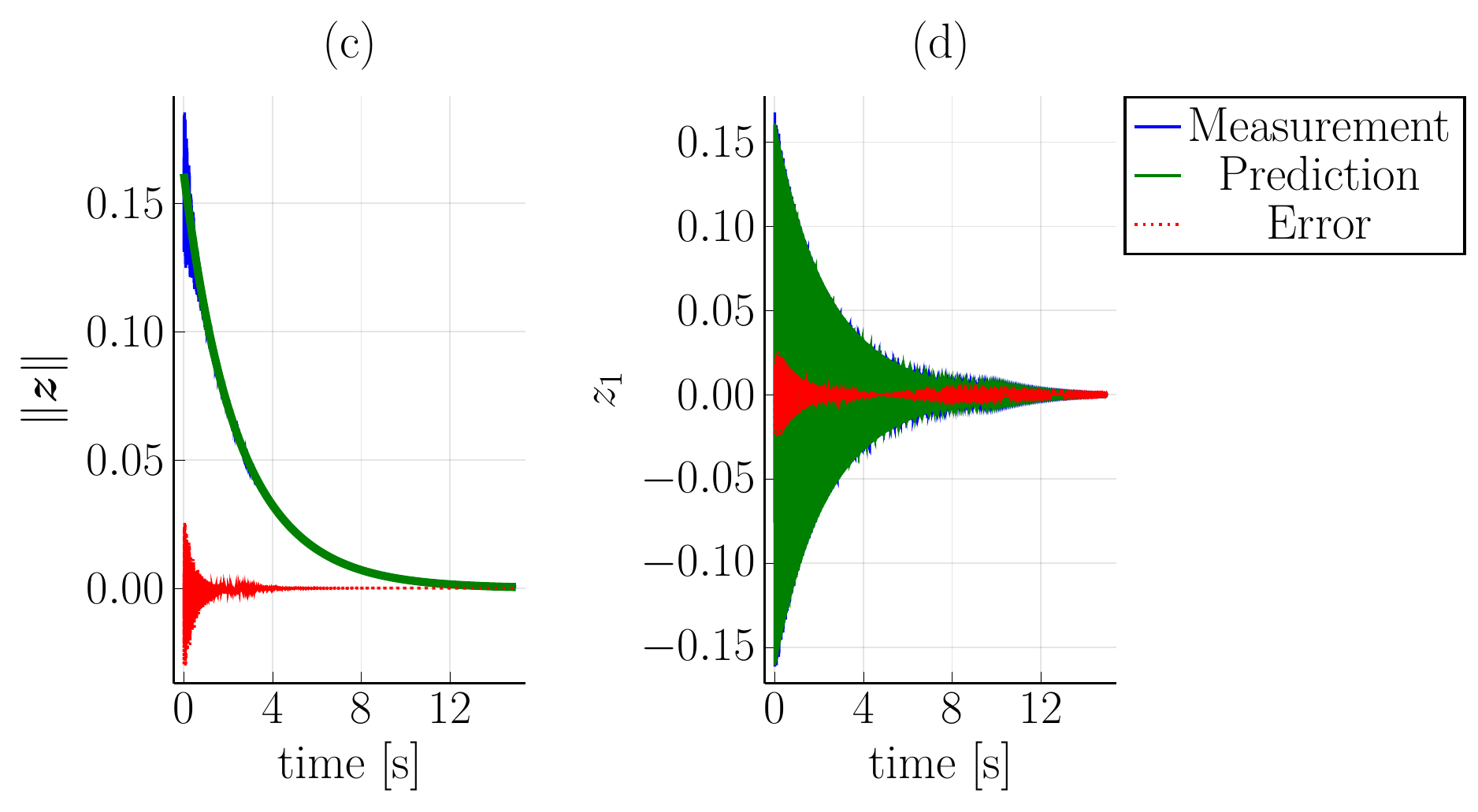}
\par\end{centering}
\caption{\label{fig:JointedBeam-reconstruction}Reconstruction error of the
foliation. (a,b) Reconstructing the first signal of the series of
experimental data with 0 Nm tightening torque. (c,d) Reconstructing
the first signal of the series of experimental data with 3.1 Nm tightening
torque. The compared values are calculated through the encoder $\boldsymbol{z}=\boldsymbol{U}\left(\boldsymbol{x}_{k}\right)$
and the reconstructed values are from the iteration $\boldsymbol{z}_{k+1}=\boldsymbol{S}\left(\boldsymbol{z}_{k}\right)$.
The amplitude error ($\left\Vert \boldsymbol{z}\right\Vert $) is
minimal, however, over time phase error accumulates, hence direct
comparison of a coordinate ($z_{1}$) can look unfavourable.}
\end{figure}

The fitting error to the invariant foliation can be assessed from
the histograms in figure \ref{fig:JointedBeam-Histogram}. The distribution
of the error is very similar to the synthetic model, which is nearly
uniform with respect to the vibration amplitude. It is also clear
that there is no real difference in the fitting error for different
tightening torques, which indicates that frictional dynamics can be
accurately characterised using invariant foliations.

As a final test, we also assess whether the measured signal is reproduced
by the invariant foliation $\left(\boldsymbol{U},\boldsymbol{S}\right)$
in figure \ref{fig:JointedBeam-reconstruction}. For this we apply
the encoder $\boldsymbol{U}$ to our original signal $\boldsymbol{x}_{k}$,
$k=1,\ldots,N$ and compare this signal to the one produced by the
recursion $\boldsymbol{z}_{k+1}=\boldsymbol{S}\left(\boldsymbol{z}_{k}\right)$,
where $\boldsymbol{z}_{1}=\boldsymbol{U}\left(\boldsymbol{x}_{1}\right)$.
The instantaneous amplitude error in both cases of figures \ref{fig:JointedBeam-reconstruction}(a,c)
was mostly due to high frequency oscillations that was not completely
filtered out by the encoder $\boldsymbol{U}$ at the high amplitudes.
The phase error seemed to have only accumulated for the lowest tightening
torque (0 Nm) in figure \ref{fig:JointedBeam-reconstruction}(b).
This is to be expected over long trajectories, since the fitting procedure
only minimises the prediction error for a single time step. Unfortunately,
accumulating phase error is rarely shown in the literature, where
only short trajectories are compared, in contrast to the $28685$
and $25149$ time-steps that are displayed in figures \ref{fig:JointedBeam-reconstruction}(b,d),
respectively.

\section{Discussion}

The main conclusion of this study is that only invariant foliations
are suitable for ROM identification from off-line data. Using an invariant
foliation avoids the need to use resonance decay \cite{EHRHARDT2016612},
or waiting for the signal to settle near the most attracting invariant
manifold \cite{Cenedese2022NatComm}, thereby throwing away valuable
data. Using invariant foliations can make use of unstructured data
with arbitrary initial conditions, such as impact hammer tests. Invariant
foliations produce genuine ROMs and not only parametrise a-priori
known invariant manifolds, like other methods \cite{Cenedese2022NatComm,Champion2019Autoencoder,Yair2017DiffusionNormal}.
We have shown that the high-dimensional function required to represent
an encoder can be represented by polynomials with compressed tensor
coefficients, which significantly reduces the computational and memory
costs. Compressed tensors are also amenable to further analysis, such
as singular value decomposition \cite{GrasedyckSVD}, which gives
way to mathematical interpretations of the decoder $\boldsymbol{U}$.
The low dimensional map $\boldsymbol{S}$ is also amenable to normal
form transformations, which can be used to extract information such
as instantaneous frequencies and damping ratios.

We have tested the related concept of Koopman eigenfuntions, which
differs from an invariant foliation in that map $\boldsymbol{S}$
is assumed to be linear. If there are no internal resonances, Koopman
eigenfunctions are theoretically equivalent to invariant foliations.
However in numerical settings unresolved near internal resonances
become important and therefore Koopman eigenfunctions become inadequate.
We have also tried to fit autoencoders to our data \cite{Cenedese2022NatComm},
but apart from the artificial case where the invariant manifold was
pre-computed, it performed even worse than Koopman eigenfunctions.

Fitting an invariant foliation to data is extremely robust when state
space data is available. However, when the state space needed to be
reproduced from a scalar signal, the results were not as accurate
as we hoped for. While Taken's theorem allows for any generic delay
coordinates, in practice a non-optimal choice can lead to poor results.
We were expecting that the embedding dimension at least for low amplitude
signals would be the same as the attractor dimension. This is however
not true if the data also includes higher amplitude points. Despite
not being theoretically optimal, we have found that perfect reproducing
filter banks produce accurate results for low amplitude signals and
at the same time provide a state-space reconstruction with the same
dimensionality as that of the attractor. Future work should include
exploring various state-space reconstruction techniques in combination
with fitting an invariant foliation to data.

We did not fully explore the idea of locally defined invariant foliations
in remark \ref{rem:extrasimpleROM}, which can lead to computationally
efficient methods. Further research can also be directed towards cases,
where the data is on a high-dimensional submanifold of an even higher
dimensional vector space $X$. This is where an autoencoder and invariant
foliation may be combined.

\appendix

\section{Implementation details}

In this appendix we deal with how to represent invariant foliations,
locally defined invariant foliations and autoencoders. We also describe
our specific techniques to carry out optimisation on these representations.

\subsection{Dense polynomial representation}

The nonlinear map $\boldsymbol{S}$, the polynomial part of the locally
defined invariant foliation $\boldsymbol{W}_{0}$ and the decoder
part an autoencoder are represented by dense polynomials, that we
define here.

A polynomial $\boldsymbol{P}:Z\to X$ of order $d$ is represented
as a linear combination of monomials formed from the coordinates of
vectors of $Z$. First we define what we mean by a monomial and how
we order them. Given a basis in $Z$ with $\nu=\dim Z$, we can represent
each vector $\boldsymbol{z}\in Z$ by $\boldsymbol{z}=\left(z_{1},z_{2},\ldots,z_{\nu}\right)$.
Then we define non-negative integer vectors $\boldsymbol{m}=\left(m_{1},\ldots,m_{\nu}\right)\in\mathbb{N}^{\nu}$
and finally define a monomial of $\boldsymbol{z}$ by $\boldsymbol{z}^{\boldsymbol{m}}=z_{1}^{m_{1}}z_{2}^{m_{2}}\cdots z_{\nu}^{m_{\nu}}$.
We also need an ordered set of integer exponents, which is denoted
by
\[
\mathfrak{M}_{c,d}=\left\{ \boldsymbol{m}\in\mathbb{N}^{\nu}:c\le m_{1}+\cdots+m_{\nu}\le d\right\} .
\]
The ordering of $\mathfrak{M}_{c,d}$ is such that $\boldsymbol{m}<\boldsymbol{n}$
if there exists $k\le\nu$ and $m_{j}<n_{j}$ for all $j=1,\ldots,k$.
The cardinality of set $\mathfrak{M}_{c,d}$ is denoted by $\#\left(\nu,c,d\right)$.
Therefore we can also write that $\mathfrak{M}_{c,d}=\left\{ \boldsymbol{m}_{1},\boldsymbol{m}_{2},\ldots,\boldsymbol{m}_{\#\left(\nu,c,d\right)}\right\} $.
Using the ordered notation of monomials, a polynomial $\boldsymbol{P}$
containing terms at least order $c$ up to order-$d$ is represented
by a matrix $\underline{\boldsymbol{P}}\in\mathbb{R}^{n\times\#\left(\nu,c,d\right)}$,
such that 
\[
\boldsymbol{P}\left(\boldsymbol{z}\right)=\underline{\boldsymbol{P}}\left(\boldsymbol{z}^{\boldsymbol{m}_{1}},\boldsymbol{z}^{\boldsymbol{m}_{2}},\cdots,\boldsymbol{z}^{\boldsymbol{m}_{\#\left(\nu,c,d\right)}}\right)^{T}.
\]
We also call such polynomials order-$\left(c,d\right)$ polynomials.
For optimisation purposes order-$\left(c,d\right)$ polynomials form
an Euclidean manifold and therefore no constraint is placed on matrix
$\underline{\boldsymbol{P}}$.

\subsection{\label{subsec:AENC-repr}Autoencoder representation}

A polynomial autoencoder of order $d$ (as defined in \cite{Cenedese2022NatComm})
is represented by an orthogonal matrix $\underline{\boldsymbol{U}}\in\mathbb{R}^{n\times\nu}$
($\underline{\boldsymbol{U}}^{T}\underline{\boldsymbol{U}}=\boldsymbol{I}$)
and an order-$\left(2,d\right)$ polynomial $\boldsymbol{W}$, represented
by matrix $\underline{\boldsymbol{W}}\in\mathbb{R}^{n\times\#\left(\nu,2,d\right)}$.
The associated encoder is given by $\boldsymbol{U}\left(\boldsymbol{x}\right)=\underline{\boldsymbol{U}}^{T}\boldsymbol{x}$
and the decoder is given by $\boldsymbol{W}\left(\boldsymbol{z}\right)=\boldsymbol{U}\boldsymbol{z}+\boldsymbol{W}\left(\boldsymbol{z}\right)$,
which must satisfy the additional constraint $\boldsymbol{U}\left(\boldsymbol{W}\left(\boldsymbol{z}\right)\right)=\boldsymbol{z}$,
that in terms of our matrices means $\underline{\boldsymbol{U}}^{T}\underline{\boldsymbol{W}}=\boldsymbol{0}$.
In summary, we have the constraints
\begin{equation}
\begin{array}{rl}
\underline{\boldsymbol{U}}^{T}\underline{\boldsymbol{U}} & =\boldsymbol{I},\\
\underline{\boldsymbol{U}}^{T}\underline{\boldsymbol{W}} & =\boldsymbol{0},
\end{array}\label{eq:AENC-constr}
\end{equation}
which turns the admissible set of matrices $\underline{\boldsymbol{U}},\,\underline{\boldsymbol{W}}$
into a matrix manifold. We call this manifold the orthogonal autoencoder
manifold and denote it by $\mathit{OAE}_{n,\nu,\#\left(\nu,2,d\right)}$.
In paper \cite{Cenedese2022NatComm}, the authors also consider the
case when the decoder is decoupled from the encoder. In this case
$\boldsymbol{W}\left(\boldsymbol{z}\right)$ does not include the
orthogonal matrix $\boldsymbol{U}$ and therefore $\boldsymbol{W}\left(\boldsymbol{z}\right)$
is represented by matrix $\underline{\boldsymbol{W}}\in\mathbb{R}^{n\times\#\left(\nu,1,d\right)}$,
which also includes linear terms. The constraint on matrices $\underline{\boldsymbol{U}},\,\underline{\boldsymbol{W}}$
is therefore different
\begin{equation}
\begin{array}{rl}
\underline{\boldsymbol{U}}^{T}\underline{\boldsymbol{U}} & =\boldsymbol{I},\\
\underline{\boldsymbol{U}}^{T}\underline{\boldsymbol{W}} & =\boldsymbol{M},
\end{array}\label{eq:AENC-constr-1}
\end{equation}
where matrix $\boldsymbol{M}$ represents the identity polynomial
with respect to the monomials $\mathfrak{M}_{1,d}$. We call the resulting
matrix manifold the generalised orthogonal autoencoder manifold and
denote it by $\mathit{GOAE}_{\boldsymbol{M},n,\nu,\#\left(\nu,1,d\right)}$.
Unfortunately, this generalised autoencoder leads to an ill-defined
optimisation problem. Indeed, if all the data is on an invariant manifold
(the only sensible case), then the directions encoded by $\boldsymbol{U}$
are not defined by the data, because any transversal direction is
equally suitable for projecting the data. In our numerical test $\boldsymbol{U}$
quickly degenerated and led to spurious results. Nevertheless, due
to the similar formulation we treat both cases at the same time, because
$\mathit{OAE}_{n,m,l}$ is a particular case of $\mathit{GOAE}_{\boldsymbol{M},n,m,l}$
with $\boldsymbol{M}$ being equal to zero. The details of how projections
and retractions are calculated for $\mathit{GOAE}_{\boldsymbol{M},n,m,l}$
are presented in section \ref{subsec:AENC-manif}.

\subsection{\label{subsec:sparse-poly}Compressed polynomial representation with
hierarchical tensor coefficients}

To represent encoders of invariant foliations we need to approximate
high-dimensional functions. Dense polynomials are not suitable to
represent high-dimensional functions, because the number of required
parameters increase combinatorially. One approach is to use multi-layer
neural networks \cite{HORNIK1991,ElbrachterBolcskei2021}, however
their training can be problematic \cite{orr2003neural}. Standard
training methods take a long time and they rarely reach the global
minimum (and there is no free lunch \cite{FreeLunch1997}). In addition,
the approximation can overfit the data, that is the accuracy on unseen
data can be significantly worse than on the training data. In terms
of the geometry of neural networks, they do not form a differentiable
manifold and therefore parameters may tend to infinity without minimising
the loss function \cite{PetersenTopologyNeuralNetworks2021}.

To represent encoders we still use polynomials, except that the polynomial
coefficients are represented by compressed tensors in the hierarchical
Tucker (HT) format. This tensor format was introduced in \cite{HackbuschKuhn2009}
and demonstrated to solve many problems \cite{TensorApproxSurvey}
that would normally wrestle with the 'curse of dimensionality' \cite{bellman2015adaptive}.
In chemisty and physics a somewhat different but related format, the
matrix product state \cite{MatrixProductStates2007} is frequently
used and in other problems a particular type of HT representation
the tensor train format \cite{Oseledets2011} is used. However, the
most important property of HT tensors is that they form a smooth quotient
manifold and therefore suitable to use in optimisation problems \cite{USCHMAJEW2013133}. 

To define the format, we partly follow the notation of \cite{USCHMAJEW2013133}
in our definitions. First, we need to define the dimension tree.
\begin{defn}
Given an order $d$, a dimension tree $T$ is a non-trivial rooted
binary tree whose nodes can be labelled by subsets of $\left\{ 1,2,\ldots,d\right\} $
such that
\begin{enumerate}
\item the root has the label $t_{r}=\left\{ 1,2,\ldots,d\right\} $, and
\item every node $t\in T$, which is not a leaf has two descendants, $t_{1}$
and $t_{2}$ that form an ordered partition of $t$, that is,
\[
t_{1}\cup t_{2}=t\quad\text{and}\quad\mu<\nu\;\text{for all}\;\mu\in t_{1},\nu\in t_{2}.
\]
\end{enumerate}
The set of leaf nodes is denoted by $L=\left\{ \left\{ 1\right\} ,\left\{ 2\right\} ,\ldots,\left\{ d\right\} \right\} $.
\end{defn}
\begin{defn}
Let $T$ be a dimension tree and $\boldsymbol{k}=\left(k_{t}\right)_{t\in T}$
a set of positive integers. The hierarchical Tucker format for $\mathbb{R}^{n_{0}\times n_{1}\times\cdots\times n_{d}}$
is defined as follows
\begin{enumerate}
\item For a leaf node $U_{t}$ is an orthonormal matrix, such that $U_{t}\in\mathbb{R}^{k_{t}\times n_{t}}$
with elements $U_{t}\left(i;j\right)$, where $i=1\ldots k_{t}$,
$j=1\ldots n_{t}$
\item For not a leaf node with $t=t_{1}\cup t_{2}$ we define
\begin{equation}
U_{t}\left(p;i_{1},\ldots,i_{\left|t_{1}\right|},j_{1},\ldots,j_{\left|t_{2}\right|}\right)=\sum_{q=1}^{k_{t_{1}}}\sum_{r=1}^{k_{t_{2}}}B_{t}\left(p,q,r\right)U_{t_{1}}\left(q;i_{1},\ldots,i_{\left|t_{1}\right|}\right)U_{t_{2}}\left(r;j_{1},\ldots,j_{\left|t_{2}\right|}\right),\label{eq:HT-notation}
\end{equation}
where $\left|t\right|$ stands for the number of indices in label
$t$.
\end{enumerate}
\end{defn}
The definition is recursive, hence data storage for $U_{t}$ is only
required for the leaf nodes, and for non-leaf nodes only storing $B_{t}$
is sufficient. In other words, the set of matrices
\[
\left\{ U_{t}:t\in L\right\} \cup\left\{ B_{t}:t\in T\setminus L\right\} 
\]
fully specifies a HT tensor. For each node, apart from the root node
$t_{r}$, one can apply a transformation, which does not change the
resulting tensor. That is, for $t\neq t_{r}$ define $\tilde{U}_{t}=A_{t}U_{t}$
and $\tilde{B}_{t}=B_{t}\times_{2,3}A_{t_{1}}^{-1}\otimes A_{t_{2}}^{-1}$,
where $A_{t}\in\mathbb{R}^{k_{t}\times k_{t}}$ are invertible matrices.
Without restricting generality, we can stipulate that the leaf nodes
$U_{t}$ and matricisation of $B_{t}$ with respect to the first index
are also orthogonal matrices, which means that transformations $A_{t}$
must be unitary to preserve orthogonality. This identity transformation
defines an equivalence relation between parametrisations of a HT tensor
and therefore the HT format is a quotient manifold as described in
\cite[chapter 7]{boumal2022intromanifolds}. Using a quotient manifold
would prevent us from optimising for the matrices $U_{t}$, $B_{t}$
individually, as we need to treat the whole HT tensor as a single
entity. In addition, we could not find any example in the literature,
where a HT tensor was treated as a quotient manifold. Instead, we
simply assume a simpler structure, that is, $B_{t_{r}}$ is on a Euclidean
manifold and the rest of the coefficients $B_{t}$ and $U_{t}$ are
orthogonal matrices and therefore elements of Stiefel manifolds, which
is the running example in \cite{boumal2022intromanifolds}. This way
we have the HT format as an element of a product manifold, where each
coefficient matrix remains independent of each other. It is also possible
to calculate singular values of HT tensors. Vanishing singular values
indicate that the required quantity is well-approximated and that
the HT tensor can be simplified to include less terms. This further
adds to our ability to identify parsimonious models. In what follows,
we use balanced dimension trees and up to rank six matrices $U_{t}$,
$B_{t}$ for simplicity.

Using notation (\ref{eq:HT-notation}), we can write the encoders
as
\begin{equation}
\boldsymbol{U}\left(\boldsymbol{x}\right)=\boldsymbol{U}^{1}\boldsymbol{x}+\sum_{d=2}^{p}\sum_{i_{1}\cdots i_{d}=1}^{n}U_{t_{r}}^{d}\left(\cdot;i_{1},\ldots,i_{d}\right)x_{i_{1}}\cdots x_{i_{d}},\label{eq:HT-encoder-repr}
\end{equation}
where $\boldsymbol{U}^{1}$ is an orthogonal matrix, similarly element
of a Stiefel manifold. We also apply the constraint that $\boldsymbol{U}\left(\boldsymbol{W}_{1}\boldsymbol{z}\right)$
must be linear. This simply means that for the nonlinear terms of
$\boldsymbol{U}$ the data must not include components in the $\boldsymbol{W}_{1}$
direction. This can be achieved by projecting our data, that is $\hat{\boldsymbol{x}}=\boldsymbol{U}_{1}^{\perp}\boldsymbol{x}$,
where $\boldsymbol{U}_{1}^{\perp}$ is obtained in step F2 of the
invariant foliation identification process of section \ref{sec:ROM_id_processes}
and consists of row vectors that are orthogonal to the column vectors
of $\boldsymbol{W}_{1}$. Therefore our constrained encoder assumes
the form of
\[
\boldsymbol{U}\left(\boldsymbol{x}\right)=\boldsymbol{U}^{1}\boldsymbol{x}+\sum_{d=2}^{p}\sum_{i_{1}\cdots i_{d}=1}^{n}U_{t_{r}}^{d}\left(\cdot;i_{1},\ldots,i_{d}\right)\hat{x}_{i_{1}}\cdots\hat{x}_{i_{d}}.
\]

\subsection{Optimisation techniques}

Given our compressed representation of the encoder in the form (\ref{eq:HT-encoder-repr}),
we can now discuss how to solve the optimisation problem (\ref{eq:MAP-U-optim}). 

HT tensors, which make up the parametrisation of the encoder $\boldsymbol{U}$,
depend linearly on each matrix $B_{t}^{d}$, $U_{t}^{d}$ individually.
Therefore, it is beneficial to carry out the optimisation for each
matrix component individually, and cycle through all matrix components
a number of times until convergence is reached. This is called batch
coordinate descent \cite{ortega1970iterative}. The complicating factor
in this approach is that the encoder also appears as an inner function
of map $\boldsymbol{S}$ and therefore the dependence of the objective
function becomes nonlinear on matrices $B_{t}$, $U_{t}$. For optimisation
in each coordinate we use the second order trust-region method \cite{conn2000trust}
as we can explicitly calculate the Hessian of the objective function
with respect to each matrix $B_{t}^{d},U_{t}^{d}$ and the parameters
of $\boldsymbol{S}$. We only take a limited number of steps with
the trust region method, as we cycle through all parameter matrices.
The algorithm cycles through each tensor and the parameters of $\boldsymbol{S}$
in a given order, but for each tensor of order $d$, only the coefficient
matrix for which the gradient of the objective function is the largest
is optimised for. This is a variant of the Gauss-Southwell algorithm
\cite{GaussSouthwell2015}. We found that this technique, as it eliminates
unnecessary optimisation steps for parameters that do not influence
the objective function much, converges relatively fast. Attempts to
use off-the-shelf optimisers were fruitless, due to computational
costs.

To identify locally defined foliations we use the Riemannian trust
region method and for identifying autoencoders, we use the Riemannian
BFGS quasi-Newton technique from software package \cite{Manopt2022}.

\subsection{Projections and retractions on matrix manifolds}

As we perform optimisation on matrix manifolds, we also require an
orthogonal projection from the ambient space to the tangent space
to the manifold and a retraction \cite{boumal2022intromanifolds,absil2009optimization}.
The advantage of optimising on a manifold instead of using standard
constrained optimisation is that retractions and projections are generally
easier to evaluate than solving the constraints using generic methods.

Let us assume that our matrix manifold $\mathcal{M}$ is embedded
into the Euclidean space $\mathcal{E}$. As computers deal with lists
of numbers, we only have a cost function defined on $\mathcal{E}$,
which we denote by $\overline{F}:\mathcal{E}\to\mathbb{R}$, instead
on the manifold $F:\mathcal{M}\to\mathbb{R}$. Carrying out optimisation
on $\mathcal{M}$ therefore requires various corrections when we use
cost function $\overline{F}$. The gradient $\nabla F$ on $\mathcal{M}$
is a map from $\mathcal{M}$ to the tangent bundle $T\mathcal{M}$.
Generally $\overline{F}$ does not map into $T\mathcal{M}$, therefore
a projection is necessary. In particular, for a specific point $\boldsymbol{p}\in\mathcal{M}$,
the projection is the linear map $\boldsymbol{P}_{\boldsymbol{p}}:T_{\boldsymbol{p}}\mathcal{E}\to T_{\boldsymbol{p}}\mathcal{M}$
and the gradient is calculated as 
\[
\nabla F\left(\boldsymbol{p}\right)=\boldsymbol{P}_{\boldsymbol{p}}\nabla\overline{F}\left(\boldsymbol{p}\right).
\]
A simple gradient descent method would move in the direction of the
gradient, however that is not necessarily on manifold $\mathcal{M}$,
so we need to bring the result back to $\mathcal{M}$ using a retraction.
A retraction is a map $R_{\boldsymbol{p}}\left(\boldsymbol{X}\right):T\mathcal{M}\to\mathcal{M}$
for which $DR_{\boldsymbol{p}}\left(\boldsymbol{0}\right)\boldsymbol{Y}=\boldsymbol{Y}$.
A retraction is supposed to approximate the so-called exponential
map on the tangent space, which produces the geodesics along the manifold
in the direction of the tangent vector $\boldsymbol{X}$ up to the
length of $\boldsymbol{X}$. A second order approximation of the exponential
retraction is the solution of the minimisation problem
\begin{equation}
R_{\boldsymbol{p}}\left(\boldsymbol{X}\right)=\arg\min_{\boldsymbol{q}\in\mathcal{M}}\left\Vert \boldsymbol{p}+\boldsymbol{X}-\boldsymbol{q}\right\Vert ^{2}.\label{eq:proj-retract}
\end{equation}
Projection like retractions are explained in detail in \cite{Absil2012Retraction}.
Using a retraction we can pull back our result onto the manifold.
The Hessian of $F$ can be defined in terms of a second order retraction
$R$. If we assume that $\tilde{F}\left(\boldsymbol{p},\boldsymbol{X}\right)=\overline{F}\left(R_{\boldsymbol{p}}\left(\boldsymbol{P}_{\boldsymbol{p}}\boldsymbol{X}\right)\right)$,
then the Hessian is 
\[
D^{2}F\left(\boldsymbol{p}\right)=D_{2}^{2}\tilde{F}\left(\boldsymbol{p},\boldsymbol{0}\right).
\]
The expression of the Hessian can be simplified in many ways, which
is described, e.g., in \cite{boumal2022intromanifolds}.

In what follows we detail two matrix manifolds that do not appear
in the literature and are specific to our problems. We also use the
Stiefel manifold of orthogonal matrices that is covered in many publications
\cite{boumal2022intromanifolds}. In particular, we use the so-called
polar retraction of the Stiefel manifold, which is equivalent to equation
(\ref{eq:proj-retract}). For our specific matrix manifolds we solve
(\ref{eq:proj-retract}).

\subsubsection{\label{subsec:AENC-manif}Autoencoder manifold}

The two matrices $\boldsymbol{p}_{1}=\underline{\boldsymbol{U}}\in\mathbb{R}^{n\times m}$
, $\boldsymbol{p}_{2}=\underline{\boldsymbol{W}}\in\mathbb{R}^{n\times l}$
that satisfy the constraints (\ref{eq:AENC-constr}) represent an
orthogonal autoencoder form the matrix manifold $\mathcal{M}=\mathit{GOAE}_{\boldsymbol{M},n,m,l}$,
which is defined by the zero-level set of submersion

\[
\boldsymbol{h}\left(\boldsymbol{p}_{1},\boldsymbol{p}_{2}\right)=\begin{pmatrix}\boldsymbol{p}_{1}^{T}\boldsymbol{p}_{1}-\boldsymbol{I}\\
\boldsymbol{p}_{1}^{T}\boldsymbol{p}_{2}-\boldsymbol{M}
\end{pmatrix}.
\]
For $\mathcal{M}$ to be a manifold the derivative $D\boldsymbol{h}$
must have full rank. Indeed, we calculate that 
\[
D\boldsymbol{h}\left(\boldsymbol{p}_{1},\boldsymbol{p}_{2}\right)\left(\boldsymbol{X}_{1},\boldsymbol{X}_{2}\right)=\begin{pmatrix}\boldsymbol{X}_{1}^{T}\boldsymbol{p}_{1}+\boldsymbol{p}_{1}^{T}\boldsymbol{X}_{1}\\
\boldsymbol{X}_{1}^{T}\boldsymbol{p}_{2}+\boldsymbol{p}_{1}^{T}\boldsymbol{X}_{2}
\end{pmatrix},
\]
and substitute $\boldsymbol{X}_{1}=\boldsymbol{p}_{1}\boldsymbol{A}$,
$\boldsymbol{X}_{2}=\boldsymbol{p}_{1}\boldsymbol{B}$, which yields
\[
D\boldsymbol{h}\left(\boldsymbol{p}_{1},\boldsymbol{p}_{2}\right)\left(\boldsymbol{X}_{1},\boldsymbol{X}_{2}\right)=\begin{pmatrix}\boldsymbol{A}^{T}+\boldsymbol{A}\\
\boldsymbol{B}
\end{pmatrix}.
\]
Since $\boldsymbol{A}$ and $\boldsymbol{B}$ are arbitrary matrices
the range of $D\boldsymbol{h}$ is the full set of symmetric and general
matrices, that is, $D\boldsymbol{h}$ has full rank, hence $\mathcal{M}$
is an embedded manifold.

\paragraph{Tangent space.}

The tangent space $T_{\left(\boldsymbol{p}_{1},\boldsymbol{p}_{2}\right)}\mathcal{M}$
is defined as the null space of $D\boldsymbol{h}\left(\boldsymbol{p}\right)$,
that is
\[
T_{\left(\boldsymbol{p}_{1},\boldsymbol{p}_{2}\right)}\mathcal{M}=\left\{ \boldsymbol{X}\in\mathbb{R}^{n\times m}:D\boldsymbol{h}\left(\boldsymbol{p}_{1},\boldsymbol{p}_{2}\right)\left(\boldsymbol{X}_{1},\boldsymbol{X}_{2}\right)=\left(\boldsymbol{0},\boldsymbol{0}\right)\right\} .
\]
We define $\boldsymbol{p}^{\perp}\in\mathbb{R}^{n\times\left(n-m\right)}$
by $\boldsymbol{p}^{\perp T}\boldsymbol{p}^{\perp}=\boldsymbol{I}$
and $\boldsymbol{p}^{T}\boldsymbol{p}^{\perp}=\boldsymbol{0}$, hence
the direct sum of ranges of $\boldsymbol{p}$ and $\boldsymbol{p}^{\perp}$
span $\mathbb{R}^{n}$. To characterise the tangent space we decompose
$\boldsymbol{X}_{1}=\boldsymbol{p}_{1}\boldsymbol{A}_{1}+\boldsymbol{p}_{1}^{\perp}\boldsymbol{B}_{1}$,
$\boldsymbol{X}_{2}=\boldsymbol{p}_{1}\boldsymbol{A}_{2}+\boldsymbol{p}_{1}^{\perp}\boldsymbol{B}_{2}$,
and find that
\begin{equation}
D\boldsymbol{h}\left(\boldsymbol{p}_{1},\boldsymbol{p}_{2}\right)\left(\boldsymbol{X}_{1},\boldsymbol{X}_{2}\right)=\begin{pmatrix}\boldsymbol{A}_{1}^{T}+\boldsymbol{A}_{1}\\
\boldsymbol{B}_{1}^{T}\boldsymbol{p}_{1}^{\perp T}\boldsymbol{p}_{2}+\boldsymbol{A}_{2}
\end{pmatrix}.\label{eq:OAE-tan}
\end{equation}
For (\ref{eq:OAE-orth-1}) to equal zero, $\boldsymbol{A}_{1}$ must
be antisymmetric, $\boldsymbol{B}_{1}$, $\boldsymbol{B}_{2}$ can
be arbitrary and $\boldsymbol{A}_{2}=-\boldsymbol{B}_{1}^{T}\boldsymbol{p}_{1}^{\perp T}\boldsymbol{p}_{2}$.

\paragraph{Normal space.}

We use $\boldsymbol{X}_{1}=\boldsymbol{p}_{1}\boldsymbol{A}_{1}+\boldsymbol{p}_{1}^{\perp}\boldsymbol{B}_{1}$,
$\boldsymbol{X}_{2}=-\boldsymbol{p}_{1}\boldsymbol{B}_{1}^{T}\boldsymbol{p}_{1}^{\perp T}\boldsymbol{p}_{2}+\boldsymbol{p}_{1}^{\perp}\boldsymbol{B}_{2}$
to represent an element of the tangent space and $\boldsymbol{Y}_{1}=\boldsymbol{p}_{1}\boldsymbol{K}_{1}+\boldsymbol{p}_{1}^{\perp}\boldsymbol{L}_{1}$,
$\boldsymbol{Y}_{2}=\boldsymbol{p}_{1}\boldsymbol{K}_{2}+\boldsymbol{p}_{1}^{\perp}\boldsymbol{L}_{2}$
to represent an arbitrary element of the ambient space. For $\left(\boldsymbol{Y}_{1},\boldsymbol{Y}_{2}\right)$
to be a normal vector in $N_{\left(\boldsymbol{p}_{1},\boldsymbol{p}_{2}\right)}\mathcal{M}$
equation
\begin{align}
\left\langle \boldsymbol{X}_{1},\boldsymbol{Y}_{1}\right\rangle +\left\langle \boldsymbol{X}_{2},\boldsymbol{Y}_{2}\right\rangle  & =0\label{eq:OAE-orth-1}
\end{align}
must hold for all $\boldsymbol{B}_{1}$, $\boldsymbol{B}_{2}$ and
$\boldsymbol{A}_{1}$ antisymmetric matrices. Therefore we expand
(\ref{eq:OAE-orth-1}) into
\begin{equation}
\left\langle \boldsymbol{A}_{1},\boldsymbol{K}_{1}\right\rangle +\left\langle \boldsymbol{B}_{1},\boldsymbol{L}_{1}\right\rangle -\left\langle \boldsymbol{B}_{1}^{T}\boldsymbol{p}_{1}^{\perp T}\boldsymbol{p}_{2},\boldsymbol{K}_{2}\right\rangle +\left\langle \boldsymbol{B}_{2},\boldsymbol{L}_{2}\right\rangle =0.\label{eq:OAE-orth-3p2}
\end{equation}
To explore what parameters are allowed, we consider (\ref{eq:OAE-orth-3p2})
term-by-term. If we set $\boldsymbol{B}_{1}=\boldsymbol{0}$, $\boldsymbol{B}_{2}=\boldsymbol{0}$,
then what remains is $\left\langle \boldsymbol{A}_{1},\boldsymbol{K}_{1}\right\rangle =0$
for all $\boldsymbol{A}_{1}$ anti-symmetric. This constraint holds
if and only if $\boldsymbol{K}_{1}$ is symmetric. Now we set $\boldsymbol{A}_{1}=\boldsymbol{0},\boldsymbol{B}_{1}=\boldsymbol{0}$,
which leads to $\left\langle \boldsymbol{B}_{2},\boldsymbol{L}_{2}\right\rangle =0$
for all $\boldsymbol{B}_{2}$, hence we must have $\boldsymbol{L}_{2}=\boldsymbol{0}$.
Finally we set $\boldsymbol{A}_{1}=\boldsymbol{0},\boldsymbol{B}_{2}=\boldsymbol{0}$,
which gives us 
\begin{align*}
\left\langle \boldsymbol{B}_{1},\boldsymbol{L}_{1}\right\rangle -\left\langle \boldsymbol{B}_{1}^{T}\boldsymbol{p}_{1}^{\perp T}\boldsymbol{p}_{2},\boldsymbol{K}_{2}\right\rangle  & =0,\\
\left\langle \boldsymbol{B}_{1},\boldsymbol{L}_{1}\right\rangle -\left\langle \boldsymbol{p}_{2},\boldsymbol{p}_{1}^{\perp}\boldsymbol{B}_{1}\boldsymbol{K}_{2}\right\rangle  & =0
\end{align*}
or in index notation
\begin{equation}
B_{1ij}L_{1ij}-p_{2ji}p_{1jk}^{\perp}B_{1kl}K_{2li}=0.\label{eq:OAE-orth-2}
\end{equation}
Now we differentiate (\ref{eq:OAE-orth-2}) with respect to $\boldsymbol{B}_{1}$,
which gives
\[
L_{1rs}-p_{2ji}p_{1jr}^{\perp}K_{2si}=0\;\implies\;\boldsymbol{L}_{1}=\boldsymbol{p}_{1}^{\perp T}\boldsymbol{p}_{2}\boldsymbol{K}_{2}^{T}.
\]
In conclusion the normal space is given by matrices of the form
\begin{equation}
\begin{array}{rl}
\boldsymbol{Y}_{1} & =\boldsymbol{p}_{1}\boldsymbol{K}_{1}+\boldsymbol{p}_{1}^{\perp}\boldsymbol{p}_{1}^{\perp T}\boldsymbol{p}_{2}\boldsymbol{K}_{2}^{T},\\
\boldsymbol{Y}_{2} & =\boldsymbol{p}_{1}\boldsymbol{K}_{2},
\end{array}\label{eq:OAE-orth-repr}
\end{equation}
where $\boldsymbol{K}_{1}$ symmetric and $\boldsymbol{K}_{2}$ is
a general matrix.

\paragraph{Projection to tangent space.}

A projection is an operation that removes a vector from the normal
space from any input such that the result becomes a tangent vector.
Therefore we need to solve equation
\[
D\boldsymbol{h}\left(\boldsymbol{p}_{1},\boldsymbol{p}_{2}\right)\begin{pmatrix}\boldsymbol{X}_{1}-\boldsymbol{Y}_{1}\\
\boldsymbol{X}_{2}-\boldsymbol{Y}_{2}
\end{pmatrix}=\begin{pmatrix}\boldsymbol{0}\\
\boldsymbol{0}
\end{pmatrix},
\]
where $\left(\boldsymbol{Y}_{1},\boldsymbol{Y}_{2}\right)$ is in
the normal space $N_{\left(\boldsymbol{p}_{1},\boldsymbol{p}_{2}\right)}\mathcal{M}$.
Using the representation (\ref{eq:OAE-orth-repr}) of the normal space,
we find that the equation to solve is

\[
D\boldsymbol{h}\left(\boldsymbol{p}_{1},\boldsymbol{p}_{2}\right)\begin{pmatrix}\boldsymbol{X}_{1}-\boldsymbol{p}_{1}\boldsymbol{K}_{1}-\boldsymbol{p}_{1}^{\perp}\boldsymbol{p}_{1}^{\perp T}\boldsymbol{p}_{2}\boldsymbol{K}_{2}^{T}\\
\boldsymbol{X}_{2}-\boldsymbol{p}_{1}\boldsymbol{K}_{2}
\end{pmatrix}=\begin{pmatrix}\boldsymbol{0}\\
\boldsymbol{0}
\end{pmatrix}.
\]
It remains to evaluate $D\boldsymbol{h}$, which yields
\begin{multline}
D\boldsymbol{h}\left(\boldsymbol{p}_{1},\boldsymbol{p}_{2}\right)\begin{pmatrix}\boldsymbol{X}_{1}-\boldsymbol{p}_{1}\boldsymbol{K}_{1}-\boldsymbol{p}_{1}^{\perp}\boldsymbol{p}_{1}^{\perp T}\boldsymbol{p}_{2}\boldsymbol{K}_{2}^{T}\\
\boldsymbol{X}_{2}-\boldsymbol{p}_{1}\boldsymbol{K}_{2}
\end{pmatrix}\\
=\begin{pmatrix}\boldsymbol{X}_{1}^{T}\boldsymbol{p}_{1}+\boldsymbol{p}_{1}^{T}\boldsymbol{X}_{1}-\boldsymbol{K}_{1}^{T}-\boldsymbol{K}_{1}\\
\boldsymbol{X}_{1}^{T}\boldsymbol{p}_{2}+\boldsymbol{p}_{1}^{T}\boldsymbol{X}_{2}-\boldsymbol{K}_{2}\left(\boldsymbol{p}_{2}^{T}\boldsymbol{p}_{1}^{\perp}\boldsymbol{p}_{1}^{\perp T}\boldsymbol{p}_{2}+\boldsymbol{I}\right)
\end{pmatrix}=\begin{pmatrix}\boldsymbol{0}\\
\boldsymbol{0}
\end{pmatrix}.\label{eq:OAE-proj-1}
\end{multline}
The solution of equation (\ref{eq:OAE-proj-1}) is 
\begin{align*}
\boldsymbol{K}_{1} & =\frac{1}{2}\left(\boldsymbol{X}_{1}^{T}\boldsymbol{p}_{1}+\boldsymbol{p}_{1}^{T}\boldsymbol{X}_{1}\right)\\
\boldsymbol{K}_{2} & =\left(\boldsymbol{X}_{1}\boldsymbol{p}_{2}+\boldsymbol{p}_{1}^{T}\boldsymbol{X}_{2}\right)\left(\boldsymbol{I}+\boldsymbol{p}_{2}^{T}\boldsymbol{p}_{1}^{\perp}\boldsymbol{p}_{1}^{\perp T}\boldsymbol{p}_{2}\right)^{-1}
\end{align*}
and therefore the required projection is written as
\[
\boldsymbol{P}_{\left(\boldsymbol{p}_{1},\boldsymbol{p}_{2}\right)}\begin{pmatrix}\boldsymbol{X}_{1}\\
\boldsymbol{X}_{2}
\end{pmatrix}=\begin{pmatrix}\boldsymbol{X}_{1}-\frac{1}{2}\boldsymbol{p}_{1}\left(\boldsymbol{X}_{1}^{T}\boldsymbol{p}_{1}+\boldsymbol{p}_{1}^{T}\boldsymbol{X}_{1}\right)-\boldsymbol{p}_{1}^{\perp}\boldsymbol{p}_{1}^{\perp T}\boldsymbol{p}_{2}\left(\boldsymbol{I}+\boldsymbol{p}_{2}^{T}\boldsymbol{p}_{1}^{\perp}\boldsymbol{p}_{1}^{\perp T}\boldsymbol{p}_{2}\right)^{-1}\left(\boldsymbol{p}_{2}^{T}\boldsymbol{X}_{1}^{T}+\boldsymbol{X}_{2}^{T}\boldsymbol{p}_{1}\right)\\
\boldsymbol{X}_{2}-\boldsymbol{p}_{1}\left(\boldsymbol{X}_{1}\boldsymbol{p}_{2}+\boldsymbol{p}_{1}^{T}\boldsymbol{X}_{2}\right)\left(\boldsymbol{I}+\boldsymbol{p}_{2}^{T}\boldsymbol{p}_{1}^{\perp}\boldsymbol{p}_{1}^{\perp T}\boldsymbol{p}_{2}\right)^{-1}
\end{pmatrix}.
\]

\paragraph{Projective retraction.}

The retraction we calculate is the orthogonal projection onto the
manifold (\ref{eq:proj-retract}). This is a second order retraction
according to \cite{Absil2012Retraction}. In our case, the projection
is defined as
\[
\boldsymbol{P}_{\mathcal{M}}\left(\boldsymbol{q}_{1},\boldsymbol{q}_{2}\right)=\arg\min_{\left(\boldsymbol{p}_{1},\boldsymbol{p}_{2}\right)\in\mathcal{M}}\frac{1}{2}\left\Vert \boldsymbol{q}_{1}-\boldsymbol{p}_{1}\right\Vert ^{2}+\frac{1}{2}\left\Vert \boldsymbol{q}_{2}-\boldsymbol{p}_{2}\right\Vert ^{2},
\]
where we can assume that $m<l$ and $m<n$. Using constrained optimisation
we define the auxiliary objective function 
\[
g\left(\boldsymbol{p}_{1},\boldsymbol{p}_{1},\boldsymbol{\alpha},\boldsymbol{\beta}\right)=\frac{1}{2}\left\Vert \boldsymbol{q}_{1}-\boldsymbol{p}_{1}\right\Vert ^{2}+\left\langle \boldsymbol{\alpha},\boldsymbol{p}_{1}^{T}\boldsymbol{p}_{1}-\boldsymbol{I}\right\rangle +\frac{1}{2}\left\Vert \boldsymbol{q}_{2}-\boldsymbol{p}_{2}\right\Vert ^{2}+\left\langle \boldsymbol{\beta},\boldsymbol{p}_{1}^{T}\boldsymbol{p}_{2}-\boldsymbol{M}\right\rangle ,
\]
where $\boldsymbol{\alpha}\in\mathbb{R}^{m\times m}$, $\boldsymbol{\beta}\in\mathbb{R}^{m\times l}$
are Lagrange multipliers. We can also write the augmented cost function
in index notation, that is
\begin{align*}
g\left(\boldsymbol{p}_{1},\boldsymbol{p}_{1},\boldsymbol{\alpha},\boldsymbol{\beta}\right) & =\frac{1}{2}q_{1ij}q_{1ij}-q_{1ij}p_{1ij}+\frac{1}{2}p_{1ij}p_{1ij}+\alpha_{ji}p_{1kj}p_{1ki}-\alpha_{ii}\\
 & \quad+\frac{1}{2}q_{2ij}q_{2ij}-q_{2ij}p_{2ij}+\frac{1}{2}p_{2ij}p_{2ij}+\beta_{ji}p_{1kj}p_{2ki}-\beta_{ji}M_{ji}.
\end{align*}
To find the stationary point of $g$, we take the derivatives
\begin{align*}
D_{1}g\left(\boldsymbol{p}_{1},\boldsymbol{p}_{2},\boldsymbol{\alpha},\boldsymbol{\beta}\right) & =-q_{1rs}+p_{1rs}+\alpha_{js}p_{1rj}+\alpha_{si}p_{1ri}+\beta_{si}p_{2ri}\\
 & =-\boldsymbol{q}_{1}+\boldsymbol{p}_{1}+\boldsymbol{p}_{1}\boldsymbol{\alpha}+\boldsymbol{p}_{1}\boldsymbol{\alpha}^{T}+\boldsymbol{p}_{2}\boldsymbol{\beta}^{T},\\
D_{2}g\left(\boldsymbol{p}_{1},\boldsymbol{p}_{2},\boldsymbol{\alpha},\boldsymbol{\beta}\right) & =-q_{2rs}+p_{2rs}+\beta_{js}p_{1rj}=-\boldsymbol{q}_{2}+\boldsymbol{p}_{2}+\boldsymbol{p}_{1}\boldsymbol{\beta},
\end{align*}
which must vanish. Let us define $\boldsymbol{\sigma}=\boldsymbol{I}+\boldsymbol{\alpha}+\boldsymbol{\alpha}^{T}$,
hence the equations to solve become
\begin{align*}
\boldsymbol{p}_{1}\boldsymbol{\sigma}+\boldsymbol{p}_{2}\boldsymbol{\beta}^{T} & =\boldsymbol{q}_{1},\\
\boldsymbol{p}_{2}+\boldsymbol{p}_{1}\boldsymbol{\beta} & =\boldsymbol{q}_{2},\\
\boldsymbol{p}_{1}^{T}\boldsymbol{p}_{1} & =\boldsymbol{I},\\
\boldsymbol{p}_{1}^{T}\boldsymbol{p}_{2} & =\boldsymbol{M}.
\end{align*}
We can eliminate unknown variables $\boldsymbol{\sigma}$, $\boldsymbol{\beta}$
and $\boldsymbol{p}_{2}$ by solving the equations, that is
\begin{align*}
\boldsymbol{\beta} & =\boldsymbol{p}_{1}^{T}\boldsymbol{q}_{2}-\boldsymbol{M},\\
\boldsymbol{\sigma} & =\boldsymbol{p}_{1}^{T}\boldsymbol{q}_{1}-\boldsymbol{M}\boldsymbol{q}_{2}^{T}\boldsymbol{p}_{1}+\boldsymbol{M}\boldsymbol{M}^{T},\\
\boldsymbol{p}_{2} & =\boldsymbol{q}_{2}-\boldsymbol{p}_{1}\left(\boldsymbol{p}_{1}^{T}\boldsymbol{q}_{2}-\boldsymbol{M}\right).
\end{align*}
The equation for the remaining $\boldsymbol{p}_{1}$ then becomes
\begin{equation}
\left(\boldsymbol{p}_{1}\boldsymbol{p}_{1}^{T}-\boldsymbol{I}\right)\left(\boldsymbol{q}_{1}-\boldsymbol{q}_{2}\boldsymbol{q}_{2}^{T}\boldsymbol{p}_{1}+\boldsymbol{q}_{2}\boldsymbol{M}^{T}\right)=\boldsymbol{0}.\label{eq:OAE-retr-null}
\end{equation}
Equation (\ref{eq:OAE-retr-null}) means that $\boldsymbol{q}_{1}+\boldsymbol{q}_{2}\boldsymbol{M}^{T}-\boldsymbol{q}_{2}\boldsymbol{q}_{2}^{T}\boldsymbol{p}_{1}\in\mathrm{range}\,\boldsymbol{p}_{1}$,
therefore there exists matrix $\boldsymbol{a}$ such that 
\begin{equation}
\begin{array}{rl}
\boldsymbol{p}_{1}\boldsymbol{a} & =\boldsymbol{q}_{1}+\boldsymbol{q}_{2}\boldsymbol{M}^{T}-\boldsymbol{q}_{2}\boldsymbol{q}_{2}^{T}\boldsymbol{p}_{1}\\
\boldsymbol{p}_{1}^{T}\boldsymbol{p}_{1} & =\boldsymbol{I}
\end{array}.\label{eq:OAE-retr-eig}
\end{equation}
The solution of equation (\ref{eq:OAE-retr-eig}) is not unique, because
for any unitary transformation $\boldsymbol{T}$, $\boldsymbol{p}_{1}\mapsto\boldsymbol{p}_{1}\boldsymbol{T}$
and $\boldsymbol{a}\mapsto\boldsymbol{T}^{T}\boldsymbol{a}$ are also
a solution. We use Newton's method in combination with the Moore--Penrose
inverse of the Jacobian to find a solution of (\ref{eq:OAE-retr-eig}).
Finally, the retraction is set to the solution of (\ref{eq:OAE-retr-eig})
\[
\boldsymbol{R}_{\left(\tilde{\boldsymbol{p}}_{1},\tilde{\boldsymbol{p}}_{2}\right)}\begin{pmatrix}\boldsymbol{X}_{1}\\
\boldsymbol{X}_{2}
\end{pmatrix}=\left(\boldsymbol{p}_{1},\boldsymbol{p}_{2}\right),
\]
where $\boldsymbol{q}_{1}=\tilde{\boldsymbol{p}}_{1}+\boldsymbol{X}_{1}$,
$\boldsymbol{q}_{2}=\tilde{\boldsymbol{p}}_{2}+\boldsymbol{X}_{2}$.

\section{\label{sec:DelayEmbedding}State-space reconstruction}

When the full state of the system cannot be measured, but it is still
observable from the output, we need to employ state-space reconstruction.
We focus on the case of a real valued scalar signal $z_{j}\in\mathbb{R}$,
$j=1\ldots,M$, sampled with frequency $f_{s}$, where $j$ stands
for time. Taken's embedding theorem \cite{TakensEmbedding1981} states
that if we know the box counting dimension of our attractor, which
is denoted by $n$, the full state can almost always be observed if
we create a new vector \textbf{$\boldsymbol{x}_{p}=\left(z_{p},z_{p+\tau_{1}},\ldots,z_{p+\tau_{d-1}}\right)\in\mathbb{R}^{d}$},
where $d>2n$ and $\tau_{1},\ldots,\tau_{d-1}$ are integer delays.
The required number of delays is a conservative estimate, in many
cases $n\le d<2n+1$ is sufficient. The general problem, how to select
the number of delays $d$ and the delays $\tau_{1},\ldots,\tau_{d-1}$
optimally \cite{Pecora2007,Kramer1991autoencoder}, is a much researched
subject.

Instead of selecting delays $\tau_{1},\ldots,\tau_{d-1}$, we use
linear combinations of all possible delay embeddings, which allows
us to consider the optimality of the embedding in a linear parameter
space. We create vectors of length $L$,  
\[
\boldsymbol{u}_{p}=\left(z_{ps+1},\cdots,z_{ps+L}\right)\in\mathbb{R}^{L},\;p=0,1,\ldots,\left\lfloor \frac{M-L}{s}\right\rfloor ,
\]
where $s$ is the number of samples by which the window is shifted
forward in time with increasing $p$. Our reconstructed state variable
$\boldsymbol{x}_{p}\in\mathbb{R}^{d}\eqsim X$ is then calculated
by a linear map $\boldsymbol{T}:\mathbb{R}^{L}\to\mathbb{R}^{d}$,
in the form $\boldsymbol{x}_{p}=\boldsymbol{T}\boldsymbol{u}_{p}$
such that our scalar signal is returned by $z_{ps+\ell}\approx\boldsymbol{v}\cdot\boldsymbol{x}_{p}$
for some $\ell\in\left\{ 1,\ldots,L\right\} $. We assume that the
signal has $m$ dominant frequencies $f_{k}$, $k=1,\ldots,m$, with
$f_{1}$ being the lowest frequency. We define $w=\left\lfloor \frac{f_{s}}{f_{1}}\right\rfloor $
as an approximate period of the signal, where $f_{s}$ is the sampling
frequency and set $L=2w+1$. This makes each $\boldsymbol{u}_{p}$
capture roughly two periods of oscillations of the lowest frequency
$f_{1}$.

In what follows we use two approaches to construct the linear map
$\boldsymbol{T}$.

\subsection{\label{subsec:PCA}Principal component analysis}

In \cite{BROOMHEAD1986delayEmbed} it is suggested to use Principal
Component Analysis to find an optimal delay embedding. Let us construct
the matrix
\[
\boldsymbol{V}=\begin{pmatrix}\boldsymbol{u}_{1} & \boldsymbol{u}_{2} & \cdots & \boldsymbol{u}_{\left\lfloor \frac{M-L}{s}\right\rfloor }\end{pmatrix},
\]
where $\boldsymbol{u}_{p}$ are the columns of $\boldsymbol{V}$.
Then calculate the singular value decomposition (or equivalently Jordan
decomposition) of the symmetric matrix in the form
\[
\boldsymbol{V}\boldsymbol{V}^{T}=\boldsymbol{H}\boldsymbol{\Sigma}\boldsymbol{H}^{T},
\]
where $\boldsymbol{H}$ is a unitary matrix and $\boldsymbol{\Sigma}$
is diagonal. Now let us take the columns of $\boldsymbol{H}$ corresponding
to the $d$ largest elements of $\boldsymbol{\Sigma}$, and these
columns then become the $d$ rows of $\boldsymbol{T}$. In this case,
we can only achieve an approximate reconstruction of the signal by
defining
\[
\boldsymbol{v}=\begin{pmatrix}T_{1\ell} & T_{2\ell} & \cdots & T_{d\ell}\end{pmatrix},
\]
where $T_{k\ell}$ is the element of matrix $\boldsymbol{T}$ in the
$k$-th row and $\ell$-th column. The dot product $z_{ps+\ell}\approx\boldsymbol{v}\cdot\boldsymbol{x}_{p}$
approximately reproduces our signal. The quality of the approximation
depends on how small the discarded singular values in $\boldsymbol{\Sigma}$
are.

\subsection{\label{subsec:DFT}Perfect reproducing filter bank}

We can also use discrete Fourier transform (DFT) to create state-space
vectors and reconstruct the original signal exactly. Such transformations
are called perfect reproducing filter banks \cite{strang1996wavelets}.
The DFT of vector $\boldsymbol{u}_{p}$ is calculated by $\hat{\boldsymbol{u}}_{p}=\boldsymbol{K}\boldsymbol{u}_{p}$,
where matrix $\boldsymbol{K}$ is defined by its elements
\begin{align*}
K_{1,q} & =\frac{\sqrt{2}}{2},\\
K_{2p,q} & =\cos\left(2\pi p\frac{q-1}{2w+1}\right),\\
K_{2p+1,q} & =\sin\left(2\pi p\frac{q-1}{2w+1}\right),
\end{align*}
where $p=1,\ldots,w$ and $q=1,\ldots,2w+1$. Note that $\boldsymbol{K}^{T}\boldsymbol{K}=\boldsymbol{I}\frac{2w+1}{2}$,
hence the inverse transform is $\boldsymbol{K}^{-1}=\frac{2}{2w+1}\boldsymbol{K}^{T}$.
Let us denote the $\ell$-th column of $\boldsymbol{K}$ by $\boldsymbol{c}_{\ell}$,
which creates a delay filter 
\[
\frac{2}{2w+1}\boldsymbol{c}_{\ell}^{T}\boldsymbol{K}\boldsymbol{u}_{p}=z_{ps+\ell},
\]
that delays the signal by $\ell$ samples. Separating $\boldsymbol{c}_{\ell}$
into components, such that 
\begin{equation}
\boldsymbol{c}_{\ell}=\boldsymbol{v}_{1}+\boldsymbol{v}_{2}+\cdots+\boldsymbol{v}_{d},\label{eq:DFT-decomp}
\end{equation}
we can create a perfect reproducing filter bank, that is the vector
\[
\hat{\boldsymbol{x}}_{p}=\frac{2}{2w+1}\begin{pmatrix}\boldsymbol{v}_{1}^{T}\boldsymbol{K}\boldsymbol{u}_{p}\\
\boldsymbol{v}_{2}^{T}\boldsymbol{K}\boldsymbol{u}_{p}\\
\vdots\\
\boldsymbol{v}_{2m}^{T}\boldsymbol{K}\boldsymbol{u}_{p}
\end{pmatrix}
\]
when summed over its components, the input $z_{ps+\ell}$ is recovered.
The question is how to create an optimal decomposition (\ref{eq:DFT-decomp}).
As we assumed that our signal has $m$ frequencies, we can further
assume that we are dealing with $m$ coupled nonlinear oscillators,
and therefore we can set $d=2m$. We therefore divide the resolved
frequencies of the DFT, which are $\tilde{f}_{k}=k\frac{f_{s}}{2w+1}$,
$k=0,\ldots,w$ into $m$ bins, which are centred around the frequencies
of the signal $f_{1},\ldots,f_{m}$. We set the boundaries of these
bins as
\[
b_{0}=0,b_{1}=\frac{f_{1}+f_{2}}{2},b_{2}=\frac{f_{2}+f_{3}}{2},\cdots,b_{m-1}=\frac{f_{m-1}+f_{m}}{2},b_{m}=\frac{wf_{s}}{2w+1}
\]
 and for each bin labelled by $k$, we create an index set
\[
\mathcal{I}_{k}=\left\{ l\in1,\ldots,w:b_{k-1}<\tilde{f}_{l}\le b_{k}\right\} ,
\]
which ensures that all DFT frequencies are taken into account without
overlap, that is $\cup_{k}\mathcal{I}_{k}=\left\{ 0,1,\ldots,w\right\} $
and $\mathcal{I}_{p}\cap\mathcal{I}_{q}=\emptyset$ for $p\neq q$.
This creates a decomposition (\ref{eq:DFT-decomp}) in the form of
\[
v_{2k-1,2j}=\begin{cases}
c_{\ell,2j} & \text{if}\quad j\in\mathcal{I}_{k}\\
0 & \text{if}\quad j\notin\mathcal{I}_{k}
\end{cases},\;v_{2k,2j+1}=\begin{cases}
c_{\ell,2j+1} & \text{if}\quad j\in\mathcal{I}_{k}\\
0 & \text{if}\quad j\notin\mathcal{I}_{k}
\end{cases}
\]
with the exception of $\boldsymbol{v}_{1}$, for which we also set
$v_{1,1}=c_{\ell,1}$ to take into account the moving average of the
signal and to make sure that $\sum\boldsymbol{v}_{k}=\boldsymbol{c}_{\ell}$.
Here $v_{k,l}$ stands for the $l$-th element of vector $\boldsymbol{v}_{k}$.
Finally, we define the transformation matrix
\[
\boldsymbol{H}=\frac{2}{2w+1}\begin{pmatrix}\left\Vert \boldsymbol{v}_{1}\right\Vert ^{-1}\boldsymbol{v}_{1}^{T}\\
\vdots\\
\left\Vert \boldsymbol{v}_{2m}\right\Vert ^{-1}\boldsymbol{v}_{2m}^{T}
\end{pmatrix},
\]
where each row is normalised such that our separated signals have
the same amplitude. The newly created signal is then
\[
\boldsymbol{x}_{p}=\boldsymbol{T}\boldsymbol{u}_{p},
\]
where $\boldsymbol{T}=\boldsymbol{H}\boldsymbol{K}$ and the original
signal can be reproduced as
\[
z_{ps+\ell}=\begin{pmatrix}\left\Vert \boldsymbol{v}_{1}\right\Vert  & \cdots & \left\Vert \boldsymbol{v}_{2m}\right\Vert \end{pmatrix}\boldsymbol{x}_{p}.
\]

\section{\label{sec:proof-FreqDamp}Proof of proposition \ref{prop:Polar-fr-dm}}
\begin{proof}[Proof of proposition \ref{prop:Polar-fr-dm}]
First we explore what happens if we introduce a new parametrisation
of the decoder $\boldsymbol{W}$ that replaces $r$ by $\rho\left(r\right)$
and $\theta$ by $\theta+\gamma\left(\rho\left(r\right)\right)$,
where $\rho:\mathbb{R}^{+}\to\mathbb{R}^{+}$ is an invertible function
with $\rho\left(0\right)=0$ and $\gamma:\mathbb{R}^{+}\to\mathbb{R}$
with $\gamma\left(0\right)=0$. This transformation creates a new
decoder $\tilde{\boldsymbol{W}}$ of $\mathcal{M}$ in the form of
\begin{align}
\tilde{\boldsymbol{W}}\left(r,\theta\right) & =\boldsymbol{W}\left(\rho\left(r\right),\theta+\gamma\left(\rho\left(r\right)\right)\right),\label{eq:Polar-re-param}\\
\tilde{\boldsymbol{W}}\left(\rho^{-1}\left(r\right),\theta-\gamma\left(r\right)\right) & =\boldsymbol{W}\left(r,\theta\right).\label{eq:Polar-inv-re-param}
\end{align}
Equation (\ref{eq:Polar-re-param}) is not the only re-parametrisation,
but this is the only one that preserves the structure of the polar
invariance equation (\ref{eq:Polar-Invariance}). After substituting
the transformed decoder, invariance equation (\ref{eq:Polar-Invariance})
becomes
\[
\tilde{\boldsymbol{W}}\left(\tilde{R}\left(r\right),\theta+\tilde{T}\left(r\right)\right)=\boldsymbol{F}\left(\tilde{\boldsymbol{W}}\left(r,\theta\right)\right),
\]
where
\begin{align*}
\tilde{R}\left(r\right) & =\rho^{-1}\left(R\left(\rho\left(r\right)\right)\right),\\
\tilde{T}\left(r\right) & =T\left(\rho\left(r\right)\right)+\gamma\left(\rho\left(r\right)\right)-\gamma\left(R\left(\rho\left(r\right)\right)\right).
\end{align*}
In the new coordinates, the instantaneous frequency $\omega\left(r\right)$
and damping ratio $\zeta\left(r\right)$ become
\begin{align}
\omega\left(r\right) & =T\left(\rho\left(r\right)\right)+\gamma\left(\rho\left(r\right)\right)-\gamma\left(R\left(\rho\left(r\right)\right)\right) & \left[\text{rad}/\text{step}\right],\label{eq:Polar-freq-newpar}\\
\zeta\left(r\right) & =-\frac{\log\left[r^{-1}\rho^{-1}\left(R\left(\rho\left(r\right)\right)\right)\right]}{\omega\left(r\right)} & \left[-\right].\label{eq:Polar-damp-newpar}
\end{align}

Before we go further, let us introduce the notation
\[
\left\langle \boldsymbol{x},\boldsymbol{y}\right\rangle =\frac{1}{2\pi}\int_{0}^{2\pi}\left\langle \boldsymbol{x}\left(\theta\right),\boldsymbol{x}\left(\theta\right)\right\rangle _{X}\mathrm{d}\theta,
\]
where $\boldsymbol{x},\boldsymbol{y}:\left[0,2\pi\right]\to X$ and
$\left\langle \cdot,\cdot\right\rangle _{X}$ is the inner product
on vector space $X$. The norm of function $\boldsymbol{x}$ is then
defined by $\left\Vert \boldsymbol{x}\right\Vert =\sqrt{\left\langle \boldsymbol{x},\boldsymbol{x}\right\rangle }$.

We now turn to the phase constraint given by equation (\ref{eq:Polar-phase}).
We choose a fixed $\epsilon>0$, substitute $r_{1}=\rho\left(r+\epsilon\right)$,
$r_{2}=\rho\left(r\right)$ and the phase shift $\gamma=\gamma\left(\rho\left(r+\epsilon\right)\right)-\gamma\left(\rho\left(r\right)\right)$
into equation (\ref{eq:Polar-phase}) and find the objective function
\[
J\left(\gamma\left(\rho\left(r+\epsilon\right)\right)\right)=\epsilon^{-1}\left\Vert \boldsymbol{W}\left(\rho\left(r+\epsilon\right),\cdot+\gamma\left(\rho\left(r+\epsilon\right)\right)\right)-\boldsymbol{W}\left(\rho\left(r\right),\cdot+\gamma\left(\rho\left(r\right)\right)\right)\right\Vert ^{2},
\]
where we also divided by $\epsilon$. We then need to find the value
of $\gamma\left(\rho\left(r+\epsilon\right)\right)$ that minimises
$J$ for a fixed $\epsilon$. A necessary condition for a local minimum
of $J$ is that the derivative $J^{\prime}=0$, that is,
\[
\epsilon^{-1}\left\langle \boldsymbol{W}\left(\rho\left(r+\epsilon\right),\cdot+\gamma\left(\rho\left(r+\epsilon\right)\right)\right)-\boldsymbol{W}\left(\rho\left(r\right),\cdot+\gamma\left(\rho\left(r\right)\right)\right),D_{2}\boldsymbol{W}\left(\rho\left(r+\epsilon\right),\cdot+\gamma\left(\rho\left(r+\epsilon\right)\right)\right)\right\rangle =0.
\]
To find a continuous parametrisation, let us now take the limit $\epsilon\to0$
to get
\[
\dot{\rho}\left(r\right)\left\langle D_{1}\boldsymbol{W}\left(\rho\left(r\right),\cdot+\gamma\left(\rho\left(r\right)\right)\right)+D_{2}\boldsymbol{W}\left(\rho\left(r\right),\cdot+\gamma\left(\rho\left(r\right)\right)\right)\dot{\gamma}\left(\rho\left(r\right)\right),D_{2}\boldsymbol{W}\left(\rho\left(r\right),\cdot+\gamma\left(\rho\left(r\right)\right)\right)\right\rangle =0.
\]
We can also remove the constant phase shift and use $t=\rho\left(r\right)$
to find the \emph{phase condition}
\begin{equation}
\left\langle D_{1}\boldsymbol{W}\left(t,\cdot\right)+D_{2}\boldsymbol{W}\left(t,\cdot\right)\dot{\gamma}\left(t\right),D_{2}\boldsymbol{W}\left(t,\cdot\right)\right\rangle =0.\label{eq:Polar-phase-cond}
\end{equation}
Integrating equation (\ref{eq:Polar-phase-cond}) leads to the phase
shift 
\begin{equation}
\gamma\left(t\right)=-\int_{0}^{t}\frac{\left\langle D_{1}\boldsymbol{W}\left(r,\cdot\right),D_{2}\boldsymbol{W}\left(r,\cdot\right)\right\rangle }{\left\langle D_{2}\boldsymbol{W}\left(r,\cdot\right),D_{2}\boldsymbol{W}\left(r,\cdot\right)\right\rangle }\mathrm{d}r\label{eq:Polar-phase-shift}
\end{equation}
Note that phase conditions are commonly used in numerical continuation
of periodic orbits \cite{Beyn2007}, for slightly different reasons.

The next quantity to fix is the amplitude constraint given by equation
(\ref{eq:Polar-amplitude}). In the new parametrisation, equation
(\ref{eq:Polar-amplitude}) can be written as
\begin{equation}
\left\Vert \boldsymbol{W}\left(\rho\left(r\right),\cdot\right)\right\Vert ^{2}=r^{2}.\label{eq:Polar-amplitude-cond}
\end{equation}
If we define $\kappa\left(r\right)=\left\Vert \boldsymbol{W}\left(r,\cdot\right)\right\Vert $,
then we have
\begin{equation}
\rho\left(r\right)=\kappa^{-1}\left(r\right).\label{eq:Polar-amplitude-final}
\end{equation}
Finally, substituting (\ref{eq:Polar-phase-shift}), (\ref{eq:Polar-amplitude-final})
into equations (\ref{eq:Polar-freq-newpar}), (\ref{eq:Polar-damp-newpar})
completes the proof.
\end{proof}
\begin{proof}[Proof of remark \ref{rem:VF-freq-dm}]
In case of a vector field $\boldsymbol{f}:X\to X$, and invariance
equation 
\[
D_{1}\boldsymbol{W}\left(r,\theta\right)R\left(r\right)+D_{2}\boldsymbol{W}\left(r,\theta\right)T\left(r\right)=\boldsymbol{f}\left(\boldsymbol{W}\left(r,\theta\right)\right),
\]
the instantaneous frequency and damping ratios with respect to the
coordinate system defined by the decoder $\boldsymbol{W}$, are
\begin{equation}
\left.\begin{array}{rl}
\omega\left(r\right) & =T\left(r\right),\\
\zeta\left(r\right) & =-\frac{R\left(r\right)}{r\omega\left(r\right)},
\end{array}\right\} \label{eq:Polar-VF-naive}
\end{equation}
respectively. However the coordinate system defined by $\boldsymbol{W}$
is nonlinear, which needs correcting. As in the proof of proposition
\ref{prop:Polar-fr-dm}, we assume a coordinate transformation
\begin{equation}
\tilde{\boldsymbol{W}}\left(r,\theta\right)=\boldsymbol{W}\left(\rho\left(r\right),\theta+\gamma\left(\rho\left(r\right)\right)\right),\label{eq:Polar-VF-tansform}
\end{equation}
and find that invariance equation (\ref{eq:Polar-VF-invariance})
becomes
\begin{equation}
D_{1}\tilde{\boldsymbol{W}}\left(r,\theta\right)\tilde{R}\left(r\right)+D_{2}\tilde{\boldsymbol{W}}\left(r,\theta\right)\tilde{T}\left(r\right)=\boldsymbol{f}\left(\tilde{\boldsymbol{W}}\left(r,\theta\right)\right).\label{eq:Polar-VF-tr-invar}
\end{equation}
When substituting (\ref{eq:Polar-VF-tansform}) into (\ref{eq:Polar-VF-tr-invar})
we find
\begin{multline}
D_{1}\tilde{\boldsymbol{W}}\left(r,\theta\right)\dot{\kappa}\left(\rho\left(r\right)\right)R\left(\rho\left(r\right)\right)+D_{2}\tilde{\boldsymbol{W}}\left(r,\theta\right)\left(T\left(\rho\left(r\right)\right)-\dot{\gamma}\left(\rho\left(r\right)\right)R\left(\rho\left(r\right)\right)\right)=\\
=\boldsymbol{f}\left(\tilde{\boldsymbol{W}}\left(r,\theta\right)\right).\label{eq:Polar-VF-tr-invar-expand}
\end{multline}
Comparing (\ref{eq:Polar-VF-tr-invar-expand}) and (\ref{eq:Polar-VF-invariance})
we extract that
\begin{equation}
\left.\begin{array}{rl}
\tilde{R}\left(r\right) & =\dot{\kappa}\left(\rho\left(r\right)\right)R\left(\rho\left(r\right)\right)\\
\tilde{T}\left(r\right) & =T\left(\rho\left(r\right)\right)-\dot{\gamma}\left(\rho\left(r\right)\right)R\left(\rho\left(r\right)\right)
\end{array}\right\} .\label{eq:Polar-VF-ROM}
\end{equation}
Recognising that $t=\rho\left(r\right)=\kappa^{-1}\left(r\right)$
and replacing $T$ with $\tilde{T}$, $R$ with $\tilde{R}$ in formulae
(\ref{eq:Polar-VF-naive}) proves remark \ref{rem:VF-freq-dm}.
\end{proof}

\subsection*{Acknowledgement}

I would like to thank Branislaw Titurus for supplying the data of
the jointed beam. I would also like to thank Sanuja Jayatilake, who
helped me collect more data on the jointed beam, which eventually
was not needed in this study. Discussions with Alessandra Vizzaccaro,
David Barton and his research group provided great inspiration to
complete this research. A.V. has also commented on a draft of this
manuscript.

\bibliographystyle{plain}
\bibliography{../../Bibliography/AllRef}

\begin{thebibliography}{10}

\bibitem{absil2009optimization}
P.~A. Absil, R.~Mahony, and R.~Sepulchre.
\newblock {\em Optimization Algorithms on Matrix Manifolds}.
\newblock Princeton University Press, 2009.

\bibitem{Absil2012Retraction}
P.~A. Absil and J.~Malick.
\newblock Projection-like retractions on matrix manifolds.
\newblock {\em SIAM Journal on Optimization}, 22(1):135--158, 2012.

\bibitem{BartonCBC2017}
D.~A.~W. Barton.
\newblock Control-based continuation: Bifurcation and stability analysis for
  physical experiments.
\newblock {\em Mechanical Systems and Signal Processing}, 84:54--64, 2017.

\bibitem{bellman2015adaptive}
R.~E. Bellman.
\newblock {\em Adaptive control processes}.
\newblock Princeton university press, 2015.

\bibitem{Manopt2022}
R.~Bergmann.
\newblock Manopt.jl: Optimization on manifolds in {J}ulia.
\newblock {\em Journal of Open Source Software}, 7(70):3866, 2022.

\bibitem{Beyn2007}
W.-J. Beyn and V.~Th{\"u}mmler.
\newblock {\em Phase Conditions, Symmetries and PDE Continuation}, pages
  301--330.
\newblock Springer Netherlands, Dordrecht, 2007.

\bibitem{billings2013nonlinear}
S.A. Billings.
\newblock {\em Nonlinear System Identification: "NARMAX" Methods in the Time,
  Frequency, and Spatio-Temporal Domains}.
\newblock Wiley, 2013.

\bibitem{boumal2022intromanifolds}
N.~Boumal.
\newblock An introduction to optimization on smooth manifolds.
\newblock To appear with Cambridge University Press, Mar 2022.

\bibitem{boyd_vandenberghe_2018}
S.~Boyd and L.~Vandenberghe.
\newblock {\em Introduction to Applied Linear Algebra: Vectors, Matrices, and
  Least Squares}.
\newblock Cambridge University Press, 2018.

\bibitem{Breunung2017}
T.~Breunung and G.~Haller.
\newblock Explicit backbone curves from spectral submanifolds of forced-damped
  nonlinear mechanical systems.
\newblock {\em Proceedings of the Royal Society of London A: Mathematical,
  Physical and Engineering Sciences}, 474(2213), 2018.

\bibitem{BROOMHEAD1986delayEmbed}
D.~S. Broomhead and G.~P. King.
\newblock Extracting qualitative dynamics from experimental data.
\newblock {\em Physica D: Nonlinear Phenomena}, 20(2):217--236, 1986.

\bibitem{Brunton2014}
S.~L. Brunton, J.~H. Tu, I.~Bright, and J.~N. Kutz.
\newblock Compressive sensing and low-rank libraries for classification of
  bifurcation regimes in nonlinear dynamical systems.
\newblock {\em SIAM Journal on Applied Dynamical Systems}, 13(4):1716--1732,
  2014.

\bibitem{BruntonPNAS2016}
S.L. Brunton, J.L. Proctor, J.N. Kutz, and W.~Bialek.
\newblock Discovering governing equations from data by sparse identification of
  nonlinear dynamical systems.
\newblock {\em Proceedings of the National Academy of Sciences of the United
  States of America}, 113(15):3932--3937, 2016.

\bibitem{CabreLlave2003}
X.~Cabr{\'e}, E.~Fontich, and R.~{de la Llave}.
\newblock The parameterization method for invariant manifolds {I}: {M}anifolds
  associated to non-resonant subspaces.
\newblock {\em Indiana Univ. Math. J.}, 52:283--328, 2003.

\bibitem{CabreP3-2005}
X.~Cabr{\'e}, E.~Fontich, and R.~{de la Llave}.
\newblock The parameterization method for invariant manifolds iii: overview and
  applications.
\newblock {\em Journal of Differential Equations}, 218(2):444--515, 2005.

\bibitem{Casdagli1989DelayEmbed}
M.~Casdagli.
\newblock Nonlinear prediction of chaotic time series.
\newblock {\em Physica D: Nonlinear Phenomena}, 35(3):335--356, 1989.

\bibitem{Cenedese2022NatComm}
M.~Cenedese, J.~Axås, B.~B\"auerlein, K.~Avila, and G.~Haller.
\newblock Data-driven modeling and prediction of non-linearizable dynamics via
  spectral submanifolds.
\newblock {\em Nat Commun}, 13(872), 2022.

\bibitem{Champion2019Autoencoder}
K.~Champion, B.~Lusch, J.~Nathan~Kutz, and S.~L. Brunton.
\newblock Data-driven discovery of coordinates and governing equations.
\newblock {\em Proceedings of the National Academy of Sciences of the United
  States of America}, 116(45):22445--22451, 2019.

\bibitem{Coifman2006}
R.~R. Coifman and S.~Lafon.
\newblock Diffusion maps.
\newblock {\em Applied and Computational Harmonic Analysis}, 21(1):5--30, 2006.

\bibitem{conn2000trust}
A.~R. Conn, N.~I.~M. Gould, and P.~L. Toint.
\newblock {\em Trust Region Methods}.
\newblock MPS-SIAM Series on Optimization. SIAM, 2000.

\bibitem{delaLlave1997}
R.~de~la Llave.
\newblock Invariant manifolds associated to nonresonant spectral subspaces.
\newblock {\em Journal of Statistical Physics}, 87(1):211--249, 1997.

\bibitem{Donoho2006}
D.L. Donoho.
\newblock Compressed sensing.
\newblock {\em IEEE Transactions on Information Theory}, 52(4):1289--1306,
  2006.

\bibitem{EHRHARDT2016612}
D.~A. Ehrhardt and M.~S. Allen.
\newblock Measurement of nonlinear normal modes using multi-harmonic stepped
  force appropriation and free decay.
\newblock {\em Mechanical Systems and Signal Processing}, 76-77:612 -- 633,
  2016.

\bibitem{ElbrachterBolcskei2021}
D.~Elbrachter, D.~Perekrestenko, P.~Grohs, and H.~Bolcskei.
\newblock Deep neural network approximation theory.
\newblock {\em IEEE Transactions on Information Theory}, 67(5):2581--2623,
  2021.

\bibitem{Fenichel}
N.~Fenichel.
\newblock Persistence and smoothness of invariant manifolds for flows.
\newblock {\em Indiana Univ. Math. J.}, 21:193--226, 1972.

\bibitem{GrasedyckSVD}
L.~Grasedyck.
\newblock Hierarchical singular value decomposition of tensors.
\newblock {\em SIAM Journal on Matrix Analysis and Applications},
  31(4):2029--2054, 2010.

\bibitem{TensorApproxSurvey}
L.~Grasedyck, D.~Kressner, and C.~Tobler.
\newblock A literature survey of low-rank tensor approximation techniques.
\newblock {\em GAMM-Mitteilungen}, 36(1):53--78, 2013.

\bibitem{HackbuschKuhn2009}
W.~Hackbusch and S.~Kühn.
\newblock A new scheme for the tensor representation.
\newblock {\em Journal of Fourier Analysis and Applications}, 15:706--722,
  2009.

\bibitem{Haller2016}
G.~Haller and S.~Ponsioen.
\newblock Nonlinear normal modes and spectral submanifolds: existence,
  uniqueness and use in model reduction.
\newblock {\em Nonlinear Dynamics}, 86(3):1493--1534, 2016.

\bibitem{Hermann1977NonlinearContrObs}
R.~Hermann and A.~Krener.
\newblock Nonlinear controllability and observability.
\newblock {\em IEEE Transactions on Automatic Control}, 22(5):728--740, 1977.

\bibitem{HORNIK1991}
K.~Hornik.
\newblock Approximation capabilities of multilayer feedforward networks.
\newblock {\em Neural Networks}, 4(2):251 -- 257, 1991.

\bibitem{JinBrake2020FreqDamp}
M.~Jin, W.~Chen, M.~R.~W. Brake, and H.~Song.
\newblock Identification of instantaneous frequency and damping from transient
  decay data.
\newblock {\em Journal of Vibration and Acoustics, Transactions of the ASME},
  142(5):051111, 2020.

\bibitem{KaliaMeijerBrunton2021}
M.~Kalia, S.~L. Brunton, H.~G.~E. Meijer, C.~Brune, and J.~N. Kutz.
\newblock Learning normal form autoencoders for data-driven discovery of
  universal,parameter-dependent governing equations, 2021.

\bibitem{Kevrekidis2003}
I.~G. Kevrekidis, C.~W. Gear, J.~M. Hyman, P.~G. Kevrekidis, O.~Runborg, and
  C.~Theodoropoulos.
\newblock Equation-free, coarse-grained multiscale computation: enabling
  microscopic simulators to perform system-level analysis.
\newblock {\em Communications in Mathematical Sciences}, 1(4):715 – 762,
  2003.

\bibitem{Samey}
Ioannis~G. Kevrekidis and Giovanni Samaey.
\newblock Equation-free multiscale computation: Algorithms and applications.
\newblock {\em Annual Review of Physical Chemistry}, 60(1):321--344, 2009.

\bibitem{Kramer1991autoencoder}
M.~A. Kramer.
\newblock Nonlinear principal component analysis using autoassociative neural
  networks.
\newblock {\em AIChE Journal}, 37(2):233--243, 1991.

\bibitem{lang2012fundamentals}
S.~Lang.
\newblock {\em Fundamentals of Differential Geometry}.
\newblock Graduate Texts in Mathematics. Springer New York, 2012.

\bibitem{Lawson1974}
H.~Blaine Lawson, Jr.
\newblock Foliations.
\newblock {\em Bull. Amer. Math. Soc.}, 80:369--418, 1974.

\bibitem{Mezic2005}
I.~Mezi\'c.
\newblock Spectral properties of dynamical systems, model reduction and
  decompositions.
\newblock {\em Nonlinear Dynamics}, 41(1-3):309--325, 2005.

\bibitem{Mezic2021}
I.~Mezi\'c.
\newblock Koopman operator, geometry, and learning of dynamical systems.
\newblock {\em Notices of the American Mathematical Society}, 68(7):1087--1105,
  2021.

\bibitem{GaussSouthwell2015}
J.~Nutini, M.~Schmidt, I.~H. Laradji, M.~Friedlander, and H.~Koepke.
\newblock Coordinate descent converges faster with the gauss-southwell rule
  than random selection.
\newblock In {\em Proceedings of the 32nd International Conference on
  International Conference on Machine Learning - Volume 37}, ICML'15, page
  1632–1641, 2015.

\bibitem{orr2003neural}
G.~B. Orr and K.~R. M{\"u}ller.
\newblock {\em Neural Networks: Tricks of the Trade}.
\newblock Lecture Notes in Computer Science. Springer Berlin Heidelberg, 2003.

\bibitem{ortega1970iterative}
J.~M. Ortega, J.~R.~W. Ortega, and W.~C. Rheinboldt.
\newblock {\em Iterative Solution of Nonlinear Equations in Several Variables}.
\newblock Classics in Applied Mathematics. SIAM, 1970.

\bibitem{Oseledets2011}
I.~V. Oseledets.
\newblock Tensor-train decomposition.
\newblock {\em SIAM J. Sci. Comput.}, 33(5):2295--2317, 2011.

\bibitem{Pecora2007}
L.~M. Pecora, L.~Moniz, J.~Nichols, and T.~L. Carroll.
\newblock A unified approach to attractor reconstruction.
\newblock {\em Chaos: An Interdisciplinary Journal of Nonlinear Science},
  17(1):013110, 2007.

\bibitem{MatrixProductStates2007}
D.~Perez-Garcia, F.~Verstraete, M.~M. Wolf, and J.~I. Cirac.
\newblock Matrix product state representations.
\newblock {\em Quantum Info. Comput.}, 7(5):401--430, 2007.

\bibitem{PetersenTopologyNeuralNetworks2021}
P.~Petersen, M.~Raslan, and F.~Voigtlaender.
\newblock Topological properties of the set of functions generated by neural
  networks of fixed size.
\newblock {\em Foundations of Computational Mathematics}, 21(2):375--444, 2021.

\bibitem{PONSIOEN2018269}
S.~Ponsioen, T.~Pedergnana, and G.~Haller.
\newblock Automated computation of autonomous spectral submanifolds for
  nonlinear modal analysis.
\newblock {\em Journal of Sound and Vibration}, 420:269 -- 295, 2018.

\bibitem{Read1998NormalForm}
N.~K. Read and W.~H. Ray.
\newblock Application of nonlinear dynamic analysis to the identification and
  control of nonlinear systems - iii. n-dimensional systems.
\newblock {\em Journal of Process Control}, 8(1):35--46, 1998.

\bibitem{Roberts89}
A.~J. Roberts.
\newblock Appropriate initial conditions for asymptotic descriptions of the
  long-term evolution of dynamical systems.
\newblock {\em J. Austral. Math. Soc. Ser. B}, 31:48--75, 1989.

\bibitem{Roberts91}
A.~J. Roberts.
\newblock Boundary conditions for approximate differential equations.
\newblock {\em J. Austral. Math. Soc. Ser. B}, 34:54--80, 1992.

\bibitem{ShawPierre}
S.~W. Shaw and C~Pierre.
\newblock Normal-modes of vibration for nonlinear continuous systems.
\newblock {\em {J. Sound Vibr.}}, {169}({3}):319--347, {1994}.

\bibitem{SieberCBC2008}
J.~Sieber and B.~Krauskopf.
\newblock Control based bifurcation analysis for experiments.
\newblock {\em Nonlinear Dynamics}, 51:365--377, 2008.

\bibitem{strang1996wavelets}
G.~Strang and T.~Nguyen.
\newblock {\em Wavelets and Filter Banks}.
\newblock Wellesley-Cambridge Press, 1996.

\bibitem{Szalai2020ISF}
R.~Szalai.
\newblock Invariant spectral foliations with applications to model order
  reduction and synthesis.
\newblock {\em Nonlinear Dynamics}, 101(4):2645--2669, 2020.

\bibitem{Szalai20160759}
R.~Szalai, D.~Ehrhardt, and G.~Haller.
\newblock Nonlinear model identification and spectral submanifolds for
  multi-degree-of-freedom mechanical vibrations.
\newblock {\em Proc. R. Soc. A}, 473:20160759, 2017.

\bibitem{TakensEmbedding1981}
F.~Takens.
\newblock Detecting strange attractors in turbulence.
\newblock In D.~Rand and Lai-Sang Young, editors, {\em Dynamical Systems and
  Turbulence, Warwick 1980}, pages 366--381. Springer Berlin Heidelberg, 1981.

\bibitem{Tenenbaum2000isomap}
J.~B. Tenenbaum, V.~De~Silva, and J.~C. Langford.
\newblock A global geometric framework for nonlinear dimensionality reduction.
\newblock {\em Science}, 290(5500):2319--2323, 2000.

\bibitem{Titurus2016}
B.~Titurus, J.~Yuan, F.~Scarpa, S.~Patsias, and S.~Pattison.
\newblock Impact hammer-based analysis of nonlinear effects in bolted lap
  joint.
\newblock In {\em Proceedings of ISMA2016 including USD2016}, ISMA2016, page
  789–802, 2016.

\bibitem{USCHMAJEW2013133}
A.~Uschmajew and B.~Vandereycken.
\newblock The geometry of algorithms using hierarchical tensors.
\newblock {\em Linear Algebra and its Applications}, 439(1):133--166, 2013.

\bibitem{VIZZACCARO2021normalForm}
A.~Vizzaccaro, Y.~Shen, L.~Salles, J.~Blahoš, and C.~Touzé.
\newblock Direct computation of nonlinear mapping via normal form for
  reduced-order models of finite element nonlinear structures.
\newblock {\em Computer Methods in Applied Mechanics and Engineering},
  384:113957, 2021.

\bibitem{WhitneyEmbedding1936}
H.~Whitney.
\newblock Differentiable manifolds.
\newblock {\em Annals of Mathematics}, 37(3):645--680, 1936.

\bibitem{FreeLunch1997}
D.~H. Wolpert and W.~G. Macready.
\newblock No free lunch theorems for optimization.
\newblock {\em IEEE Transactions on Evolutionary Computation}, 1(1):67--82,
  1997.

\bibitem{Yair2017DiffusionNormal}
O.~Yair, R.~Talmon, R.~R. Coifman, and I.~G. Kevrekidis.
\newblock Reconstruction of normal forms by learning informed observation
  geometries from data.
\newblock {\em Proceedings of the National Academy of Sciences of the United
  States of America}, 114(38):E7865--E7874, 2017.

\end{thebibliography}

\subsection*{Conflict of Interest}

The author has no relevant financial or non-financial interests to
disclose. No funds, grants, or other support was received.

\subsection*{Data availability}

The datasets generated and/or analysed during the current study are
available in the FMA repository, \href{https://github.com/rs1909/FMA}{https://github.com/rs1909/FMA}.
\end{document}